\documentclass[%
]{mpi2015-cscpreprint}

\usepackage[caption=false]{subfig}% Support for small, `sub' figures and tables
\usepackage{amsmath,graphicx}

\newcommand{\diagdots}[3][-25]{%
	\rotatebox{#1}{\makebox[0pt]{\makebox[#2]{\xleaders\hbox{$\cdot$\hskip#3}\hfill\kern0pt}}}%
}
\usepackage{amsfonts}
\usepackage{amssymb}
\usepackage[english]{babel}
\usepackage{pgfplots}
\usepackage{nomencl}
\makenomenclature
\usetikzlibrary{matrix}
\usepackage{parskip}
\usepackage{amsthm}
\usepackage{enumerate}
\usepackage{cite}
\usepackage{esvect} % Vektorpfeil
\usepackage{soul} % allows strikethrough text \st
\usepackage{float}
\usepackage{mathtools}
\usepackage{todonotes}
\usepackage{booktabs}
\usepackage{ifdraft}
\usepackage{algorithm}
\usepackage{algpseudocode}
\usepackage{setspace}
\usepackage{color}

\renewcommand{\hat}{\widehat}
\renewcommand{\tilde}{\widetilde}

\newcommand\ri{\ensuremath{\mathrm{i}}}

\newcommand\bcL{\ensuremath{\boldsymbol{\mathcal{L}}}}
\newcommand\bA{\ensuremath{\mathbf{A}}}
\newcommand\bE{\ensuremath{\mathbf{E}}}
\newcommand\bB{\ensuremath{\mathbf{B}}}
\newcommand\bC{\ensuremath{\mathbf{C}}}

\newcommand\bM{\ensuremath{\mathbf{M}}}
\newcommand\bN{\ensuremath{\mathbf{N}}}
\newcommand\bJ{\ensuremath{\mathbf{J}}}
\newcommand\bQ{\ensuremath{\mathbf{Q}}}
\newcommand\bP{\ensuremath{\mathbf{P}}}
\newcommand\bW{\ensuremath{\mathbf{W}}}
\newcommand\bT{\ensuremath{\mathbf{T}}}
\newcommand\bF{\ensuremath{\mathbf{F}}}
\newcommand\bS{\ensuremath{\mathbf{S}}}
\newcommand\bR{\ensuremath{\mathbf{R}}}
\newcommand\bU{\ensuremath{\mathbf{U}}}
\newcommand\bV{\ensuremath{\mathbf{V}}}
\newcommand\bz{\ensuremath{\mathbf{z}}}
\newcommand\by{\ensuremath{\mathbf{y}}}
\newcommand\be{\ensuremath{\mathbf{e}}}
\newcommand\br{\ensuremath{\mathbf{r}}}
\newcommand\bu{\ensuremath{\mathbf{u}}}
\newcommand\bx{\ensuremath{\mathbf{x}}}
\newcommand\bb{\ensuremath{\mathbf{b}}} 
\newcommand\bZ{\ensuremath{\mathbf{Z}}}
\newcommand\bY{\ensuremath{\mathbf{Y}}}
\newcommand\bSigma{\ensuremath{\boldsymbol{\Sigma}}}

\newcommand\bPil{\ensuremath{\boldsymbol{\Pi}_{\mathrm{l}}}}
\newcommand\bPir{\ensuremath{\boldsymbol{\Pi}_{\mathrm{r}}}}

\DeclareMathOperator{\im}{im}

\DeclareMathOperator{\vect}{vec}

\usepackage[square,sort&compress,comma,numbers]{natbib}% Citation support using natbib.sty
\bibpunct[, ]{[}{]}{,}{n}{,}{,}% Citation support using natbib.sty
\newcounter{mymac@matlab}
\newcommand{\matlab}{MATLAB% 
	\ifnum\value{mymac@matlab}<1%
	\textsuperscript{\textregistered}%
	\setcounter{mymac@matlab}{1}%
	\fi%
}

\theoremstyle{plain}% Theorem-like structures provided by amsthm.sty
\newtheorem{theorem}{Theorem}[section]
\newtheorem{lemma}[theorem]{Lemma}

\newtheorem{proposition}[theorem]{Proposition}

\theoremstyle{definition}
\newtheorem{definition}[theorem]{Definition}

\newtheorem{assumption}[theorem]{Assumption}

\theoremstyle{remark}
\newtheorem{remark}{Remark}

\usepackage{amssymb}
\usepackage{amsthm}

\begin{document}
  
%%%%%%%%%%%%%%%%%%%%%%%%%%%%%%%%%%%%%%%%%%%%%%%%%%%%%%%%%%%%%%%%%%%%%%%%%%%%%%%%
% PAPER INFORMATION.                                                           %
%%%%%%%%%%%%%%%%%%%%%%%%%%%%%%%%%%%%%%%%%%%%%%%%%%%%%%%%%%%%%%%%%%%%%%%%%%%%%%%%

\title{Model Reduction of Parametric Differential-Algebraic Systems by Balanced Truncation}
  
\author[$\ast$]{J. Przybilla}
\affil[$\ast$]{Max Planck Institute for Dynamics of Complex Technical Systems, Sandtorstra{\ss}e 1, 39106 Magdeburg, Germany.\authorcr
  \email{przybilla@mpi-magdeburg.mpg.de}, \orcid{0000-0002-8703-8735}}
  
\author[$\dagger$]{M. Voigt}
\affil[$\dagger$]{UniDistance Suisse, Schinerstrasse 18, 3900 Brig, Switzerland.\authorcr
  \email{matthias.voigt@fernuni.ch}, \orcid{0000-0001-8491-1861}}
  
\shorttitle{MOR of Parametric DAE Systems by Balanced Truncation}
\shortauthor{J. Przybilla, M. Voigt}
\shortdate{}
  

%37C83 Dynamical systems with singularities (billiards,etc.)
%65F10 Iterative numerical methods for linear systems
%65F45 Numerical methods for matrix equations
%65F50 Computational methods for sparse matrices
%65F55 Numerical methods for low-rank matrix approximation; matrix compression
%65M15 Error bounds for initial value and initialboundary value problems involving PDEs
%65N22 Numerical solution of discretized equations for boundary value problems involving PDEs
%70G60 Dynamical systems methods for problems in mechanics
%76D05 Navier-Stokes equations for incompressible viscous fluids
  
\abstract{%
We deduce a procedure to apply balanced truncation to parameter-depen\-dent differential-algebraic systems.
For that we solve multiple projected Lyapunov equations for different parameter values to compute the Gramians that are required for the truncation procedure.
As this process would lead to high computational costs if we perform it for a large number of parameters, we combine this approach with the reduced basis method that determines a reduced representation of the Lyapunov equation solutions for the parameters of interest. Residual-based error estimators are then used to evaluate the quality of the approximations.
After introducing the procedure for a general class of differential-algebraic systems we turn our focus to systems with a specific structure, for which the method can be applied particularly efficiently.
We illustrate the efficiency of our approach on several models from fluid dynamics and mechanics.
%We further consider an application of the method in the context of damping optimization.
}

\novelty{We propose a new method that reduces parametric differential-algebraic equations by combining projection methods for differential-algebraic equations and the reduced basis methods for Lyapunov equations. To apply the reduced basis method new error estimators are invented. The performance of our new method is illustrated for several examples.}

\maketitle

%%%%%%%%%%%%%%%%%%%%%%%%%%%%%%%%%%%%%%%%%%%%%%%%%%%%%%%%%%%%%%%%%%%%%%%%%%%%%%%%
% PAPER CONTENT.                                                               %
%%%%%%%%%%%%%%%%%%%%%%%%%%%%%%%%%%%%%%%%%%%%%%%%%%%%%%%%%%%%%%%%%%%%%%%%%%%%%%%%

\section{Introduction}\label{sec:intro}
In the modeling of various industrial processes one often obtains systems of differential-algebraic equations (DAEs) that are of high dimension.
Typical examples are electrical circuits, thermal and diffusion processes, multibody systems, and specific discretized partial differential equations \cite{Campbell1980, Brenan1995}.
Often, these systems are further dependent on parameters that can describe physical properties such as material constants and which are often subject to optimization.
The resulting parameter-dependent differential-algebraic systems have the form
\begin{subequations}\label{eq:FOM}
\begin{align} 
\frac{\mathrm{d}}{\mathrm{d}t}\bE(\mu){\bz}(t) &= \bA(\mu)\bz(t) + \bB(\mu)\bu(t),\label{eq:FOM_1}\\
\by(t) &= \bC(\mu)\bz(t),\label{eq:FOM_2}
\end{align}
\end{subequations}
where $\bE(\mu),\, \bA(\mu)\in \mathbb{R}^{N, N}$, $\bB(\mu)\in\mathbb{R}^{N, 
m}$, and $\bC(\mu)\in\mathbb{R}^{p,N}$ are dependent on parameters 
$\mu\in\mathcal{D}$, where $\mathcal{D}\subset\mathbb{R}^{d}$.
The input, the state, and the output are denoted by $\bu(t)\in\mathbb{R}^{m}$, 
$\bz(t)\in\mathbb{R}^{N}$ and $\by(t)\in\mathbb{R}^p$ for $t \in [0,\infty)$.
The matrix $\bE(\mu)$ can be a singular. 
However, throughout this paper, we assume \emph{uniform regularity}, i.\,e., the pencil $s\bE(\mu)- \bA(\mu)$, which is a matrix polynomial in the variable $s$, is assumed to be regular for all $\mu \in 
\mathcal{D}$, that is $\det(s\bE(\mu)- \bA(\mu))$ is not the zero polynomial for 
all $\mu \in \mathcal{D}$. Moreover, we assume that the system under consideration is \emph{uniformly asymptotically stable}, i.\,e., the system (or the pencil $s\bE(\mu)-\bA(\mu)$) is asymptotically stable for all $\mu \in \mathcal{D}$, meaning that all finite eigenvalues of the pencil $s\bE(\mu)- \bA(\mu)$ have negative real part for all $\mu \in \mathcal{D}$.
Throughout the derivations provided in this work we assume 
%w.\,l.\,o.\,g. that 
%$0 \in 
%\mathcal{D}$
%and
that $\bE(\cdot)$, $\bA(\cdot)$, $\bB(\cdot)$, 
and $\bC(\cdot)$ are \emph{affine in the parameter $\mu$}, i.\,e.,
\begin{align}\label{eq:affinedeco}
\begin{split}
	\bE(\mu)  &= \bE_0 + \sum_{k=1}^{n_{\bE}}\Theta_{k}^{\bE}(\mu)\bE_{k},  
\quad     \bA(\mu)  = \bA_0 + \sum_{k=1}^{n_{\bA}}\Theta_{k}^{\bA}(\mu)\bA_{k}, 
\\
\bB(\mu)  &= \bB_0 + \sum_{k=1}^{n_{\bB}}\Theta_{k}^{\bB}(\mu)\bB_{k},  
\quad     \bC(\mu)  = \bC_0 + \sum_{k=1}^{n_{\bC}}\Theta_{k}^{\bC}(\mu)\bC_{k},
\end{split}
\end{align}
where all 
$\Theta_{k}^{\bE},\,\Theta_{k}^{\bA},\,\Theta_{k}^{\bB},\,\Theta_{k}^{\bC}
:\mathcal{D} \to \mathbb{R}$ are parameter-dependent continuous functions
%with 
%$\Theta_{k}^{\bE}(0) = \Theta_{k}^{\bA}(0) = \Theta_{k}^{\bB}(0) = \Theta_{
%k}^{\bC}(0) = 0$ 
and 
$\bE_{k}$, $\bA_{k}$, $\bB_k$, $\bC_k$ are parameter-independent matrices. 
%(If $0 \notin \mathcal{D}$, then we can pick some $\mu_0 \in \mathcal{D}$ and 
%redefine $\mathcal{D}$ as $\mathcal{D} - \mu_0$ and the functions 
%$\Theta_{k}^{\bE},\,\Theta_{k}^{\bA},\,\Theta_{k}^{\bB},\,\Theta_{k}^{\bC}$ 
%accordingly to achieve this condition.)
For reasons of computational efficiency, we always assume that 
$n_{\bE},\,n_{\bA},\,n_{\bB},\,n_{\bC} \ll N$ and we also assume that the number of inputs and outputs is small compared to the dimension of the states, i.\,e., $m,\,p \ll N$. 
For ease of notation, it is sometimes handy to put the matrices $\bE_0$, $\bA_0$, $\bB_0$, and $\bC_0$ into 
the sums in \eqref{eq:affinedeco} and multiply them with the factors 
$\Theta_{0}^{\bE} = \Theta_{0}^{\bA} = \Theta_{0}^{\bB} = \Theta_{
0}^{\bC} \equiv 1$. The model \eqref{eq:FOM} is also referred to as the 
full-order model (FOM).
Often, one also considers the input/output mapping of the system in the frequency domain. This relation is typically expressed by the \emph{transfer function}
$
\boldsymbol{\mathcal{G}}(\mu, s):=\bC(\mu)(s\bE(\mu)-\bA(\mu))^{-1}\bB(\mu).
$

Because of the high state-space dimension $N$ of \eqref{eq:FOM} it is useful to apply reduction methods to extract the essential information of the system and its solution. More precisely, we want to determine a (uniformly regular and asymptotically stable) reduced-order model (ROM) 
\begin{align}\label{eq:DAE}
\begin{split}
\frac{\mathrm{d}}{\mathrm{d}t}\bE_{\mathrm{R}}(\mu){\bz}_{\mathrm{R}}(t) &= \bA_{\mathrm{R}}(\mu)\bz_{\mathrm{R}}(t) + \bB_{\mathrm{R}}(\mu)\bu(t),\\
\by_{\mathrm{R}}(t) &= \bC_{\mathrm{R}}(\mu)\bz_{\mathrm{R}}(t),
\end{split}
\end{align}
with $\bE_{\mathrm{R}}(\mu),\, \bA_{\mathrm{R}}(\mu)\in \mathbb{R}^{r, r}$, 
$\bB_{\mathrm{R}}(\mu)\in\mathbb{R}^{r, m}$, and 
$\bC_{\mathrm{R}}\in\mathbb{R}^{p,r}$, where $r\ll N$.
The reduced state and output are $\bz_{\mathrm{R}}(t)\in\mathbb{R}^{r}$ and 
$\by_{\mathrm{R}}(t)\in\mathbb{R}^{p}$. The aim is to determine the surrogate 
model in such a way that the output of the ROM well-approximates the output of 
the FOM for all admissible inputs $\bu(\cdot)$ and all parameters $\mu \in 
\mathcal{D}$. In terms of the transfer function this means that the error 
$\left\| \boldsymbol{\mathcal{G}}(\mu,\cdot) - 
\boldsymbol{\mathcal{G}}_{\mathrm{R}}(\mu,\cdot) \right\|$ is sufficiently small 
in an appropriate norm for all $\mu \in \mathcal{D}$, where 
$\boldsymbol{\mathcal{G}}_{\mathrm{R}}(\mu,s)$ denotes the transfer function of 
the ROM.

For parameter-independent problems, i.\,e.,
\begin{equation*}
\bE(\mu) \equiv \bE, \quad \bA(\mu) \equiv \bA, \quad \bB(\mu) \equiv \bB, \quad \bC(\mu) \equiv \bC, \quad \boldsymbol{\mathcal{G}}(\mu,s) \equiv \boldsymbol{\mathcal{G}}(s),
\end{equation*}
there are several classes of model order reduction techniques. 
Examples are singular value based approaches like balanced truncation \cite{morMoo81, morTomP87, morBenOCetal17} and Hankel norm approximations \cite{Glo84}.
Additionally, there are Krylov subspace-based methods such as the iterative 
rational Krylov algorithm (IRKA) \cite{morGugAB08, morGugSW13, morBenOCetal17, 
morFlaBG12} and moment matching as well as data-driven methods such as the 
Loewner framework \cite{morMayA07,morGosPA21b,morBeaGM22}.
Corresponding methods for parameter-dependent systems are introduced in \cite{morBenKTetal16, morBauBF14, Geuss2013}.

In this article we focus on balanced truncation which is one of the most 
popular reduction techniques. This is mainly due to the guaranteed asymptotic 
stability of the ROM, the existence of an error bound, and appealing numerical 
techniques for the involved Lyapunov equations \cite{DruS11, BenS13, 
Schmidt2018, morSonS17}. Within balanced truncation, Lyapunov equations 
corresponding to the original system need to be solved.
The solutions of these equations are called Gramians. They describe the input-to-state and state-to-output behavior and are used to construct projection matrices for the reduction.
The multiplication of the system matrices with these projection matrices then results in a ROM.

All the above-mentioned methods focus on the case $\bE = I_N$ and are, however, not directly applicable in the DAE case. Even if the problem is not parameter-dependent, there are several challenges that one has to face (here we assume for simplicity, that the problem is parameter-independent):
\begin{enumerate}[a)]
    \item Since the matrix $\bE$ is typically singular, the transfer function $\boldsymbol{\mathcal{G}}(s)$ is possibly improper, i.\,e., it may have a polynomial part which is unbounded for growing values of $|s|$. If the reduced transfer function $\boldsymbol{\mathcal{G}}_{\mathrm{R}}(s)$ does not match this polynomial part, then the error transfer function $\boldsymbol{\mathcal{G}}(s) - \boldsymbol{\mathcal{G}}_{\mathrm{R}}(s)$ is not an element of the rational Hardy spaces $\mathcal{RH}_\infty^{p , m}$ or $\mathcal{RH}_2^{p , m}$ (see, e.\,g, \cite{ZhoDG96} for a definition) and the output error cannot be bounded by the input norm. Thus, a model reduction scheme for DAEs must preserve the polynomial part of its transfer function. This is addressed in \cite{morMehS05,morSty04,morGugSW13}.
\item In balanced truncation, one has to solve two large-scale Lyapunov equations. In the DAE case, one has to consider a generalization of these -- so-called \emph{projected} Lyapunov equations. 
Thereby, the Lyapunov equations are projected onto the left and right deflating subspaces corresponding to the finite or infinite eigenvalues of the matrix pencil $s\bE - \bA$.
%There, the Lyapunov equations are projected to the subspace of the DAE in which the dynamics evolves. 
This involves specific projectors which are hard to form explicitly and that might destroy the sparsity structure of the coefficient matrices. 
However, for DAE systems of special structure, the projectors are of a particular form which can be exploited from the numerical point of view. 
%More precisely, in the solution algorithms for Lyapunov equations, the projectors can be applied implicitly without forming them.
For details, we refer to \cite{Sty08,  morHeiSS08, morSaaV18, morBenSU16}.
\end{enumerate}

Besides the approaches mentioned above, some of the issues in model reduction of DAEs are also addressed in \cite{morAliBSetal14,morBanAS16}. There, the authors introduce an index-aware model order reduction scheme that splits the DAE system into a differential and an algebraic part. However, in this method it is often difficult to maintain sparsity of the system matrices, and hence, it is in general not feasible for many systems with large dimensions. Also data-driven approaches have been recently extended to differential-algebraic systems, see, e.\,g., \cite{morAntGH20b,morGosZA20}. Finally, we would like to mention \cite{morBenS17} which provides an  overview of model order reduction for DAE systems.  

% %the authors use rational interpolation to find $\mathcal{H}_2$-optimal approximations for DAE systems.
% Furthermore, in \cite{morAliBSetal14,morBanAS16}, the authors introduce an index-aware model order reduction scheme that splits the DAE system into a differential and an algebraic part. 
% However, in this method it is often difficult to maintain sparsity of matrices, and hence, this method is in general not feasible for many systems with large dimensions.
% %\color{black}

Different approaches for parametric model reduction have been developed as well. Besides others, there exist two main classes of such methods. First, one could determine \emph{global bases} for the projection spaces for the entire parametric model, e.\,g., by concatenating the local bases for different parameter values, see, e.\,g., \cite{morBauBBetal11}. Second, there exist approaches for constructing \emph{local bases} for a specific parameter value by combining local basis information of several prereduced models at different parameter values, e.\,g., by interpolating between different models \cite{Eid2011,Geuss2013}. We also refer to \cite{morBenGW15} for an extensive survey on such methods.

However, to the best of our knowledge, none of the techniques presented in \cite{morBenGW15} has already been extended to the case of DAE systems. Therefore, the motivation of this paper is to develop a first method for this purpose based on balanced trunction. If we aim to reduce a parameter-dependent system by balanced truncation in a naive fashion, we would have to solve the Lyapunov equations for each individual parameter of interest which is computationally prohibitive. 
%\color{blue}
To avoid these computational costs, for the case $\bE=I_N$, the authors in \cite{morSonGMS21} apply some interpolation techniques to approximate the Gramians for all admissible parameters.
The authors in \cite{morSonS17}, on the other hand, apply the reduced basis method which is a well-established method to reduce parameter-dependent partial differential equations \cite{morHesRS16, morQuaMN16}. Using this method, the authors determine the approximate subspaces in which the solutions of the corresponding Lyapunov equations are contained in. 

In this paper we extend the reduced basis method to the case of DAE systems where the parametric projected Lyapunov equations are only solved for few sampling parameters. 
%\color{black}
Then, based on these solutions, a reduced subspace in which the Lyapunov equation solutions for all $\mu \in \mathcal{D}$ approximately live, is constructed. 
The latter steps form the computationally expensive \emph{offline phase}.
After that, using the reduced basis representation, the Lyapunov equations can be solved much more efficiently for all $\mu \in \mathcal{D}$ in the \emph{online phase}.
A crucial question in the offline phase is the choice of the sample parameters. Usually, a grid of test parameters is selected. For these, the error is quantified using a posteriori error estimators. Then new samples are taken at those parameters at which the error estimator estimates the largest error.

In this paper we generalize the reduced basis balanced truncation method of \cite{morSonS17} to differential-algebraic systems with a focus on structured systems from specific application areas. 
The main problems we solve in this paper are the following:
\begin{enumerate}[a)]
  \item We derive error estimators for parameter-dependent projected 
Lyapunov equations. These can be improved if the given system is uniformly 
strictly dissipative. Since this condition is not always satisfied, we discuss 
transformations to achieve this for systems of a particular index 3 structure. 
  \item We discuss in detail how the polynomial part of parameter-dependent 
transfer functions of DAE systems can be extracted by efficiently approximating 
their Markov parameters. 
  \item We apply this approach to several application problems and illustrate its effectiveness. %in optimization problems for mechanical systems.  
\end{enumerate}

Our method also follows the paradigm of decomposing the procedure into an offline and online phase.
In the offline phase we approximate the image spaces of controllability and obervability Gramians that are the solutions of projected Lyapunov equations.
These approximate image spaces are then used in the online phase, to determine approximations of the Gramians, and hence, the reduced order model, efficiently for all required parameters.  
%With the help of these Grammian's approximations projections can then be derived that reduce the system.
The online phase can be performed for any parameter $\mu \in \mathcal{D}$. Due to the reduced dimension obtained by the offline phase, this step is very fast and can be performed repeatedly for all required parameters. Moreover, in this work, we introduce new error estimators, which are used in the offline phase to assess the quality of the reduction.

This paper is organized as follows.
In Section \ref{sec:examples}, two model problems are introduced that motivate the method presented in this paper.
In Section \ref{sec:prelim}, we review existing concepts and approaches that will be used later in this paper. 
Thereby, in Subsection \ref{ChapterProj}, we recall projection techniques to decouple the differential and algebraic equations in \eqref{eq:FOM}. 
Afterwards, in Subsection \ref{sec:MOR}, we consider model order reduction by balanced truncation for DAE systems. We further address the numerical challenges that arise in computing the solutions of the required Lyapunov equations.
Since solving Lyapunov equations for every requested parameter leads to high 
computational costs, in Section~\ref{reduced basis method} the reduced basis method is 
presented, which was first applied to standard Lyapunov equations in 
\cite{morSonS17}.
We derive two error estimators for our method, one of them is motivated by the 
estimator from \cite{morSonS17}, the other one is an adaption of the estimator 
presented in \cite{Schmidt2018}. Here we also discuss the treatment of the 
algebraic part of the system that may lead to polynomial parts in the transfer 
function.
Finally, in Section \ref{ChapterResults}, we evaluate the method of this paper by applying it to our model problems from Section~\ref{sec:examples}.

\section{Model problems}\label{sec:examples}
\subsection{Problem 1: Stokes-like DAEs}
The first example is the system 
\begin{align}\label{Index2}
\begin{split}
\frac{\mathrm{d}}{\mathrm{d}t}\begin{bmatrix}
E(\mu) & 0 \\
0 & 0
\end{bmatrix}\begin{bmatrix}
x(t)\\
{\lambda}(t)
\end{bmatrix}&=\begin{bmatrix}
A(\mu) & G(\mu) \\
G(\mu)^{\mathrm{T}} & 0
\end{bmatrix}\begin{bmatrix}
x(t) \\
\lambda(t)
\end{bmatrix} + \begin{bmatrix}
B_1(\mu)\\
B_2(\mu)
\end{bmatrix}u(t),\\
y(t) &= \begin{bmatrix}
C_1(\mu) & C_2(\mu)
\end{bmatrix}\begin{bmatrix}
x(t) \\
\lambda(t)
\end{bmatrix},
\end{split}
\end{align}
which arises, e.\,g., if we discretize the incompressible Navier-Stokes equation. The parameter-independent version is presented in \cite{morMehS05,morSty04}.
The system matrices are dependent on a parameter $\mu\in\mathcal{D}$, where $\mathcal{D}\subset\mathbb{R}^d$ is the parameter domain.
For fluid-flow problems, the matrices $E(\mu),\, A(\mu)\in\mathbb{R}^{n,n}$ represent the masses and the discretized Laplace operator. 
Naturally, it holds that $E(\mu) = E(\mu)^{\mathrm{T}}>0$ and $A(\mu)=A(\mu)^{\mathrm{T}}<0$ for all $\mu \in \mathcal{D}$, where $H>0$ ($H\geq 0$) denotes a positive (semi)definite matrix $H$ and correspondingly $H<0$ ($H\leq 0$) denotes a negative (semi)definite matrix.
The discrete gradient is given by $G(\mu)\in\mathbb{R}^{n,q}$ which we assume to be of full column rank for all $\mu \in \mathcal{D}$.
The matrices $B_1(\mu)\in\mathbb{R}^{n,m},\, B_2(\mu)\in\mathbb{R}^{q,m}$ and $C_1(\mu)\in\mathbb{R}^{p,n}, \, C_2(\mu)\in\mathbb{R}^{p,q}$ are the input and the output matrices, respectively.
The state of the system at time $t$ is given by $x(t)\in\mathbb{R}^{n}$ and $\lambda(t)\in\mathbb{R}^{q}$.
The vectors $u(t)\in\mathbb{R}^{m}$ and $y(t)\in\mathbb{R}^{p}$ are the input and output of the system.

In view of \eqref{eq:affinedeco} we assume that the symmetry and definitess of 
$E(\cdot)$ and $A(\cdot)$ is represented in the affine parameter structure as
\begin{equation*}
	E(\mu)  = \sum_{k=0}^{n_{E}}\Theta_{k}^{E}(\mu)E_{k}\quad \text{and} 
\quad A(\mu)  = \sum_{k=0}^{n_{A}}\Theta_{k}^{A}(\mu)A_{k},
\end{equation*}
where $\Theta_{k}^{E}(\cdot)$, $\Theta_{k}^{A}(\cdot)$ are  
\emph{positive} parameter-dependent continuous functions % with 
%$\Theta_{k}^{E}(0) = \Theta_{k}^{A}(0) = 0$ for $k \ge 1$
%and where we also use 
with the convention $\Theta_{0}^{E} = \Theta_{0}^{A} \equiv 1$. Furthermore, $E_0 = 
E_0^\mathrm{T} > 0$, $A_0 = A_0^\mathrm{T} < 0$, $E_{k} = E_k^\mathrm{T} \ge 0$, 
$A_{k} = A_k^{\mathrm{T}} \le 0$ for $k \ge 1$ are parameter-independent 
matrices. Moreover, for reasons of computational efficiency, we always assume 
that $n_{E},\, n_{A} \ll n$. 

\subsection{Problem 2: Mechanical systems}\label{sec:2.2}
The second system we consider is of the form 
\begin{align}\label{Index3}
\begin{split}
\frac{\mathrm{d}}{\mathrm{d}t}\begin{bmatrix}
I_{n_x} & 0 & 0 \\
0 & M(\mu) & 0 \\
0 & 0  & 0
\end{bmatrix}\begin{bmatrix}
x_1(t)\\
x_2(t)\\
\lambda(t)
\end{bmatrix} &= \begin{bmatrix}
0 & I_{n_x} & 0\\
-K(\mu) & -D(\mu) & G(\mu)\\
G(\mu)^{\mathrm{T}} & 0 & 0
\end{bmatrix}\begin{bmatrix}
x_1(t)\\
x_2(t)\\
\lambda(t)
\end{bmatrix} + \begin{bmatrix}
0\\
B_{x}(\mu)\\
0
\end{bmatrix}u(t),\\
y(t) &= \begin{bmatrix}
	C_{x}(\mu) & 0&0
\end{bmatrix}\begin{bmatrix}
x_1(t)\\
x_2(t)\\
\lambda(t)
\end{bmatrix},
\end{split}
\end{align}
which results from a linearization of the spring-mass-damper model presented in \cite{morMehS05}.
The mass, damping, and stiffness matrices $M(\mu), \, D(\mu),\, K(\mu)\in\mathbb{R}^{n_x,n_x}$ are assumed to be symmetric and positive definite for all $\mu\in\mathcal{D}$.
The matrices $B_{x}(\mu)\in\mathbb{R}^{n_x,m}$ and $C_{x}(\mu)\in\mathbb{R}^{p,n_x}$ are the input and the output matrices.
The matrix $G(\mu)\in\mathbb{R}^{n_x,q}$ reflects algebraic constraints on the system and is of full column rank.
In this example, the state includes the displacement 
$x_1(t)\in\mathbb{R}^{n_x}$, the velocity $x_2(t)\in\mathbb{R}^{n_x}$, and 
Lagrange multiplier $\lambda(t)\in\mathbb{R}^{q}$.

For convenience, we also define
\begin{align}\label{eq:EAIndex3}
\begin{split}
E(\mu) &:= \begin{bmatrix}
I_{n_x}&0\\
0&M(\mu)
\end{bmatrix},\quad A(\mu):= \begin{bmatrix}
0 & I_{n_x}\\
-K(\mu) & -D(\mu)
\end{bmatrix},\quad  B(\mu):=\begin{bmatrix}
0\\
B_{x}(\mu)
\end{bmatrix},\\
C(\mu) &:=\begin{bmatrix}
C_{x}(\mu) & 0
\end{bmatrix}, \quad n:=2n_x.
\end{split}
\end{align}
Then with this notation, we can write the mechanical system similarly as in \eqref{Index2} with the difference that the off-diagonal blocks of the state matrix are not the transposes of each other.

Similarly to the first model problem, we assume that $M(\cdot)$, $D(\cdot)$, and $K(\cdot)$ can be written as 
\begin{equation*}
	M(\mu)  = \sum_{k=0}^{n_{M}}\Theta_{k}^{M}(\mu)M_{k}, \quad D(\mu)  
= \sum_{k=0}^{n_{D}}\Theta_{k}^{D}(\mu)D_{k},\quad \text{and} \quad  
K(\mu)  = \sum_{k=0}^{n_{K}}\Theta_{k}^{K}(\mu)K_{k},
\end{equation*}
where $\Theta_{k}^{M}(\cdot)$, $\Theta_{k}^{D}(\cdot)$, $\Theta_{k}^{K}(\cdot)$ 
are \emph{positive} parameter-dependent continuous functions %with 
%$\Theta_{k}^{M}(0) = \Theta_{k}^{D}(0) = \Theta_{k}^{K}(0) = 0$ for $k \ge 
%1$ 
using the convention $\Theta_{0}^{M} = 
\Theta_{0}^{D} = \Theta_{0}^{K} \equiv 1$. Furthermore, $M_0 = 
M_0^\mathrm{T} > 0$, $D_{0} = D_0^\mathrm{T} > 0$, $K_0 = K_0^\mathrm{T} > 0$ 
and $M_{k} = M_k^\mathrm{T}\geq 0$, $D_{k} = D_k^\mathrm{T}\geq 0$, $K_{k} = K_k^\mathrm{T}\geq 0$ 
for $k \ge 1$ are parameter-independent matrices with $n_{M},\, n_{D}, \, 
n_{K}\ll n_x$.

\section{Preliminaries}\label{sec:prelim}
\subsection{Decoupling of differential and algebraic equations}\label{ChapterProj}
In this section we briefly describe the projection technique for 
decoupling the algebraic constraints in DAEs of the form \eqref{eq:FOM} from 
the remaining differential equations
for a fixed parameter $\mu \in \mathcal{D}$ and under the assumptions of regularity and asymptotic stability. 
Therefore, for simplicity of presentation, we omit $\mu$ in this section. 
The details of this can be found in \cite{morSty04,morSty06,morHeiSS08,morSaaV18}.

As shown in \cite{BerIT12} there exist regular matrices 
$\bW,~\bT\in\mathbb{R}^{N,N}$ that transform the system matrices into 
quasi-Weierstra{\ss} form, that is
\begin{equation}\label{eq:QWF}
\bE = \bW\begin{bmatrix}
I_{N_\mathrm{f}} & 0 \\ 0 & \bN
\end{bmatrix}\bT, \qquad
\bA = \bW\begin{bmatrix}
\bJ & 0 \\ 0 & I_{N_{\infty}}
\end{bmatrix}\bT,
\end{equation}
where $\bJ\in\mathbb{R}^{{N_\mathrm{f}}, {N_\mathrm{f}}}$ and 
$\bN\in\mathbb{R}^{N_{\infty}, N_{\infty}}$ is nilpotent with nilpotency index 
$\nu$ that coincides with the \emph{(Kronecker) index} of the DAE system 
\eqref{eq:FOM}. Here, ${N_\mathrm{f}}$ is the number of the finite eigenvalues 
of the matrix pencil $s\bE-\bA$ and $N_{\infty}$ is  the number of the infinite 
ones.
Accordingly, spectral projectors onto the left and right deflating subspace of 
$s\bE-\bA$ corresponding to the ${N_\mathrm{f}}$ finite eigenvalues are defined 
as 
\[
\bPil := \bW\begin{bmatrix}
I_{N_\mathrm{f}}& 0 \\ 0 & 0
\end{bmatrix}\bW^{-1}, \qquad \bPir := \bT^{-1}\begin{bmatrix}
I_{N_\mathrm{f}}& 0 \\ 0 & 0
\end{bmatrix}\bT.
\]
We define 
\begin{align*}
\bF_{\bJ}(t) &:= \bT^{-1}\begin{bmatrix}
\mathrm{e}^{\bJ t} & 0 \\ 0 & 0
\end{bmatrix}\bW^{-1},  \quad t \in \mathbb{R}\qquad\text{and} \\
\bF_{\bN}(k) &:= \bT^{-1}\begin{bmatrix}
0 & 0 \\ 0 & -\bN^k 
\end{bmatrix}\bW^{-1},\quad   k=0,\,\dots,\,\nu-1
\end{align*}
to derive the states
\[
\bz_{\mathrm{p}}(t):= \int_{0}^{t}\bF_{\bJ}(t-\tau)\bB 
\bu(\tau)\,\mathrm{d}\tau, 
\qquad
\bz_{\mathrm{i}}(t):= \sum_{k = 0}^{\nu-1}\bF_{\bN}(k)\bB \bu^{(k)}(t),
\]
under the assumption that the input $\bu(\cdot)$ is sufficiently smooth {and that the consistency conditions are satisfied, i.\,e., we assume that $\bz_{\mathrm{i}}(0):= \sum_{k = 0}^{\nu-1}\bF_{\bN}(k)\bB \bu^{(k)}(0)$}.
They satisfy $\bz(t) = \bz_{\mathrm{p}}(t) + \bz_{\mathrm{i}}(t)$ such that $\bz(t)$ solves equation \eqref{eq:FOM_1}.
We refer to $\bz_{\mathrm{p}}(t)$ and $\bz_{\mathrm{i}}(t)$ as differential and algebraic states.
%We aim to find a reduced approximation of the proper state $x_{\mathrm{p}}(t)$ while deriving a minimal realization of the improper state $x_{\mathrm{i}}(t)$.
%In the following we want to describe the controllability and observability spaces of the differential and algebraic states. 
%These can then be used to reduce these states accordingly.
%Therefor, we make use of the proper and improper controllability Graminas of the system 
These states are associated with the \emph{proper and improper controllability Gramians}
\begin{align}\label{eq:contrGram}
\bP_{\mathrm{p}} = 
\int_{0}^{\infty}\bF_{\bJ}(t)\bB\bB^{\mathrm{T}}\bF_{\bJ}(t)^{\mathrm{T}}\,
\mathrm { d } t , \qquad \bP_{\mathrm{i}} = \sum_{k=0}^{\nu 
-1}\bF_{\bN}(k)\bB\bB^{\mathrm{T}}\bF_{\bN}(k)^{\mathrm{T}},
\end{align}
respectively. They span the reachable subspaces of the differential and algebraic states. 
Analogously, the \emph{proper and improper observability Gramians}
\begin{align*}
\bQ_{\mathrm{p}} = 
\int_{0}^{\infty}\bF_{\bJ}(t)^{\mathrm{T}}\bC^{\mathrm{T}}\bC\bF_{\bJ}(t)\,
\mathrm{d} t , \qquad \bQ_{\mathrm{i}} = \sum_{k=0}^{\nu 
-1}\bF_{\bN}^{\mathrm{T}}(k)\bC^{\mathrm{T}}\bC\bF_{\bN}(k)
\end{align*}
span the corresponding observability spaces.
The proper Gramians are obtained by solving the \emph{projected continuous-time Lyapunov equations} 
\begin{align}
\bE\bP_{\mathrm{p}} \bA^{\mathrm{T}} + \bA\bP_{\mathrm{p}}\bE^{\mathrm{T}} &= - \bPil \bB\bB^{\mathrm{T}} \bPil^{\mathrm{T}}, \qquad \bP_{\mathrm{p}} =\bPir \bP_{\mathrm{p}} \bPir^{\mathrm{T}},\label{eq:propLE_contr}\\
\bE^{\mathrm{T}}\bQ_{\mathrm{p}} \bA + \bA^{\mathrm{T}}\bQ_{\mathrm{p}}\bE &= - 
\bPir^{\mathrm{T}} \bC^{\mathrm{T}}\bC \bPir, \qquad \bQ_{\mathrm{p}} = 
\bPil^{\mathrm{T}} \bQ_{\mathrm{p}} \bPil \label{eq:propLE_observ}
\end{align}
and the improper ones by solving the \emph{discrete-time Lyapunov equations} 
\begin{align}
\bA\bP_{\mathrm{i}}\bA^{\mathrm{T}} - \bE\bP_{\mathrm{i}}\bE^{\mathrm{T}} &= (I_N- \bPil) \bB\bB^{\mathrm{T}} (I_N - \bPil)^{\mathrm{T}}, \qquad 0 = \bPir \bP_{\mathrm{i}} \bPir^{\mathrm{T}},\label{eq:impropLE_contr}\\
\bA^{\mathrm{T}}\bQ_{\mathrm{i}} \bA -\bE^{\mathrm{T}}\bQ_{\mathrm{i}}\bE &= 
(I_N-\bPir)^{\mathrm{T}} \bC^{\mathrm{T}}\bC (I_N-\bPir), \qquad 0 = 
\bPil^{\mathrm{T}} \bQ_{\mathrm{i}} \bPil.\label{eq:impropLE_obs}
\end{align}
In Subsection \ref{ChapterLE} we will describe how these projected Lyapunov equations can be solved.

In practice, the projectors $\bPil$ and $\bPir$ are computationally expensive to 
form and even more expensive to compute with as they are dense matrices of 
large dimension in general.
To determine them, the derivations from \cite{KunM06} can be applied, however
they are not numerically stable since a Jordan canonical form needs to be 
computed.
On the other hand, the methods presented in \cite{Ban14}, \cite{CamMR96}, or 
\cite{Mae96} might be used to generate the differential and algebraic state 
spaces, however, they involve the determination of matrix kernels which 
are based on matrix decompositions. 
Suitable parts of the matrices involved in these decompositions are then used to determine the projection matrices $\bPil$ and $\bPir$.
They are in general not of a sparse structure and, therefore, not applicable if we consider matrices of large dimensions.
Hence, if possible, it is advantageous to determine the projection matrices 
analytically. In many application fields this is possible because the 
matrices $\bE$ and $\bA$ are of special structure that can be utilized for this 
purpose. 
The examples in Section~\ref{sec:examples} have such a numerically advantageous structure. Further examples can be found in \cite{Sty08,Mae96}.

\subsection{Model reduction of differential-algebraic systems}\label{sec:MOR}
The aim of this section is to review methods to reduce the DAE system \eqref{eq:FOM} for a fixed parameter.
We utilize balanced truncation modified for projected systems which is presented in Subsection~\ref{ChapterBT}.
Afterwards, in Subsection~\ref{ChapterLE}, we summarize the ADI and Smith
methods to solve the involved projected Lyapunov equations. Finally, in 
Subsection \ref{ssec:D0D1determ}, we present an alternative approach to 
approximate the algebraic parts of the system (potentially leading to 
polynomial parts in the transfer function) which is inspired by 
\cite{morMosSMV22}.
%that are needed for balanced truncation by adapting the ADI method and the Stein equation to projected Lyapunov equations.
%In this section we will still focus entirely on the fixed-parameter case. 

\subsubsection{Balanced truncation}\label{ChapterBT}
The aim of this subsection is to find a reduced system
\begin{align}\label{Index2red}
\begin{split}
	\frac{\mathrm{d}}{\mathrm{d}t}\bE_{\mathrm{R}}\bz_{\mathrm{R}}(t) &= 
\bA_{\mathrm{R}}\bz_{\mathrm{R}}(t) + \bB_{\mathrm{R}}\bu(t),\\
	\by_{\mathrm{R}}(t) &= \bC_{\mathrm{R}}\bz_{\mathrm{R}}(t),
	\end{split}
\end{align}
with $\bE_{\mathrm{R}},\, \bA_{\mathrm{R}} \in \mathbb{R}^{r,r},\, \bB_{\mathrm{R}}\in\mathbb{R}^{r,m},\, \bC_{\mathrm{R}}\in\mathbb{R}^{p,r}$, $r \ll N$ which approximates the input/output behavior of the original system \eqref{eq:FOM} % \eqref{Index2} 
(for a fixed parameter).% and \eqref{Index2proj}.

%We consider the controllability Gramian $P_{\Pi}$ and the observability Gramian $Q_{\Pi}$ of the projected system \eqref{Index2proj}.
%We obtain them by solving the projected Lyapunov equations
%\begin{subequations}\label{LyEq}
%\begin{align}
%	E_{\Pi}P_{\Pi}A_{\Pi}^{\mathrm{T}} + A_{\Pi}P_{\Pi}E_{\Pi}^{\mathrm{T}} &= -B_{\Pi}B_{\Pi}^{\mathrm{T}},\qquad	\Pi^{\mathrm{T}}P_{\Pi}\Pi=P_{\Pi},\label{LyEqP} \\ 
%	E_{\Pi}^{\mathrm{T}}Q_{\Pi}A_{\Pi}+ A_{\Pi}^{\mathrm{T}}Q_{\Pi}E_{\Pi} &= -C_{\Pi}^{\mathrm{T}}C_{\Pi},\qquad\Pi^{\mathrm{T}}Q_{\Pi}\Pi=Q_{\Pi}.\label{LyEqQ}
%\end{align}
%\end{subequations}
%These Gramians $P_\Pi$ and $Q_\Pi$ correspond to the proper controllability and observability Gramians as introduced by Stykel \cite{Sty02}.

%We summarize solution techniques for projected Lyapunov equations in Subsection~\ref{ChapterLE}.

Since the different Gramians are positive semidefinite, there exist factorizations
\begin{equation*}
\bP_{\mathrm{p}} = \bR_{\mathrm{p}}\bR_{\mathrm{p}}^{\mathrm{T}}, \quad \bP_{\mathrm{i}} = \bR_{\mathrm{i}}\bR_{\mathrm{i}}^{\mathrm{T}}, \quad  \bQ_{\mathrm{p}} = \bS_{\mathrm{p}}\bS_{\mathrm{p}}^{\mathrm{T}}, \quad  \bQ_{\mathrm{i}} = \bS_{\mathrm{i}}\bS_{\mathrm{i}}^{\mathrm{T}}
\end{equation*}
with factors $\bR_{\mathrm{p}},\,\bR_{\mathrm{i}},\,\bS_{\mathrm{p}},\,\bS_{\mathrm{i}} \in \mathbb{R}^{N,N}$.
We consider the singular value decompositions
\begin{align}
\bS_{\mathrm{p}}^{\mathrm{T}}\bE \bR_{\mathrm{p}} &= 
\bU_{\mathrm{p}}\bSigma_{\mathrm{p}} \bV_{\mathrm{p}}^{\mathrm{T}} = 
\begin{bmatrix}
\bU_{\mathrm{p},1} & \bU_{\mathrm{p},2}
\end{bmatrix}\begin{bmatrix}
\bSigma_{\mathrm{p},1} & 0\\
0 & \bSigma_{\mathrm{p},2}
\end{bmatrix}\begin{bmatrix}
\bV_{\mathrm{p},1}^{\mathrm{T}}\\\bV_{\mathrm{p},2}^{\mathrm{T}}
\end{bmatrix},\nonumber \\ 
\bS_{\mathrm{i}}^{\mathrm{T}}\bA \bR_{\mathrm{i}} &= 
\bU_{\mathrm{i}}\bSigma_{\mathrm{i}}\bV_{\mathrm{i}}^{\mathrm{T}}
=\begin{bmatrix}
\bU_{\mathrm{i},1} & \bU_{\mathrm{i},2}
\end{bmatrix}\begin{bmatrix}
\bSigma_{\mathrm{i},1} & 0\\
0 & 0
\end{bmatrix}\begin{bmatrix}
\bV_{\mathrm{i},1}^{\mathrm{T}}\\\bV_{\mathrm{i},2}^{\mathrm{T}}
\end{bmatrix},\label{eq:SVD_improp}
\end{align}
where the matrix $\bSigma_{\mathrm{p}}=\mathrm{diag}(\sigma_1, \dots, 
\sigma_N)$ is a diagonal matrix with nonincreasing nonnegative entries that are 
the 
proper Hankel singular values and $\bSigma_{\mathrm{i},1}$ is a diagonal matrix 
including the improper nonzero Hankel singular values.
With the matrix $\bSigma_{\mathrm{p},1}$ which contains the $r_{\mathrm{p}}$ 
largest Hankel singular values and $\bSigma_{\mathrm{i},1} \in 
\mathbb{R}^{r_{\mathrm{i}},r_{\mathrm{i}}}$ we construct the left and right 
projection matrices
\begin{equation*}
\bW_{\mathrm{R}}:= \begin{bmatrix}
\bS_{\mathrm{p}}\bU_{\mathrm{p},1}\bSigma_{\mathrm{p},1}^{-\frac{1}{2}} & 
\bS_{\mathrm{i}}\bU_{\mathrm{i},1}\bSigma_{\mathrm{i},1}^{-\frac{1}{2}}
\end{bmatrix}
\quad \text{and}\quad \bT_{\mathrm{R}}:= \begin{bmatrix}
\bR_{\mathrm{p}}\bV_{\mathrm{p},1}\bSigma_{\mathrm{p},1}^{-\frac{1}{2}} & 
\bR_{\mathrm{i}}\bV_{\mathrm{i},1}\bSigma_{\mathrm{i},1}^{-\frac{1}{2}}
\end{bmatrix}
\end{equation*}
that balance the system, reduce the differential subsystem, and generate a 
minimal realization of the algebraic subsystem.

We obtain the reduced system \eqref{Index2red} by setting
\begin{align}\label{eq:redMat}
\begin{split}
\bE_{\mathrm{R}} :=&{} \bW_{\mathrm{R}}^{\mathrm{T}}\bE \bT_{\mathrm{R}} := 
\begin{bmatrix}
I_{r_{\mathrm{p}}}& 0 \\ 0 & \bN_{\mathrm{R}}
\end{bmatrix},\qquad \bA_{\mathrm{R}} := \bW_{\mathrm{R}}^{\mathrm{T}}\bA 
\bT_{\mathrm{R}} := \begin{bmatrix}
\bJ_{\mathrm{R}}& 0 \\ 0 & I_{r_{\mathrm{i}}}
\end{bmatrix},\\ \bB_{\mathrm{R}} :=&{}\begin{bmatrix}
\bB_{\mathrm{R},1}\\ \bB_{\mathrm{R}, 2}
\end{bmatrix} := \bW_{\mathrm{R}}^{\mathrm{T}}\bB,\qquad \bC_{\mathrm{R}} 
:= \begin{bmatrix}
\bC_{\mathrm{R},1}& \bC_{\mathrm{R}, 2}
\end{bmatrix} := \bC \bT_{\mathrm{R}}.
\end{split}
\end{align}
The generated reduced system has decoupled differential and algebraic states.

%\begin{remark}
%The projection matrices fulfill $W = \Pi^{\mathrm{T}}\bW$ and $T = \Pi^{\mathrm{T}} T$ and hence, we also have 
%\begin{align*}
%\bE_{\mathrm{r}} &:= \bW^{\mathrm{T}}ET,\quad \bA_{\mathrm{r}} := \bW^{\mathrm{T}}AT,\quad \b_{\mathrm{r}} := \bW^{\mathrm{T}}B,\quad \bC_{\mathrm{r}} := CT.
%\end{align*}
%\end{remark}

%As in \cite[Theorem 7.9]{morAnt05} and \cite[Theorem 6.4]{morBenOCetal17}, if $\sigma_{\mathrm{r}}>\sigma_{r+1}$, then the ROM is asymptotically stable and 
If $\sigma_{r_{\mathrm{p}}} >\sigma_{r_{\mathrm{p}}+1}$, then the ROM is asymptotically stable and one can estimate 
the output error of the reduced system by
\begin{align}\label{ErrBo}
	{\|\by-\by_{\mathrm{R}}\|}_{L^{2}([0,\infty),\mathbb{R}^p)} \leq 
{\|\boldsymbol{\mathcal{G}}-\boldsymbol{\mathcal{G}}_{\mathrm{R}}\|}_{\infty}{
\|\bu\|}_{L^{2}([0,\infty),\mathbb{R}^m)}
	\leq \left(2 \sum_{j=r_{\mathrm{p}}+1}^{n} \sigma_j \right) 
{\|\bu\|}_{L^{2}([0,\infty),\mathbb{R}^m)},
\end{align}
where 
$\boldsymbol{\mathcal{G}}(s) := \bC\left(s\bE-\bA\right)^{-1}\bB\in\mathcal{RH}_{\infty}^{p, m}$ and $\boldsymbol{\mathcal{G}}_{\mathrm{R}}(s) := \bC_{\mathrm{R}}\left(s\bE_{\mathrm{R}} - \bA_{\mathrm{R}}\right)^{-1}\bB_{\mathrm{R}}\in\mathcal{RH}_{\infty}^{p, m}$
are the \emph{transfer functions} of the original and the reduced system \cite[Theorem 7.9]{morAnt05}. 
The \emph{$\mathcal{H}_{\infty}$-norm} is defined as
${\|\boldsymbol{\mathcal{G}}\|}_{\infty} := \sup_{\omega\in\mathbb{R}}\sigma_{\max}(\boldsymbol{\mathcal{G}}(\mathrm{i} \omega))$ where $\sigma_{\max}(\cdot)$ denotes the maximum singular value of its matrix argument.

\begin{remark}
 In the balancing procedure outlined above, it is important that the number of inputs and outputs is very small compared to the state-space dimension, i.\,e., $m,\,p \ll N$, which we assume throughout this work. In this case, the Gramians will typically have low numerical rank~\cite{SorZ02}, which ensures that a low-rank approximation of the Gramians is possible and numerically efficient and that the FOM can be well-approximated by a ROM of low order. Balanced truncation model reduction is also possible in the case that one of the dimensions $m$ or $p$ is large while the other one is small, see, e.\,g., \cite{BenS13,PulNS21}. However, in this case, the numerical procedure is based on the approximation of an integral by quadrature. Our method cannot be generalized to this more general situation in a straightforward manner, but we believe that it would be an interesting topic for future research.  
\end{remark}

\subsubsection{Solving projected Lyapunov equations}\label{ChapterLE}
The aim of this subsection is to present numerical techniques to solve the projected Lyapunov equations \eqref{eq:propLE_contr} and \eqref{eq:impropLE_contr} in order to approximate the Gramians of system \eqref{eq:FOM}.
For standard systems with $\bE = I_N$, there are several methods to solve the corresponding Lyapunov equations.
If small-scale matrices are considered, the Bartels-Stewart algorithm \cite{BarS72} or Hammarling's method \cite{Ham82b} are used.
These methods however, are inefficient in the case of large matrices. 
In typical applications of practical relevance, the solution of a Lyapunov equation is often of very low numerical rank. So it is desired to compute low-rank factors of these solutions directly.  
State-of-the-art methods are the ADI method \cite{Pen00b,Kue16}, the sign function method \cite{morBenOCetal17} or Krylov subspace methods \cite{SimD09}.
In the literature there are also several extensions to projected Lyapunov equations such as \cite{morHeiSS08,Sty08}.

In this section, we utilize the ADI method to solve the projected continuous-time Lyapunov equation \eqref{eq:propLE_contr} and the generalized Smith method to solve the discrete-time Lyapunov equation \eqref{eq:impropLE_contr}.
Here we follow the ideas from \cite{Sty08}.

We can use the following lemma from \cite{Sty08} to derive an equation that is 
equivalent to the projected continuous-time Lyapunov equation 
\eqref{eq:propLE_contr}.
\begin{lemma}\label{lemma:SteinADI}
Let the matrix pencil $s\bE-\bA$ with $\bE,\,\bA \in\mathbb{R}^{N, N}$ be regular and asymptotically stable. Let further the matrix $\bA$ be nonsingular and $\bB \in\mathbb{R}^{N, m} $. Assume that the left and right spectral projectors onto the finite spectrum of $s\bE - \bA$ are denoted by $\bPil,\,\bPir \in\mathbb{R}^{N, N}$.
If $p \in \mathbb{C}$ with $\operatorname{Re}(p) < 0$ is not an eigenvalue of the pencil $s\bA-\bE$, then the projected discrete-time Lyapunov equation 
\begin{equation}\label{eq:Stein}
\bP_{\mathrm{p}}  = \bS(p) \bR(p) \bP_{\mathrm{p}} \bR(p)^{\mathrm{H}}\bS(p)^{\mathrm{H}} - 2\mathrm{Re}(p)\bS(p)\bPil\bB\bB^{\mathrm{T}}\bPil^{\mathrm{T}}\bS(p)^{\mathrm{H}}, \qquad \bP_{\mathrm{p}} = \bPir \bP_{\mathrm{p}} \bPir^{\mathrm{T}}
\end{equation}
with
\begin{equation}\label{eq:ADI_ST}
\bS(p):= (\bE + p \bA)^{-1}\qquad\text{and}\qquad \bR(p):= \bE - \overline{p} \bA
\end{equation}
is equivalent to the projected continuous-time Lyapunov equation \eqref{eq:propLE_contr}, i.\,e., their solution sets coincide.
\end{lemma}

Lemma~\ref{lemma:SteinADI} motivates the ADI iteration
\begin{align}\label{eq:ADIiteration}
\begin{split}
\bP_0 &:= 0, \\
\bP_k &:= \bS(p_k) \bR(p_k) \bP_{k-1} \bR(p_k)^{\mathrm{H}}\bS(p_k)^{\mathrm{H}} - 2\mathrm{Re}(p_k)\bS(p_k)\bPil\bB\bB^{\mathrm{T}}\bPil^{\mathrm{T}}\bS(p_k)^{\mathrm{H}}, \quad k=1,\,2,\,\ldots.
\end{split}
\end{align}
As shown in \cite{Sty08} the following proposition provides the convergence of the iteration \eqref{eq:ADIiteration}.

\begin{proposition}
 Let the matrix pencil $s\bE-\bA$ with $\bE,\,\bA \in\mathbb{R}^{N, N}$ be regular and \emph{asymptotically stable} (i.\,e., all its finite eigenvalues have negative real part) and let $\bB \in\mathbb{R}^{N, m}$. Assume that the left and right spectral projectors onto the finite spectrum of $s\bE - \bA$ are denoted by $\bPil,\,\bPir\in\mathbb{R}^{N, N}$.
Suppose that we have given a sequence of shift parameters ${(p_k)}_{k \ge 0}$ with $\operatorname{Re}(p_k) <0$ for all $k\ge 0$ and with $p_{k+\ell} = p_k$ for some $\ell \ge 1$ and all $k=0,\,1,\,2,\,\ldots$. Then the iteration \eqref{eq:ADIiteration} converges to the solution $\bP_{\mathrm{p}}$ of the projected Lyapunov equation \eqref{eq:propLE_contr}.
\end{proposition}

\begin{remark}
 The periodicity of the shift parameters ensures that they fulfill a \emph{non-Blaschke condition} that is sufficient for the convergence of the ADI iteration \cite{MasOR17}.
\end{remark}

To work with the ADI iteration more efficiently, our aim is to compute low-rank factors of $\bP_{\mathrm{p}}$, i.\,e., we aim to compute a tall and skinny matrix $\bZ$ such that $\bP_{\mathrm{p}}\approx\bZ\bZ^{\mathrm{H}}$.
We can represent the iteration \eqref{eq:ADIiteration} by the low-rank factors of $\bP_{k} = \bZ_k\bZ_k^{\mathrm{H}}$ with
\begin{align*}
\bZ_k &= \begin{bmatrix}
\kappa(p_k) \bS(p_k) \bPil\bB & \bS(p_k)\bR(p_k)\bZ_{k-1}
\end{bmatrix}\\
&=\big[
\kappa(p_k) \bS(p_k) \bPil\bB \quad  
\kappa(p_{k-1})\bS(p_k)\bR(p_k)\bS(p_{k-1})\bPil\bB \\ 
&\hspace*{150pt}\ldots \quad \kappa(p_1)\bS(p_k)\bR(p_k)\cdot\ldots \cdot 
\bS(p_2)\bR(p_2)\bS(p_{1})\bPil\bB
\big],
\end{align*}
where $\kappa(p_k)=\sqrt{-\mathrm{Re}(p_k)}$ and $\bZ_{0}$ is chosen as the empty matrix in $\mathbb{R}^{N , 0}$.
We note that the following properties hold for all $j,\,k = 0,\,1,\,2,\ldots$:
\begin{align}\label{eq:commut}
\begin{split}
&\bS(p_k)\bA \bS(p_j) = \bS(p_j)\bA \bS(p_k), \qquad \bR(p_k)\bA^{-1}\bR(p_j) = \bR(p_j)\bA^{-1}\bR(p_k),\\
&\bS(p_k)\bR(p_j) = \bA^{-1}\bR(p_j)\bS(p_k)\bA.
\end{split}
\end{align}
We further define
\[
\bB_0:=\kappa(p_k) \bS(p_k) \bPil\bB\qquad\text{and} \qquad \bF_j := \frac{\kappa(p_j)}{\kappa(p_{j-1})}\bS(p_j)\bR(p_{j+1}), \quad j=1,\,\ldots,\,k.
\]
Using \eqref{eq:commut}, we obtain 
\begin{align*}
\bZ_k 
&= \begin{bmatrix}
\bB_0 & \bF_{k-1}\bB_0 & \ldots & \bF_{1}\cdot\ldots\cdot \bF_{k-1}\bB_0
\end{bmatrix}.
\end{align*}

It remains to solve the discrete-time Lvapunov equation \eqref{eq:impropLE_contr}.
Under the assumption that $\bA$ is nonsingular,  \eqref{eq:impropLE_contr} is equivalent to the transformed discrete-time Lyapunov equation 
\begin{equation*}%\label{eq:trasDiscLE}
\bP_{\mathrm{i}} - \bA^{-1}\bE \bP_{\mathrm{i}}\bE^{\mathrm{T}}\bA^{-\mathrm{T}} = \bA^{-1}(I_N- \bPil) \bB\bB^{\mathrm{T}} (I_N - \bPil)^{\mathrm{T}}\bA^{-\mathrm{T}}, \qquad 0 = \bPir \bP_{\mathrm{i}} \bPir^{\mathrm{T}}.
\end{equation*}
This equation could be solved with the Smith method \cite{Sty08}. Since $\bA^{-1}(I_N-\bPil) = (I_N-\bPir)\bA^{-1}$ and the matrix $ (I_N-\bPir)\bA^{-1}\bE = \bA^{-1}\bE (I_N-\bPir)$ is nilpotent with the nilpotency index $\nu$ of the matrix $\bN$ of the quasi-Weierstrass form of the pencil $s\bE-\bA$, the iteration stops after $\nu$ steps. This leads to the unique solution 
\[
\bP_{\mathrm{i}} = \sum_{k=0}^{\nu-1} (\bA^{-1}\bE)^k  (I_N- \bPir)\bA^{-1} \bB\bB^{\mathrm{T}}\bA^{-\mathrm{T}} (I_N - \bPir)^{\mathrm{T}}  ((\bA^{-1}\bE)^{\mathrm{T}})^k.
\]
We note that we have to solve multiple linear systems with the matrix $\bA$ to 
obtain the solution $\bP_{\mathrm{i}}$. Therefore, we can utilize the sparse LU 
decomposition to do this efficiently. 
Instead of computing the full matrix $\bP_{\mathrm{i}}$ we can also generate the low-rank factors $\bP_{\mathrm{i}} = \bY\bY^{\mathrm{T}}$ as 
\[
\bY = \begin{bmatrix}
(I- \bPir)\bA^{-1} \bB & 
\bA^{-1}\bE  (I- \bPir)\bA^{-1} \bB &
\ldots &
(\bA^{-1}\bE)^{\nu-1}  (I- \bPir)\bA^{-1} \bB
\end{bmatrix}.
\]

\begin{remark}
In many relevant applications, the solution $\bP_{\mathrm{i}}$ can be explicitly calculated by making use of specific block structures in $s\bE-\bA$. Moreover, the improper Hankel singular values are often zero (namely, when the FOM's transfer function is strictly proper) in which case we do not have to solve the improper Lyapunov equations at all,
%The examples from Section~\ref{sec:examples} are of this type, 
see also \cite{morSty06}.
\end{remark}

\subsubsection{Alternative approach to determine the algebraic  
parts}\label{ssec:D0D1determ}
Since the reduced basis method presented later in this paper is applied only to
the differential part of the system, in this section we discuss ways to 
construct reduced representations of the algebraic part. One way is to use the 
reduction provided by the improper Gramians as explained in 
Subsection~\ref{ChapterBT}. This is a suitable choice for parameter-independent 
problems. 
However, it is difficult to build a reduced representation for 
parameter-dependent problems, if the improper part depends on parameters. 
Then the discrete-time Lyapunov equations have to be solved for each parameter value of interest in the online phase which can be too expensive. 
For this reason, in this subsection, we present an alternative approach that is also viable in the parametric setting and allows an efficient evaluation in the online phase (see also Subsection~\ref{ssec:treatmentalg}).   

Our method is inspired by \cite{morMosSMV22} in which the 
polynomial part of the transfer functions of DAE systems of index at most two 
is approximated, 
but the methodology can also be extended to systems of even higher index. 
Consider the transfer function $\boldsymbol{\mathcal{G}}(s)$ of system 
\eqref{eq:FOM} (for a fixed parameter) which can be decomposed into its 
strictly proper part and a polynomial part
\[
\boldsymbol{\mathcal{G}}(s) = \boldsymbol{\mathcal{G}}_{\mathrm{sp}}(s) + 
\boldsymbol{\mathcal{G}}_{\mathrm{poly}}(s),
\]
where the polynomial part is given as
\[
\boldsymbol{\mathcal{G}}_{\mathrm{poly}}(s) = -\sum_{k=0}^{\nu-1} 
 \bC\bT^{-1}\begin{bmatrix} 0 & 0 \\ 0 & \bN^k \end{bmatrix}\bW^{-1}\bB s^k,
\]
with $\bT$, $\bW$, and $\bN$ with index of nilpotency $\nu$ as in 
\eqref{eq:QWF}, see \cite{morSty04} for the details. Hence, for $\bW^{-1}\bB =: 
\begin{bsmallmatrix}
\bB_1 \\ \bB_2
\end{bsmallmatrix}$ and $\bC\bT^{-1} =: \begin{bmatrix}
\bC_1 & \bC_2
\end{bmatrix}$ partioned conforming to  the partioning of \eqref{eq:QWF}, the 
polynomial part can be rewritten as
\begin{align*}
\boldsymbol{\mathcal{G}}_{\mathrm{poly}}(s) = - \sum_{k=0}^{\nu-1} \bC_2 \bN^k 
\bB_2 s^k,
\end{align*}
where $\bM_k = -\bC_2 \bN^k \bB_2$, $k=0,\,\ldots,\,\nu-1$ denote the $k$-th 
\emph{Markov parameters} of the transfer function.
As shown in \cite{morAntGH20b}, the strictly proper part 
$\boldsymbol{\mathcal{G}}_{\mathrm{sp}}(\ri\omega)$ converges to zero 
for $\omega \to \infty$. For this reason we can approximate the Markov 
parameters if we consider sufficiently large values of $\omega\in\mathbb{R}$. 
We exemplify the methodology for a few cases.
\paragraph{The index one case}
In this case we simply have 
\begin{equation*}
 \bM_0 \approx \boldsymbol{\mathcal{G}}(\ri\omega)
\end{equation*}
for a sufficiently large value of $\omega\in\mathbb{R}$ and moreover, $\bM_k = 0$ for $k 
\ge 1$.

\paragraph{The index two case}
In this case, $\bM_k = 0$ for $k \ge 2$ and hence $\boldsymbol{\mathcal{G}}_{\mathrm{poly}}(s) = \bM_0 + \bM_1 s$.
By inserting $\ri\omega\in\ri\mathbb{R}$ we obtain that 
\[
\mathrm{Re}\left(\boldsymbol{\mathcal{G}}_{\mathrm{poly}}(\ri\omega)\right) = \bM_0,\qquad
\mathrm{Im}\left(\boldsymbol{\mathcal{G}}_{\mathrm{poly}}(\ri\omega)\right) = 
\omega\bM_1.
\]
Hence, for sufficiently large values 
$\omega\in\mathbb{R}$ we can approximate
\[
\bM_0 \approx \mathrm{Re}\left(\boldsymbol{\mathcal{G}}(\ri\omega)\right)
\qquad
\text{and}
\qquad
\bM_1 \approx\frac{\mathrm{Im}\left(\boldsymbol{\mathcal{G}}(\ri\omega)\right)}{\omega}.
\]

\paragraph{The index three case}

%Based on the approximations of $\mathcal{D}_0$ and $\mathcal{D}_1$ we can 
%determine a minimal realization of the improper parts of the reduced system, 
%i.e. we can determine $N_{\mathrm{r}},~B_{\mathrm{r},2},~C_{\mathrm{r},2}$ as 
%in Equation \eqref{eq:redMat}.
%This is an alternative approach to solving the discrete-time Lyapunov equations 
%\eqref{eq:impropLE_contr}, \eqref{eq:impropLE_obs} and generating a minimal 
%realization of the improper part of the system using the improper Gramians
In the index three case, we approximate the polynomial part of the form 
$
\boldsymbol{\mathcal{G}}_{\mathrm{poly}}(s) = \bM_0 + \bM_1 s + \bM_2 s^2.
$
Inserting $\ri\omega\in\ri\mathbb{R}$ into the polynomial part of the transfer function $\boldsymbol{\mathcal{G}}_{\mathrm{poly}}$
provides a real and an imaginary part 
\[
\mathrm{Re}\left(\boldsymbol{\mathcal{G}}_{\mathrm{poly}}(\ri\omega)\right) = \bM_0 - \omega^2\bM_2,\qquad
\mathrm{Im}\left(\boldsymbol{\mathcal{G}}_{\mathrm{poly}}(\ri\omega)\right) = \omega\bM_1.
\]
Then the Markov parameters are approximated by 
\[
\bM_2 \approx \frac{\mathrm{Re}\left(\boldsymbol{\mathcal{G}}(\ri\omega_1)-\boldsymbol{\mathcal{G}}(\ri\omega_2)\right)}{\omega_2^2 - \omega_1^2},\quad 
\bM_1 \approx \frac{\mathrm{Im}\left(\boldsymbol{\mathcal{G}}(\ri\omega)\right)}{\omega},\quad 
\bM_0 \approx \mathrm{Re}\left(\boldsymbol{\mathcal{G}}(\ri\omega)\right) + 
\omega^2 \bM_2
\]
for sufficiently large $\omega$, $\omega_1$, and $\omega_2$ with 
$\omega_1\neq\omega_2$.

\paragraph{System realization from Markov parameters}
If the polynomial part $\boldsymbol{\mathcal{G}}_{\mathrm{poly}}(s) = 
\sum_{k=0}^{\nu-1} \bM_k s^k$ has been extracted and all Markov parameters 
$\bM_k$, $k=0,\,\ldots,\,\nu-1$ have been determined we can realize it by a 
DAE system as follows. Assume that $\bM_k = \bU_k \bSigma_k \bV_k^{\mathrm{T}}$ 
are compact SVDs of $\bM_k$ with positive definite $\bSigma_k \in 
\mathbb{R}^{r_k,r_k}$ for all $k=1,\,\ldots,\,\nu-1$. Then for 
$k=1,\,\ldots,\,\nu-1$ we define the matrices
% DAE systems
%\begin{align*}
% \frac{\mathrm{d}}{\mathrm{d}t}\widehat{\bE}_k \bz_k(t) &= \widehat{\bA}_k \bz_k(t) + 
%\widehat{\bB}_k \bu(t), \\
% \widehat{\by}_k(t) &= \widehat{\bC}_k\bz_k(t)
%\end{align*}
%with 
\[
\widehat{\bB}_k := \begin{bmatrix} 0 \\ \vdots \\ 0 \\ \bV_k^{\mathrm{T}} 
 \end{bmatrix} \in \mathbb{R}^{k\cdot r_k,m}, \quad \widehat{\bC}_k := 
\begin{bmatrix} -\bU_k & 0 & \cdots & 
0\end{bmatrix} \in \mathbb{R}^{p,k\cdot r_k}
\]
and
\[
\widehat{\bE}_k := \begin{bmatrix}
0 & \bSigma_k^{\frac{1}{k}} & & \\
& \ddots & \ddots & \\ 
& & \ddots & \bSigma_k^{\frac{1}{k}}  \\
& & & 0
\end{bmatrix} \in \mathbb{R}^{k\cdot r_k,k\cdot r_k},
\quad\text{so that}\quad
\widehat{\bE}_k^k = 
\begin{bmatrix}
0 & \cdots &  0 & \bSigma_{k} \\
& \ddots &  & 0 \\ 
& & \ddots &\vdots  \\
& & & 0
\end{bmatrix}.
\]
This yields 
\begin{align*}
 \widehat{\bC}_k \left(s\widehat{\bE}_k-I_{k\cdot r_k}\right)^{-1} \widehat{\bB}_k &= 
-\sum_{j=0}^{k} \widehat{\bC}_k \widehat{\bE}_k^{j} \widehat{\bB}_k s^j 
% \\ &= 
%\begin{bmatrix} \bU_k & 0 & \ldots & 0 \end{bmatrix} \begin{bmatrix}
%  0 & & & \bSigma_k \\
%  & \ddots & & \\ 
%  & & \ddots & \\
%  & & & 0
% \end{bmatrix} \begin{bmatrix} 0 \\ \vdots \\ 0 \\ \bV_k^{\mathrm{T}} 
% \end{bmatrix} s^k
=-\widehat{\bC}_k \widehat{\bE}_k^{k} \widehat{\bB}_k s^k
 = \bM_ks^k.
\end{align*}
Now we obtain an overall realization of the algebraic part of the 
system as 
\begin{align*}
 \frac{\mathrm{d}}{\mathrm{d}t} \begin{bmatrix} \hat{\bE}_1 &  &  \\  & 
\ddots & \\ & & \hat{\bE}_{\nu-1} \end{bmatrix} \begin{bmatrix} \bz_1(t) \\ 
\vdots \\ \bz_{\nu-1}(t) \end{bmatrix} &= \begin{bmatrix} 
\bz_1(t) \\ \vdots \\ \bz_{\nu-1}(t) \end{bmatrix} + \begin{bmatrix} 
\hat{\bB}_1 \\ \vdots \\ \hat{\bB}_k \end{bmatrix} \bu(t), \\
\widehat{\by}(t) &= \begin{bmatrix} \widehat{\bC}_1 & \ldots & \hat{\bC}_k \end{bmatrix} 
\begin{bmatrix} \bz_1(t) \\ \vdots \\ \bz_{\nu-1}(t) \end{bmatrix} + \bM_0 
\bu(t).
\end{align*}

% \begin{align*}
% \frac{\mathrm{d}}{\mathrm{d}t} \begin{bmatrix}
%   0 & \bSigma_k^{1/k} & & \\
%     & \ddots & \ddots & \\ 
%     & & \ddots & \bSigma_k^{1/k}  \\
%   & & & 0
%  \end{bmatrix} \begin{bmatrix} x_{k1} \\ \vdots \\ \vdots \\ x_{kk} 
% \end{bmatrix} &= \begin{bmatrix}
%   I & & & \\
%   & \ddots & & \\
%   & & \ddots & \\
%   & & & I \\
%  \end{bmatrix} \begin{bmatrix} x_{k1} \\ \vdots \\ \vdots \\ x_{kk} 
% \end{bmatrix} + \begin{bmatrix} 0 \\ \vdots \\ 0 \\ \bV_k^{\mathrm{T}} 
% \end{bmatrix} u(t) \\
%  y_k(t) &= \begin{bmatrix} \bU & 0 & \ldots & 0\end{bmatrix} \begin{bmatrix} 
% x_{k1} \\ \vdots \\ \vdots \\ x_{kk} 
% \end{bmatrix}
% \end{align*}

\section{Reduced basis method}\label{reduced basis method}
In this section, we return to the problem of reducing parameter-dependent 
systems. 
Since the solution of the Lyapunov equations is the most expensive part of 
balanced truncation, we aim to limit the number of Lyapunov equation solves.  
To achieve this, the reduced basis method presented by Son and Stykel in 
\cite{morSonS17} can be applied. Besides the regularity and asymptotic stability of the matrix pencil 
$s\bE(\mu)-\bA(\mu)$ for all $\mu \in \mathcal{D}$ we impose the following 
assumption on the problem \eqref{eq:FOM} for the rest of the paper. 
\begin{assumption}\label{assumpt}
We assume that $\mathcal{D} \subset \mathbb{R}^d$ is a bounded and connected 
domain. We assume further that the matrix-valued functions $\bE(\cdot)$, 
$\bA(\cdot)$, $\bB(\cdot)$, $\bC(\cdot)$ in \eqref{eq:FOM} as well as the left 
and right spectral projector functions $\bPil(\cdot)$ and $\bPir(\cdot)$ on the 
deflating subspace associated with the finite eigenvalues of  
$(s\bE-\bA)(\cdot)$ depend \emph{continuously} on the parameter $\mu$.
\end{assumption}
Assumption~\ref{assumpt} ensures that the solutions of the parameter-dependent projected Lyapunov equations also depend continuously on $\mu \in \mathcal{D}$. 
Then, also the ROM depends continuously on $\mu$ for a fixed reduced order.

In practice, many examples are continuously parameter-dependent, e.\,g., in electrical circuits, the resistances can be continuously dependent on the temperature. Moreover, inductances and capacitances may be assumed continuous if one wishes to tune them within a parameter optimization.  In other examples, for instance, mechanical systems, the parameters may represent (continuously varying) material constants. On the other hand, there are also examples where the dependence on the parameters is not continuous. This is, for example, the case for electrical circuits in which the network topology changes by switching. These discontinuities are often known since, in practice, the user has insight into the considered problem. In this case, the parameter domain could be decomposed into several subdomains, on which the reduction can be performed separately.

In the following we will focus on the reduction of the differential part of the 
DAE system \eqref{eq:FOM}. 
The main idea consists of finding a reduced representation of the solutions of the Lyapunov equations \eqref{eq:propLE_contr} of the form
\begin{equation}\label{RBMb}
\bP_{\mathrm{p}}(\mu) \approx \bZ(\mu) \bZ(\mu)^{\mathrm{H}}\quad\text{with}\quad \bZ(\mu) = \bV(\mu) \tilde{\bZ}(\mu)\quad \forall\,\mu \in \mathcal{D},
\end{equation}
where $\bV(\mu) \in \mathbb{R}^{N,r_{\bV}}$ has 
orthonormal columns with $\bPir(\mu) \bV(\mu) = \bV(\mu)$ and where $\bPir(\mu)$ 
denotes the projector onto the right deflating subspace of $s\bE(\mu) - \bA(\mu)$ for 
the parameter value $\mu$. 
Under the assumption that $\bV(\mu)^{\mathrm{T}}\bE(\mu)\bV(\mu)$ is invertible and the pencil $s\bV(\mu)^{\mathrm{T}}\bE(\mu)\bV(\mu) - \bV(\mu)^{\mathrm{T}}\bA(\mu)\bV(\mu)$ is regular and asymptotically stable (This condition can be guaranteed a priori by the strict dissipativity assumptions in Subsection~\ref{subsec:strictdissip}.), the low-rank factor $\tilde{\bZ}(\mu)$ solves the reduced Lyapunov equation  
\begin{multline}\label{RBMa}
\bV(\mu)^{\mathrm{T}}\bE(\mu)\bV(\mu)\tilde{\bZ}(\mu)\tilde{\bZ}(\mu)^{\mathrm{H}}\bV(\mu)^{\mathrm{T}}\bA(\mu)^{\mathrm{T}}\bV(\mu) \\ +\bV(\mu)^{\mathrm{T}}\bA(\mu)\bV(\mu)\tilde{\bZ}(\mu)\tilde{\bZ}(\mu)^{\mathrm{H}}\bV(\mu)^{\mathrm{T}}\bE(\mu)^{\mathrm{T}}\bV(\mu) = -\bV(\mu)^{\mathrm{T}}\bPil(\mu)\bB(\mu)\bB(\mu)^{\mathrm{T}}\bPil(\mu)^{\mathrm{T}}\bV(\mu).
\end{multline}
The equation $\bPir(\mu) \bV(\mu) = \bV(\mu)$ replaces the condition 
\begin{equation*}
 \bPir(\mu)\big(\bV(\mu)\tilde{\bZ}(\mu)\big)\big(\bV(\mu)\tilde{\bZ}(\mu)\big)^{\mathrm{H}}\bPir(\mu)^{\mathrm{T}} = \big(\bV(\mu)\tilde{\bZ}(\mu)\big)\big(\bV(\mu)\tilde{\bZ}(\mu)\big)^{\mathrm{H}}.
\end{equation*} 
% For the case of a Stokes-like system \eqref{Index2}, this condition is further equivalent to the property $G(\mu)^{\mathrm{T}}\bV(\mu) = 0$, i.\,e., all columns of $V(\mu)$ lie in the kernel of $G(\mu)^{\mathrm{T}}$.

In practice, we aim to determine a matrix $\bV_{\mathrm{glob}} \in 
\mathbb{R}^{N,r_{\mathrm{glob}}}$ with $N \gg r_{\mathrm{glob}}$ that contains 
the information of $\bV(\mu)$ for all $\mu \in \mathcal{D}$ globally. 
When we have determined such a $\bV_{\mathrm{glob}}$, we set
\begin{equation} \label{eq:boldV}
 \bV(\mu) = \operatorname{orth}\big(\bPir(\mu)\bV_{\mathrm{glob}}\big),
\end{equation}
where $\operatorname{orth}(\bPir(\mu)\bV_{\mathrm{glob}})$ denotes the orthonormalized columns of $\bPir(\mu)\bV_{\mathrm{glob}}$.
If the matrix $\bV_{\mathrm{glob}}$ has been found, then the ROM for a particular parameter $\mu \in \mathcal{D}$ can be determined very efficiently in the so-called \emph{online phase}, where one simply solves \eqref{RBMa} with \eqref{RBMb} and \eqref{eq:boldV} to determine the two Gramians and then projects the system as in Subsection~\ref{ChapterBT}.

The computation of the \emph{reduced basis} $\bV_{\mathrm{glob}}$ is done beforehand in the \emph{offline phase}. 
There, we solve the full-order Lyapunov equation only for those parameters, where their solutions are currently approximated worst to enrich the current reduced basis. A posteriori error estimators are employed on a test parameter set to find these points efficiently. 

We determine $\bV_{\mathrm{glob}}$ by considering a test parameter set 
$\mathcal{D}_{\mathrm{Test}}\subset \mathcal{D}$.
We begin by computing a low-rank factor $\bZ(\mu_{1})$ in $\mu_{1}\in\mathcal{D}_{\mathrm{Test}}$ such that $\bP_{\mathrm{p}}(\mu_{1})=\bZ(\mu_{1})\bZ(\mu_{1})^{\mathrm{H}}$ solves the projected Lyapunov equation
\begin{multline} \label{eq:projL}
\bE(\mu)\bP_{\mathrm{p}}(\mu)\bA(\mu)^{\mathrm{T}} + \bA(\mu)\bP_{\mathrm{p}}(\mu)\bE(\mu)^{\mathrm{T}} = - \bPil(\mu)\bB(\mu)\bB(\mu)^{\mathrm{T}}\bPil(\mu)^{\mathrm{T}},\\ \bPir(\mu) \bP_{\mathrm{p}}(\mu)\bPir(\mu)^{\mathrm{T}} = \bP_{\mathrm{p}}(\mu)
\end{multline}
for $\mu = \mu_1$.

The first reduced basis is then given by $\bV_{\mathrm{glob}} := \operatorname{orth}(\bZ(\mu_{1}))$. Next we determine the test parameter $\mu_{2}\in\mathcal{D}_{\mathrm{Test}}$ for which the solution $\bP_{\mathrm{p}}(\mu_{2})$ of the Lyapunov equation \eqref{eq:projL} is approximated worst by using \eqref{RBMb}, \eqref{RBMa}, and \eqref{eq:boldV}.
For that, one of the error estimators $\Delta_{\bV}(\mu)$ presented in Subsection~\ref{ErrEst} is utilized.
For this parameter $\mu_{2}$, we solve the projected Lyapunov equation \eqref{eq:projL} with $\mu = \mu_2$ to obtain the low-rank factor $\bZ(\mu_{2})$. Then we define the new reduced basis by setting $\bV_{\mathrm{glob}} := \operatorname{orth}([\bV_{\mathrm{glob}},\,  \bZ(\mu_{2})])$.
This procedure is repeated until the error estimator $\Delta_{\bV}$ is smaller than a prescribed tolerance for every test parameter in $\mathcal{D}_{\mathrm{Test}}$.
This method results in Algorithm~\ref{AlgoRBMOff}.
\begin{algorithm}[bt]
	\caption{Reduced basis method for projected Lyapunov equations (offline phase)}
	\label{AlgoRBMOff}
		\begin{algorithmic}[1] 
		\Require Dynamical system \eqref{eq:FOM}, test parameter set $\mathcal{D}_{\mathrm{Test}}$, tolerance $\mathrm{Tol}$.
	      \Ensure Reduced basis matrix $\bV_{\mathrm{glob}}$
			\State Choose any $\mu_{1}\in\mathcal{D}_{\mathrm{Test}}$.
			\State Solve the projected Lyapunov equation \eqref{eq:projL} for $\mu = \mu_{1}$ to obtain $\bZ(\mu_{1})$.\label{StepSol1}
			\State Set $\mathcal{M}:=\{\mu_{1}\}$.
			\State Set $\bV_{\mathrm{glob}} := \operatorname{orth}(\bZ(\mu_{1}))$.\label{StepONB1}
			\State Set $\widehat{\mu} := \operatorname{arg\,max}_{\mu\in\mathcal{D}_{\mathrm{Test}}\setminus\mathcal{M}}\Delta_{\bV}(\mu)$.\label{Sr5}
			\State Set $\Delta_{\bV}^{\max}:=\Delta_{\bV}(\widehat{\mu})$.
			\While{$\Delta_{\bV}^{\max}>\mathrm{Tol}$}
			\State Solve the projected Lyapunov equation \eqref{eq:projL} for $\mu = \widehat{\mu}$ to obtain $\bZ(\widehat{\mu})$.\label{StepSol2}
			\State Set $\mathcal{M} := \mathcal{M}\cup\{\widehat{\mu}\}$.
			\State Set $\bV_{\mathrm{glob}} := \operatorname{orth}([\bV_{\mathrm{glob}},\,\bZ(\widehat{\mu})])$.\label{StepONB2}
			\State Set $\widehat{\mu}:= \operatorname{arg\,max}_{\mu\in\mathcal{D}_{\mathrm{Test}}\setminus\mathcal{M}}\Delta_{\bV}(\mu)$.\label{S11}
			\State Set $\Delta_{\bV}^{\max}:=\Delta_{\bV}(\widehat{\mu})$.
			\EndWhile
		\end{algorithmic}
\end{algorithm}

\begin{remark}
 Under the assumption that the projected Lyapunov equations \eqref{eq:propLE_contr},\,\eqref{eq:propLE_observ}  are solved \emph{exactly} for all parameters $\mu \in \mathcal{D}$ using the reduced Lyapunov equations of the form \eqref{RBMa}, the $\mathcal{H}_\infty$ error bound \eqref{ErrBo} is also valid in the entire parameter domain. Unfortunately, this bound is no longer valid if numerical approximation errors affect the numerically computed Gramians which is already the case in the parameter-independent case. A general and in-depth analysis of the effect of such errors on the Hankel singular values is still an open problem. On the other hand, the numerically computed Hankel singular values can still be used to obtain an \emph{approximate} error bound.
\end{remark}

\subsection{Error estimation}\label{ErrEst}

We want to estimate the error $\mathfrak{E}(\mu):= \bP_{\mathrm{p}}(\mu)-\bZ(\mu)\bZ(\mu)^{\mathrm{H}}$ with the help of the residual
\begin{equation*}
\mathfrak{R}(\mu):=\bA(\mu)\bZ(\mu)\bZ(\mu)^{\mathrm{H}}\bE(\mu)^{\mathrm{T}}+\bE(\mu)\bZ(\mu)\bZ(\mu)^{\mathrm{H}}\bA(\mu)^{\mathrm{T}}+\bPil(\mu)\bB(\mu)\bB(\mu)^{\mathrm{T}}\bPil(\mu)^{\mathrm{T}}.
\end{equation*}
\begin{lemma}\label{lemma:errEq}
Let $\mu\in\mathcal{D}$ be given. Let the matrix pencil $s\bE(\mu)-\bA(\mu)$ with $\bE(\mu),\,\bA(\mu) \in\mathbb{R}^{N, N}$ be regular and asymptotically stable. Assume that the left and right spectral projectors onto the finite spectrum of $s\bE(\mu) - \bA(\mu)$ are denoted by $\bPil(\mu),\,\bPir(\mu) \in\mathbb{R}^{N, N}$.
Let the controllability Gramian $\bP_{\mathrm{p}}(\mu) \in\mathbb{R}^{N, N}$ be defined as in \eqref{eq:contrGram} and let it be approximated as described in equation \eqref{RBMb}.
Then the error $\mathfrak{E}(\mu):= \bP_{\mathrm{p}}(\mu)-\bZ(\mu)\bZ(\mu)^{\mathrm{H}}$ solves the projected error equation
\begin{align}\label{eq:ErrEq}
\begin{split}
\bA(\mu)\mathfrak{E}(\mu)\bE(\mu)^{\mathrm{T}}+\bE(\mu)\mathfrak{E}(\mu)\bA(\mu)^{\mathrm{T}}&= - \mathfrak{R}(\mu) \\ 
                          &= - \bPil(\mu)\mathfrak{R}(\mu)\bPil(\mu)^{\mathrm{T}},\qquad \mathfrak{E}(\mu) = \bPir(\mu) \mathfrak{E}(\mu) \bPir(\mu)^{\mathrm{T}}.
\end{split}                          
\end{align}
\end{lemma}
\begin{proof}
The condition 
\[
\mathfrak{E}(\mu) = \bPir(\mu) \mathfrak{E}(\mu) \bPir(\mu)^{\mathrm{T}}
\]
follows from $\bP_{\mathrm{p}}(\mu) = \bPir(\mu) \bP_{\mathrm{p}}(\mu) \bPir(\mu)^{\mathrm{T}}$ and $\bZ(\mu) = \bPir(\mu) \bZ(\mu)$.
For the remaining equation we consider its left-hand side
\begin{align*}
&\bA(\mu)\mathfrak{E}(\mu)\bE(\mu)^{\mathrm{T}}+\bE(\mu)\mathfrak{E}(\mu)\bA(\mu)^{\mathrm{T}}\\
&= \bA(\mu)\bP_{\mathrm{p}}(\mu)\bE(\mu)^{\mathrm{T}}+\bE(\mu)\bP_{\mathrm{p}}(\mu)\bA(\mu)^{\mathrm{T}} - \bA(\mu)\bZ(\mu)\bZ(\mu)^{\mathrm{H}}\bE(\mu)^{\mathrm{T}} - \bE(\mu)\bZ(\mu)\bZ(\mu)^{\mathrm{H}}\bA(\mu)^{\mathrm{T}}\\
&= -\bPil(\mu)\bB(\mu)\bB(\mu)^{\mathrm{T}}\bPil(\mu)^{\mathrm{T}} - \bA(\mu)\bZ(\mu)\bZ(\mu)^{\mathrm{H}}\bE(\mu)^{\mathrm{T}} - \bE(\mu)\bZ(\mu)\bZ(\mu)^{\mathrm{H}}\bA(\mu)^{\mathrm{T}}.
\end{align*}
With $\bPil(\mu) \bA(\mu) = \bA(\mu)\bPir(\mu)$, $\bPil(\mu) \bE(\mu) = \bE(\mu)\bPir(\mu)$ and $\bZ(\mu) = \bPir(\mu)\bZ(\mu)$ we obtain 
\begin{align*}
&\bA(\mu)\mathfrak{E}(\mu)\bE(\mu)^{\mathrm{T}}+\bE(\mu)\mathfrak{E}(\mu)\bA(\mu)^{\mathrm{T}}\\
&= -\bPil(\mu)\bB(\mu)\bB(\mu)^{\mathrm{T}}\bPil(\mu)^{\mathrm{T}} - \bPil(\mu)\bA(\mu)\bZ(\mu)\bZ(\mu)^{\mathrm{H}}\bE(\mu)^{\mathrm{T}}\bPil(\mu)^{\mathrm{T}} \\
&\hspace*{240pt}- \bPil(\mu)\bE(\mu)\bZ(\mu)\bZ(\mu)^{\mathrm{H}}\bA(\mu)^{\mathrm{T}}\bPil(\mu)^{\mathrm{T}}
\end{align*}
which proves the statement.
\end{proof}
Similarly to \cite{morSonS17}, we consider the linear system
\begin{equation*}
	\bcL(\mu)\bx(\mu) = \bb(\mu), \quad \bx(\mu)=\bPir^{\otimes}(\mu)\bx(\mu)
\end{equation*}
with
\begin{align}\label{eq:defL}
\begin{split}
\bcL(\mu) &:=-\bA(\mu)\otimes \bE(\mu)-\bE(\mu)\otimes \bA(\mu),  \\
\bx(\mu) &:= \mathrm{vec}\left(\bP_{\mathrm{p}}(\mu)\right),\quad \bb(\mu) :=\mathrm{vec}\left(\bPil(\mu)\bB\bB^{\mathrm{T}}\bPil(\mu)^{\mathrm{T}}\right),\quad \bPir^{\otimes}(\mu) := \bPir(\mu)\otimes \bPir(\mu),
\end{split}
\end{align}
where the operator $\otimes$ denotes the \emph{Kronecker product} and 
$\mathrm{vec}$ the \emph{vectorization operator} that stacks the columns of 
the matrix on top of one another. This linear system is equivalent to the 
projected Lyapunov equation \eqref{eq:projL}. 
We note that $\bcL(\mu)$ is a singular matrix and that the unique solvability is enforced by the additional restriction given by the projection with $\bPir^{\otimes}(\mu)$.

Using the linear system representation and the abbreviations $\be(\mu):= \vect(\mathfrak{E}(\mu))$, $\br(\mu) := \vect(\mathfrak{R}(\mu))$, 
% , the fact that $\bPsi^{\otimes}_{\mathrm{l}}(\mu)$ and $\bPhi^{\otimes}_{\mathrm{r}}(\mu)$ have orthonormal columns 
%and the identity $\mathrm{vec}(AXB^{\mathrm{T}}) = ( B\otimes A)\mathrm{vec}(X)$ for matrices $A,\,B,\,X$ of conforming dimension, %for $\widehat{\bx}(\mu) := \mathrm{vec} \left( \bZ(\mu)\bZ(\mu)^{\mathrm{H}}\right)$ we conclude 
we can rewrite the error equation \eqref{eq:ErrEq} as
\begin{equation*}
	\bcL(\mu)\be(\mu) = \br(\mu), \quad \be(\mu)=\bPir^{\otimes}(\mu)\be(\mu). 
\end{equation*}
This equation is equivalent to
\begin{equation} \label{eq:defPil}
	\bPil^{\otimes}(\mu)\bcL(\mu)\bPir^{\otimes}(\mu)\be(\mu) = \bPil^{\otimes}(\mu)\br(\mu) = \br(\mu), \quad \text{where} \quad \bPil^{\otimes}(\mu) := \bPil(\mu)\otimes \bPil(\mu).
\end{equation}
Due to the unique solvability of the latter, we obtain
\begin{equation*}
	\be(\mu) = \left( \bPil^{\otimes}(\mu)\bcL(\mu)\bPir^{\otimes}(\mu) \right)^\dagger \br(\mu), 
\end{equation*}
where $(\cdot)^\dagger$ denotes the Moore-Penrose inverse. Then we can estimate
\begin{align}\label{eq:err1_part1}
\begin{split}
 \left\|\mathfrak{E}(\mu)\right\|_{\mathrm{F}} &= \left\|\be(\mu)\right\|_{2} \\
 &\le \big\|\left(\bPil^{\otimes}(\mu)\bcL(\mu)\bPir^{\otimes}(\mu) \right)^\dagger\big\|_2 \left\|\br(\mu)\right\|_2 \\
 &=\frac{1}{\sigma_{\min}^{+}\left(\bPil^{\otimes}(\mu)\bcL(\mu)\bPir^{\otimes}(\mu) \right)} \left\|\mathfrak{R}(\mu)\right\|_{\mathrm{F}} \le \frac{1}{\alpha(\mu)} \left\|\mathfrak{R}(\mu)\right\|_{\mathrm{F}} =: \Delta_\bV^{(1)}(\mu).
\end{split}
\end{align}

Here, $\sigma_{\min}^{+}(\cdot)$ denotes the smallest positive singular value of its matrix argument and $\alpha(\mu)$ is an easily computable lower bound of $\sigma_{\min}^{+}\left(\bPil^{\otimes}(\mu)\bcL(\mu)\bPir^{\otimes}(\mu)\right)$.
The bound $\Delta_\bV^{(1)}(\mu)$ is the first possible error estimator that provides a rather conservative bound in practice.

To improve the estimator \eqref{eq:err1_part1}, we consider the error bound presented in \cite{Schmidt2018}.
Assume that $\widehat{\mathfrak{E}}(\mu)\approx \mathfrak{E}(\mu)$ is an error 
estimate that is not necessarily an upper bound. 
An error bound $\Delta_{\bV}^{(2)}(\mu)$ is derived by
\begin{align}\label{ErrEstF}
\begin{split}
	{\big\|\mathfrak{E}(\mu)\big\|}_{\mathrm{F}} &= {\big\|\mathfrak{E}(\mu)+\widehat{\mathfrak{E}}(\mu)-\widehat{\mathfrak{E}}(\mu)\big\|}_{\mathrm{F}} 
	\leq {\big\|\widehat{\mathfrak{E}}(\mu)\big\|}_{\mathrm{F}}+{\big\|\mathfrak{E}(\mu)-\widehat{\mathfrak{E}}(\mu)\big\|}_{\mathrm{F}}\\ &\leq \big\|\widehat{\mathfrak{E}}(\mu)\big\|_{\mathrm{F}} +\frac{1}{\alpha(\mu)} {\big\|\widehat{\mathfrak{R}}(\mu)\big\|}_{\mathrm{F}}:=\Delta_{\bV}^{(2)}(\mu),
\end{split}
\end{align}
where the residual $\widehat{\mathfrak{R}}(\mu)$ is defined as 
\begin{multline*}
\widehat{\mathfrak{R}}(\mu):= \bA(\mu)\left(\bZ(\mu)\bZ(\mu)^{\mathrm{H}}+\widehat{\mathfrak{E}}(\mu)\right)\bE(\mu)^{\mathrm{T}}\\\ +\bE(\mu)\left(\bZ(\mu)\bZ(\mu)^{\mathrm{H}}+\widehat{\mathfrak{E}}(\mu)\right)\bA(\mu)^{\mathrm{T}}+\bPil(\mu)\bB(\mu)\bB(\mu)^{\mathrm{T}}\bPil(\mu)^{\mathrm{T}}.
\end{multline*}

It remains to obtain an error estimate $\widehat{\mathfrak{E}}(\mu)$ for all $\mu \in \mathcal{D}_{\mathrm{Test}}$. To do so, we solve the error equation \eqref{eq:ErrEq} approximately by modifying the projected ADI iteration from Section~\ref{ChapterLE} and the reduced basis method in Algorithm~\ref{AlgoRBMOff}.

The right-hand side of the error equation \eqref{eq:ErrEq} can be written as the product $\bB_{\mathrm{l}}(\mu)\bB_{\mathrm{r}}(\mu)^{\mathrm{H}}$, where
\begin{align*}
\bB_{\mathrm{l}}(\mu) &:= \bPil(\mu)\begin{bmatrix}
\bA(\mu)\bZ(\mu)&
\bE(\mu)\bZ(\mu)&
\bB(\mu)
\end{bmatrix},\\
\bB_{\mathrm{r}}(\mu)&:= \bPil(\mu)\begin{bmatrix}
\bE(\mu)\bZ(\mu)&
\bA(\mu)\bZ(\mu)&
\bB(\mu)
\end{bmatrix}.
\end{align*}

For simplicity of notation we consider the modified ADI iteration for the parameter-independent case such that $\mathfrak{E}(\mu) \equiv  \mathfrak{E}$, $\bB_{\mathrm{l}}(\mu)\equiv \bB_{\mathrm{l}}$, $\bB_{\mathrm{r}}(\mu)\equiv \bB_{\mathrm{r}}$, and $\bPil(\mu) \equiv \bPil$.
Analogously to the derivation in Section~\ref{ChapterLE} (and under the same 
conditions), 
it can be shown that the iteration 
\begin{align*}
\mathfrak{E}_k =\bZ_{\mathrm{l}, k}(\bZ_{\mathrm{r},k})^{\mathrm{H}}&= \bS(p_k) \bR(p_k)\mathfrak{E}_{k-1}\bR(p_k)^{\mathrm{H}}\bS(p_k)^{\mathrm{H}}  -2\mathrm{Re}(p_k)\bS(p_k)\bB_{\mathrm{l}}\bB_{\mathrm{r}}^{\mathrm{T}}\bS(p_k)^{\mathrm{H}}\\
&= \bS(p_k) \bR(p_k)\bZ_{\mathrm{l}, k-1}(\bZ_{\mathrm{r},k-1})^{\mathrm{H}}\bR(p_k)^{\mathrm{H}} \bS(p_k)^{\mathrm{H}} -2\mathrm{Re}(p_k)\bS(p_k)\bB_{\mathrm{l}}\bB_{\mathrm{r}}^{\mathrm{T}}\bS(p_k)^{\mathrm{H}}
\end{align*}
converges to the solution $\mathfrak{E}$ of \eqref{eq:ErrEq}, where $\bS(p_k)$ and $\bR(p_k)$ are defined as in \eqref{eq:ADI_ST} for a shift parameter $p_k\in\mathbb{C}$ with $\operatorname{Re}(p_k) < 0$.
The factors $\bZ_{\mathrm{l},k}$ and $\bZ_{\mathrm{r},k}$ can be written 
as
\begin{align*}
	\bZ_{\mathrm{l},k} &= \begin{bmatrix}
	\bB_{\mathrm{l}, 0} & \bF_{k-1}\bB_{\mathrm{l}, 0} & \ldots & 
\bF_{1}\cdot\ldots\cdot \bF_{k-1}\bB_{\mathrm{l}, 0}
	\end{bmatrix},\\
	\bZ_{\mathrm{r}, k} &= \begin{bmatrix}
	\bB_{\mathrm{r}, 0} & \bF_{k-1}\bB_{\mathrm{r}, 0} & \ldots & 
\bF_{1}\cdot\ldots \cdot\bF_{k-1}\bB_{\mathrm{r}, 0}
	\end{bmatrix},
\end{align*}
where 
\[
\bB_{\mathrm{l}, 0} := \kappa(p_k)\bS(p_k)\bPil\bB_{\mathrm{l}},\quad
\bB_{\mathrm{r}, 0} := \kappa(p_k)\bS(p_k)\bPil\bB_{\mathrm{r}}\quad \text{ and }\quad \bF_j := \frac{\kappa(p_j)}{\kappa(p_{j-1})}\bS(p_j)\bR(p_{j+1})
\]
for $j=1, \dots, k$ and with $\kappa(p_j) = \sqrt{-2\mathrm{Re}(p_j)}$.

%Consequently, one can modify the projected ADI iteration from Section~\ref{ChapterLE} in such a way, that we multiply in every iteration step the last columns of the current factors $\bZ_{\mathrm{l}, k},\, \bZ_{\mathrm{r}, k}$ by $\bF_{k+1}$ and concatenate the product with the current factors to obtain $\bZ_{\mathrm{l}, k+1}$ and $\bZ_{\mathrm{r}, k+1}$.

We include the modified projected ADI iteration into the reduced basis method presented in Algorithm \ref{AlgoRBMOff} by solving the error equation \eqref{eq:ErrEq} in Step \ref{StepSol1} and \ref{StepSol2} to obtain the the two factors $\bZ_{\mathrm{l}}(\mu),\; \bZ_{\mathrm{r}}(\mu)$ with $\bZ_{\mathrm{l}}(\mu)\bZ_{\mathrm{r}}(\mu)^{\mathrm{H}}\approx \mathfrak{E}(\mu)$.
The orthonormal basis computation in Step \ref{StepONB1} and \ref{StepONB2} is replaced by $\bV_{\mathrm{glob}} := \operatorname{orth}([\bV_{\mathrm{glob}}, \, \bZ_{\mathrm{l}}(\mu)])$.
Note that here, the columns of $\bZ_{\mathrm{l}}(\mu)$ and $\bZ_{\mathrm{r}}(\mu)$ span the same subspace. As error estimator we use $\Delta_\bV^{(1)}$.
After we have determined the orthonormal basis, we solve equation \eqref{eq:ErrEq} on the corresponding subspace to get an approximate error $\widehat{\mathfrak{E}}(\mu)$, that we use for the error estimator $\Delta_\bV^{(2)}(\mu)$.

%\color{blue}
\begin{remark}
If the factor $\bZ(\mu)\in\mathbb{R}^{N, n_{\mathbf{Z}}}$ is already of large dimension $n_{\bZ}$, the right-hand side of the error equation \eqref{eq:ErrEq} described by $\bB_{\mathrm{l}}(\mu)$ and $\bB_{\mathrm{r}}(\mu)$ is of even larger rank $2n_{\bZ}+m$, and hence, the solution $\mathfrak{E}(\mu)$ cannot in general be well-represented by low-rank factors. This means that the error space basis might be large which leads to a time consuming error estimation.
\end{remark}
%\color{black}

\subsection{Strictly dissipative systems}\label{subsec:strictdissip}
For the model examples from Section \ref{sec:examples} we can make use of their structure to derive the bound $\alpha(\mu)$ in a computationally advantageous way. 
Therefore, we will impose an additional assumption on the matrix pencils 
$sE(\mu)-A(\mu)$ in \eqref{Index2} and \eqref{eq:EAIndex3} that is formed from 
submatrices of $\bE(\mu)$ and $\bA(\mu)$.

\begin{definition}\label{def:StrDef}
	The set $\{sE(\mu) - A(\mu) \in \mathbb{R}[s]^{n,n} \; | \; \mu \in \mathcal{D} \}$ of matrix pencils is called \emph{uniformly strictly dissipative}, if 
	\begin{equation*}
	E(\mu)=E(\mu)^{\mathrm{T}}>0 \quad \text{ and } \quad A^S(\mu) :=\frac{1}{2}\big({A(\mu)+A(\mu)^{\mathrm{T}}}\big)<0 \quad \forall\,\mu \in \mathcal{D}.
	\end{equation*}
\end{definition}
We will consider this condition for our motivating examples from Section~\ref{sec:examples} and determine a corresponding lower bound $\alpha(\mu)$. %and upper bound $\beta(\mu)$.
% On the other hand, the pencil set $\{sE(\mu) - A(\mu) \; | \; \mu \in \mathcal{D} \}$ in \eqref{eq:EAIndex3}  is not uniformly strictly dissipative. 

\subsubsection{Stokes-like systems}
\label{ssec:dissip}

We consider the Stokes-like system \eqref{Index2}. Here, the pencil set $\{sE(\mu) - A(\mu) \; | \; \mu \in \mathcal{D} \}$ is naturally uniformly strictly dissipative.

As shown in \cite{morSty06} the projection matrices are given as 
\begin{equation}\label{eq:defbigPi}
\bPil(\mu)  = \bPir(\mu) ^{\mathrm{T}} = \begin{bmatrix}
\Pi(\mu)  & -\Pi(\mu)  A(\mu)  E(\mu) ^{-1} G(\mu)  (G(\mu)^{\mathrm{T}}E(\mu) ^{-1}G(\mu) )^{-1} \\
0 & 0 
\end{bmatrix}
\end{equation}
where 
\begin{equation}\label{eq:defPi}
\Pi(\mu) = I_n - G(\mu)  (G(\mu)^{\mathrm{T}}E(\mu) ^{-1}G(\mu) )^{-1}G(\mu)^{\mathrm{T}}E(\mu) ^{-1}.
\end{equation}
% We choose a factorization
% \begin{equation}\label{eq:defPi}
% \Pi(\mu) = \Xi_{\mathrm{l}}(\mu) \Xi_{\mathrm{r}}(\mu)^{\mathrm{T}}\quad\text{ with }\quad \Xi_{\mathrm{r}}(\mu)^{\mathrm{T}}\Xi_{\mathrm{l}}(\mu)=I_{n_\mathrm{f}}.
% \end{equation} 
%where $\Xi_{\mathrm{l}}(\mu)$, $\Xi_{\mathrm{r}}(\mu)$ are matrices with orthonormal columns. 
%We obtain such a factorization if using, e.\,g., a singular value decomposition of $\Pi(\mu)$.
% Based on $\Xi_{\mathrm{l}}(\mu)$ and $\Xi_{\mathrm{r}}(\mu)$ we obtain $\bPil(\mu) = \bPhi_{\mathrm{l}}(\mu)\bPhi_{\mathrm{r}}(\mu)^{\mathrm{T}}$ with
% \begin{equation}\label{eq:facPhi}
% \bPhi_{\mathrm{l}}(\mu) := \begin{bmatrix}
% \Xi_{\mathrm{l}}(\mu) \\ 0
% \end{bmatrix}, \qquad \bPhi_{\mathrm{r}}(\mu) := \begin{bmatrix}\Xi_{\mathrm{r}}(\mu) \\  \widetilde{\Xi}_{\mathrm{r}}(\mu)
% \end{bmatrix} := \begin{bmatrix}\Xi_{\mathrm{r}}(\mu) \\  (G(\mu)^{\mathrm{T}}E(\mu) ^{-1}G(\mu) )^{-1}G(\mu)^{\mathrm{T}}E(\mu)^{-1}A(\mu)\Xi_{\mathrm{r}}(\mu)
% \end{bmatrix}.
% \end{equation}

Even though for the Stokes-like system in \eqref{Index2} we have $A^S(\mu) = A(\mu) < 0$ for all $\mu \in \mathcal{D}$, in this section we will already treat the more general case of a uniformly strictly dissipative set of matrix pencils $sE(\mu) - A(\mu)$ with \emph{nonsymmetric} $A(\mu)$ as this case will arise later in Subsubsection~\ref{sssec:sdIndex3}. 

We continue finding a value for $\alpha(\mu)$. 

\begin{theorem}\label{thm:estimalpha}
Let $\mu\in\mathcal{D}$ be given and consider the Stokes-like system \eqref{Index2}. Let $A^{S}(\mu)$, $\bcL(\mu)$, $\bPir^\otimes(\mu)$, $\bPil^\otimes(\mu)$ be defined as in Definition~\ref{def:StrDef}, \eqref{eq:defL}, \eqref{eq:defPil} with $\bPil(\mu) = \bPir(\mu)^{\mathrm{T}}$ as in \eqref{eq:defbigPi}, respectively. Then the bound
\[
\sigma_{\min}^+\left(\bPil^{\otimes}(\mu)\bcL(\mu)\bPir^{\otimes}(\mu)\right) \geq 
2\lambda_{\min}(E(\mu))\lambda_{\min}\big(-A^S(\mu)\big)
\]
is satisfied, where $\lambda_{\min}(\cdot)$ denotes the minimum eigenvalue of its (symmetric) matrix arguments.
\end{theorem}
\begin{proof}
%Lemma \ref{lemma_EAPhiPsi} implies 
%\[
%\bcL_{\Phi, \Psi}(\mu) = \left(\Xi_{\mathrm{r}}(\mu)^{\mathrm{T}}\otimes\Xi_{\mathrm{r}}(\mu)^{\mathrm{T}}\right)L(\mu) \left(\Xi_{\mathrm{r}}(\mu)\otimes\Xi_{\mathrm{r}}(\mu)\right),
%\]
%where $\Xi^{\otimes}_{\mathrm{r}}(\mu):=\Xi_{\mathrm{r}}(\mu)\otimes\Xi_{\mathrm{r}}(\mu)$.

Let $\Pi(\mu)$ be as in \eqref{eq:defPi}. Making use of the condition $\Pi(\mu)G(\mu) = 0$, we can deduce
\begin{align*}
 \bPil^\otimes(\mu) \bcL(\mu) \bPir^\otimes(\mu) &= -\left(\bPil(\mu) \begin{bmatrix} A(\mu) & G(\mu) \\ G(\mu)^{\mathrm{T}} & 0 \end{bmatrix} \bPir(\mu)\right) \otimes \left(\bPil(\mu) \begin{bmatrix} E(\mu) & 0 \\ 0 & 0 \end{bmatrix} \bPir(\mu)\right) \\ 
 &\hspace*{50pt}  -\left(\bPil(\mu) \begin{bmatrix} E(\mu) & 0 \\ 0 & 0 \end{bmatrix} \bPir(\mu)\right) \otimes \left(\bPil(\mu) \begin{bmatrix} A(\mu) & G(\mu) \\ G(\mu)^{\mathrm{T}} & 0 \end{bmatrix} \bPir(\mu)\right) \\
 &= -\begin{bmatrix} \Pi(\mu) A(\mu) \Pi(\mu)^{\mathrm{T}} & 0 \\ 0 & 0 \end{bmatrix} \otimes \begin{bmatrix} \Pi(\mu) E(\mu) \Pi(\mu)^{\mathrm{T}} & 0 \\ 0 & 0 \end{bmatrix} \\ 
 &\hspace*{50pt} - \begin{bmatrix} \Pi(\mu) E(\mu) \Pi(\mu)^{\mathrm{T}} & 0 \\ 0 & 0 \end{bmatrix} \otimes \begin{bmatrix} \Pi(\mu) A(\mu) \Pi(\mu)^{\mathrm{T}} & 0 \\ 0 & 0 \end{bmatrix} 
\end{align*}

Define $\Pi^\otimes(\mu) := \Pi(\mu) \otimes \Pi(\mu)$ and the matrix 
\[
L(\mu):=-A(\mu)\otimes E(\mu)-E(\mu)\otimes A(\mu).
\]
Denoting the smallest positive eigenvalue of a symmetric matrix by $\lambda_{\min}^+(\cdot)$, Corollary~3.1.5 from \cite{HorJ91} and the eigenvalue properties of the Kronecker product, given in \cite{HorJ91, LanT85} (see also \cite{morSonS17}) as well as a Rayleigh quotient argument lead to the estimate
\begin{align*}
\sigma_{\min}^+\left(\bPil^{\otimes}(\mu)\bcL(\mu)\bPir^{\otimes}(\mu)\right)
%&=\sigma_{\min}^+\left(\bPhi^{\otimes}_{\mathrm{l}}(\mu)\Xi^{\otimes}_{\mathrm{r}}(\mu)^{\mathrm{T}}L(\mu)\Xi^{\otimes}_{\mathrm{r}}(\mu)\bPsi^{\otimes}_{\mathrm{r}}(\mu)^{\mathrm{T}}\right) \\
&=\sigma_{\min}^+\left(\Pi^{\otimes}(\mu)L(\mu)\Pi^{\otimes}(\mu)^{\mathrm{T}}\right) \\
&\geq \lambda_{\min}^+\left(\frac{1}{2}\Pi^{\otimes}(\mu)\left(L(\mu)+L(\mu)^{\mathrm{T}}\right)\Pi^{\otimes}(\mu)^\mathrm{T}\right)\\
&= \min_{v \in \im\Pi^{\otimes}(\mu)^{\mathrm{T}},\,{\|v\|}_2=1}\frac{1}{2}v^{\mathrm{T}}\Pi^{\otimes}(\mu)\left(L(\mu)+L(\mu)^{\mathrm{T}}\right)\Pi^{\otimes}(\mu)^\mathrm{T}v \\
&= \min_{v \in \im\Pi^{\otimes}(\mu)^\mathrm{T},\,{\|v\|}_2=1}\frac{1}{2}v^{\mathrm{T}}\left(L(\mu)+L(\mu)^{\mathrm{T}}\right)v \\
&\ge \min_{v \in \mathbb{R}^{n^2},\,{\|v\|}_2=1}\frac{1}{2}v^{\mathrm{T}}\left(L(\mu)+L(\mu)^{\mathrm{T}}\right)v \\
&= \lambda_{\min}\left(\frac{1}{2}\left(L(\mu)+L(\mu)^{\mathrm{T}}\right)\right)
 = 2\lambda_{\min}(E(\mu))\lambda_{\min}\big(-A^S(\mu)\big).
\end{align*}
\end{proof}
Since computing the eigenvalues of $E(\mu)$ and $A(\mu)$ for all requested parameters $\mu$ can be expensive, there are several ways to find cheap lower bounds of $2\lambda_{\min}(E(\mu))\lambda_{\min}\big(-A^S(\mu)\big)$, some of which are derived in \cite{morSonS17}. Based on these bounds, we can, e.\,g., estimate
\begin{align}
2\lambda_{\min}(E(\mu))\lambda_{\min}\big(-A^S(\mu)\big) 
&\geq 
2\min_{k=0,\ldots,n_{E}}\frac{\Theta_{k}^{E}(\mu)}{\Theta_{k}^{E}(\bar{\mu})}
\lambda_{\min}(E(\bar{\mu}))\min_{k=0,\ldots,n_{A}}\frac{\Theta_{k}^{A}(\mu)}{
\Theta_{k}^{A}(\bar{\mu})}\lambda_{\min}\big(-A^S(\bar{\mu})\big) \label{eq:mintheta}\\ 
&=:\alpha(\mu), \notag
\end{align}
where $\bar{\mu}$ is an arbitrary fixed parameter in $\mathcal{D}$ and where we 
make use of the convention that $\Theta_{0}^{E} = \Theta_{0}^{A} \equiv 1$.
Here we compute the eigenvalues of $E(\bar{\mu})$ and $A^S(\bar{\mu})$ only once to find a lower bound $\alpha(\mu)$ of $\sigma_{\min}^+\left(\bPil^{\otimes}(\mu)\bcL(\mu)\bPir^{\otimes}(\mu)\right)$ for every parameter $\mu\in\mathcal{D}$.

\subsubsection{Mechanical systems}\label{sssec:sdIndex3}
Next we turn to the mechanical system from Subsection~\ref{sec:2.2}. There, the corresponding matrix pencil set 
\begin{equation*}
\left\{ s\begin{bmatrix}
I_{n_x} & 0 \\ 0 & M(\mu) 
\end{bmatrix}- \begin{bmatrix}
0 & I_{n_x}\\ -K(\mu) & -D(\mu)
\end{bmatrix} \in\mathbb{R}[s]^{2n_x , 2n_x}\; \bigg| \; \mu \in \mathcal{D}
\right\}
\end{equation*}
is \emph{not} uniformly strictly dissipative. However, it can be made uniformly strictly dissipative by appropriate generalized state-space transformations as we will show now.  

Here we apply a transformation presented in \cite{Eid2011, PanWL12}.
We observe that by adding productive zeros and for an auxiliary function $\gamma:\mathcal{D} \to (0,\infty)$ yet to be chosen, the system \eqref{Index3} is equivalent to the system 
\begin{align*}
\frac{\mathrm{d}}{\mathrm{d}t}K(\mu)x(t) + \gamma(\mu)\frac{\mathrm{d}^2}{\mathrm{d}t^2}M(\mu)x(t) 
&= -\gamma(\mu)K(\mu)x(t)+ \frac{\mathrm{d}}{\mathrm{d}t}K(\mu)x(t)\\
&\hspace{0pt} - \gamma(\mu)\frac{\mathrm{d}}{\mathrm{d}t}D(\mu)x(t) + \gamma(\mu)G(\mu)\lambda(t) + \gamma(\mu)B_xu(t), \\
\gamma(\mu)\frac{\mathrm{d}}{\mathrm{d}t}M(\mu)x(t) + \frac{\mathrm{d}^2}{\mathrm{d}t^2}M(\mu)x(t) 
&= -K(\mu)x(t) + \gamma(\mu)\frac{\mathrm{d}}{\mathrm{d}t}M(\mu)x(t)\\
& \hspace{50pt} - \frac{\mathrm{d}}{\mathrm{d}t}D(\mu)x(t)  + G(\mu)\lambda(t) + B_x(\mu)u(t),\\
0 &= \gamma(\mu)G(\mu)^{\mathrm{T}}x(t),\\
0 &= \frac{\mathrm{d}}{\mathrm{d}t}G(\mu)^{\mathrm{T}}x(t),
\end{align*}
where $x(t)$ is equal to $x_1(t)$ and $\frac{\mathrm{d}}{\mathrm{d}t}x(t)$ equal to $x_2(t)$ from system \eqref{Index3}.
The last equation is a direct consequence from the equation above and poses no extra restrictions.
These equations can be written as first-order system. By defining the matrices
\begin{align}\label{eq:sdMatrices}
E(\mu)&:=\begin{bmatrix}
K(\mu) & \gamma(\mu)M(\mu)\\
\gamma(\mu)M(\mu) & M(\mu)
\end{bmatrix},\quad A(\mu) := \begin{bmatrix}
-\gamma(\mu)K(\mu) & K(\mu)-\gamma(\mu)D(\mu)\\
-K(\mu) & -D(\mu)+\gamma(\mu)M(\mu)
\end{bmatrix},\\ B(\mu)&:=\begin{bmatrix}
\gamma(\mu)B_{x}(\mu)\\
B_{x}(\mu)
\end{bmatrix},\nonumber
\end{align}
and replacing $G(\mu)$ by $\begin{bsmallmatrix}
\gamma(\mu)G(\mu) & 0 \\ 0 & G(\mu)\end{bsmallmatrix}$, we generate a system in the form \eqref{Index2} such that we can apply the methods from Subsection~\ref{ssec:dissip} for this kind of system in the following, if $\gamma(\mu)$ is chosen properly.
This is possible, because even if $A(\mu)$ is unsymmetric, the projector $\Pi(\mu)$ in \eqref{eq:defPi} will be the same in both left and right spectral projectors $\bPil(\mu)$ and $\bPir(\mu)$ and hence, the estimate in Theorem~\ref{thm:estimalpha} is still valid. Only an estimation of the form \eqref{eq:mintheta} is not possible, since the matrices $E(\mu)$ and $A^S(\mu)$ cannot be written as affine decompositions in which all summands are positive or negative semidefinite, respectively. So in our numerical experiments we make use of Theorem~\ref{thm:estimalpha} directly.

\begin{theorem}
Consider the mechanical system \eqref{Index3}. Then the matrix pencil set $\{sE(\mu)-A(\mu)\; | \; \mu \in \mathcal{D}\}$ with $E(\mu)$ and $A(\mu)$ defined in \eqref{eq:sdMatrices} is uniformly strictly dissipative, if $0 < \gamma(\mu) < \gamma^*(\mu)$ with
\[	\gamma^*(\mu) = \frac{\lambda_{\min}(D(\mu))}{\lambda_{\max}(M(\mu))+\frac{1}{4}\lambda_{\max}(D(\mu))^{2}\lambda_{\max}(K(\mu)^{-1})}
\]
is satisfied for all $\mu \in \mathcal{D}$.
\end{theorem}
\begin{proof}
As shown in \cite{PanWL12}, the pencil set $\{sE(\mu)-A(\mu)\; | \; \mu \in \mathcal{D}\}$ is uniformly strictly dissipative, if we choose $\gamma(\mu) > 0$ such that
\begin{equation*}
 \gamma(\mu) < \lambda_{\min}\left( D(\mu) \left( M(\mu) + \frac{1}{4} D(\mu)K(\mu)^{-1}D(\mu) \right)^{-1} \right) \quad \forall\,\mu \in \mathcal{D}.
\end{equation*}
Using the estimate $\lambda_{\min}(XY) \ge \lambda_{\min}(X)\lambda_{\min}(Y)$ for symmetric and positive definite matrices $X$ and $Y$ of conforming dimension \cite[Theorem~4]{WanZ92}, we arrive at the choice for $\gamma(\mu)$ given by
\begin{equation}\label{gamma}
	\gamma(\mu)\lambda_{\max}\left(M(\mu)+\frac{1}{4}\; D(\mu)K(\mu)^{-1}D(\mu)\right) < \lambda_{\min}(D(\mu))\quad \forall\,\mu \in \mathcal{D}.
\end{equation}
Because of the symmetry and positive definiteness of $D(\mu)K(\mu)^{-1}D(\mu)$ and the submultiplicativity of the norm we obtain
\begin{align*}
\lambda_{\max}\left(D(\mu)K(\mu)^{-1}D(\mu)\right)
&= \left\| D(\mu) K(\mu)^{-1} D(\mu) \right\|_2 \\ &\le {\|D(\mu)\|}_2^2 \cdot \big\| K(\mu)^{-1} \big\|_2 = 
 \lambda_{\max}(D(\mu))^2 \cdot \lambda_{\max}(K(\mu)^{-1}).
\end{align*}

By Weyl's lemma \cite{HorJ91,LanT85}, we also have
\[\lambda_{\max}\left(M(\mu)+\frac{1}{4}D(\mu)K(\mu)^{-1}D(\mu)\right)\leq \lambda_{\max}(M(\mu)) + \frac{1}{4}\lambda_{\max}\big(D(\mu)K(\mu)^{-1}D(\mu)\big).\]

Therefore, condition \eqref{gamma} is satisfied, if we choose $\gamma(\mu)$ such that
\begin{equation*}
	\gamma(\mu)\left(\lambda_{\max}(M(\mu))+\frac{1}{4}\; \lambda_{\max}(D(\mu))^{2}\lambda_{\max}\big(K(\mu)^{-1}\big)\right) < \lambda_{\min}(D(\mu)) \quad \forall\,\mu \in \mathcal{D}.
\end{equation*}
This can be achieved by the choice
\begin{equation}\label{eq:fin_gamma}
	\gamma(\mu)<\frac{\lambda_{\min}(D(\mu))}{\lambda_{\max}(M(\mu))+\frac{1}{4}\lambda_{\max}(D(\mu))^{2}\lambda_{\max}(K(\mu)^{-1})}.
\end{equation}
\end{proof}

To avoid computing the eigenvalues of $M(\mu),\, D(\mu),\, K(\mu)$ for every parameter $\mu$, we further use the estimates
\begin{align*}
	\lambda_{\max}(D(\mu))&\leq 
\lambda_{\max}(D_0) + \sum_{k=1}^{n_{D}}\Theta_{k}^{D}(\mu)\lambda_{\max}(D_{k}
) = \sum_{k=0}^{n_{D}}\Theta_{k}^{D}(\mu)\lambda_{\max}(D_{k}
), \\
\lambda_{\min}(D(\mu))&\geq 
\lambda_{\min}(D_0) + 
\sum_{k=1}^{n_{D}}\Theta_{k}^{D}(\mu)\lambda_{\min}(D_{k}) = 
\sum_{k=0}^{n_{D}}\Theta_{k}^{D}(\mu)\lambda_{\min}(D_{k}),
\end{align*}
that are a consequence of Weyl's lemma. Similar estimates are also valid for 
$M(\mu)$ and $K(\mu)$.
%Assume now that at there exist $k_D \in \{ 1,\,\ldots,\,n_D \}$ and $k_K \in 
%\{1,\,\ldots,\,n_K\}$ such that $D_{k_D}$ and $K_{k_K}$ are \emph{positive 
%definite} and that $\Theta_{k_D}^D(\mu) > 0$ and $\Theta_{k_K}^K(\mu) > 0$ for 
%all $\mu \in \mathcal{D}$. 
Then we can further estimate $\gamma(\mu)$ as
\begin{equation}\label{eq:gamma_star}
\gamma(\mu)<\gamma^*(\mu)=\frac{\sum_{k=0}^{n_{D}}\Theta_{k}^{D}(\mu)\lambda_{\min}(D_{k})}
{ 
\sum_{k=0}^{n_{M}}\Theta_{k}^{M}(\mu)\lambda_{\max}(M_{k})+\frac{1}{4}
\left(\sum_ { k=0 
}^{n_{D}}\Theta_{k}^{D}(\mu)\lambda_{\max}(D_{k})\right)^{2}/\left(\sum_{k=0}^{
n_ { K } } \Theta_{k}^{K}(\mu)\lambda_{\min}(K_{k})\right)},
\end{equation}
while ensuring that we can choose $\gamma(\mu) > 0$ for all $\mu \in \mathcal{D}$.

The benefit of the above estimate is that the extremal eigenvalues of the matrices $M_k$, $D_k$, and $K_k$ have to be computed only once in the beginning and otherwise, its evaluation is cheap, since $\Theta_k^M(\cdot)$, $\Theta_k^D(\cdot)$, $\Theta_k^K(\cdot)$ are scalar-valued functions.

\subsection{Treatment of the algebraic parts}\label{ssec:treatmentalg}
The goal of this subsection is to discuss the reduction of the algebraic part 
of the system which is potentially also parameter-dependent. To do so, we first 
precompute specific matrices in the offline phase which will later be used to 
efficiently construct the Markov parameters in the online phase for a 
particular parameter configuration and then determine a reduced representation 
of the algebraic subsystem using the construction from the last paragraph in 
Subsection~\ref{ssec:D0D1determ}. In order for this construction to be 
efficient, we assume that the parameter influence has an additional 
\emph{low-rank} structure, i.\,e., with regard to the affine decomposition 
\eqref{eq:affinedeco}, we assume that $\bE_k$, $\bA_k$, $\bB_k$, and $\bC_k$ 
are low-rank matrices for $k \ge 1$, i.\,e., they admit low-rank factorizations
\begin{align}\label{eq:lowrank}
\begin{split}
 \bE_k &= \bU_{\bE,k} \bV_{\bE,k}^{\mathrm{T}}, \quad k=1,\,\ldots,\,n_{\bE}, \\
 \bA_k &= \bU_{\bA,k} \bV_{\bA,k}^{\mathrm{T}}, \quad k=1,\,\ldots,\,n_{\bA}.
%\\
 %\bB_k &= \bU_{\bB,k} \bV_{\bB,k}^{\mathrm{T}}, \quad k=1,\,\ldots,\,n_{\bB}, 
%\\
% \bC_k &= \bU_{\bC,k} \bV_{\bC,k}^{\mathrm{T}}, \quad k=1,\,\ldots,\,n_{\bC}.
 \end{split}
\end{align}
To determine the Markov parameters, for a fixed set of sample points $\omega_i 
\in \mathbb{R}$, $i=1,\,\ldots,\,n_{\mathrm{S}}$, we aim to evaluate 
$\boldsymbol{\mathcal{G}}(\mu,\ri\omega_i)$ efficiently for all needed parameters $\mu\in\mathcal{D}$. In other words, we aim for an efficient procedure to evaluate the 
transfer function $\boldsymbol{\mathcal{G}}(\mu,\ri\omega)$ for fixed value $\omega \in 
\mathbb{R}$ and varying $\mu$.
%\begin{equation*}
% \boldsymbol{\mathcal{G}}(0,\ri\omega_i) = \bC_0(\ri\omega_i\bE_0 - 
%\bA_0)^{-1} 
%\bB_0, \quad i=1,\,\ldots,\,n_{\mathrm{S}}.
%\end{equation*}
Due to \eqref{eq:lowrank} and since $n_{\bE}$ and $n_{\bA}$ are assumed to be 
very small, also the matrix 
$\ri\omega \sum_{k=1}^{n_{\bE}}\Theta_{k}^{\bE}(\mu)\bE_{k}-  
\sum_{k=1}^{n_{\bA}}\Theta_{k}^{\bA}(\mu)\bA_{k}$ is of  low rank and we can 
factorize it as
\begin{multline*}
 \ri\omega \sum_{k=1}^{n_{\bE}}\Theta_{k}^{\bE}(\mu)\bE_{k}-  
\sum_{k=1}^{n_{\bA}}\Theta_{k}^{\bA}(\mu)\bA_{k} \\ = \begin{bmatrix} 
\bU_{\bE,1}^\mathrm{T} \\ \vdots \\ \bU_{\bE,n_{\bE}}^\mathrm{T} 
\\ \bU_{\bA,1}^\mathrm{T} \\ \vdots \\ \bU_{\bA,n_{\bA}}^\mathrm{T} 
\end{bmatrix}^\mathrm{T} \begin{bmatrix} \ri\omega 
\Theta_{1}^{\bE}(\mu) &&&&& \\ & \ddots &&&& \\
&&\ri\omega\Theta_{n_{\bE}}^{\bE}(\mu)&&& \\ &&&
-\Theta_{1}^{\bA}(\mu) && \\ &&&&\ddots& \\ 
&&&&&-\Theta_{n_{\bA}}^{\bA}(\mu) \end{bmatrix} 
\begin{bmatrix}\bV_{\bE,1}^\mathrm{T} \\ \vdots \\ \bV_{\bE,n_{\bE}}^\mathrm{T} 
\\ \bV_{\bA,1}^\mathrm{T} \\ \vdots \\ \bV_{\bA,n_{\bA}}^\mathrm{T} 
\end{bmatrix} =: \bU \boldsymbol{\Omega}(\mu,\ri\omega) \bV^\mathrm{T}.
\end{multline*}
It holds that
\begin{align*}
 \boldsymbol{\mathcal{G}}(\mu,\ri\omega) = 
\sum_{i=0}^{n_\bC} \sum_{j=0}^{n_\bB} 
\Theta_{i}^{\bC}(\mu)\Theta_{j}^{\bB}(\mu) \bC_i\left( \ri\omega\bE_0 - 
\bA_0 + \bU \boldsymbol{\Omega}(\mu,\ri\omega) \bV^\mathrm{T}\right)^{-1} 
\bB_j.
\end{align*}
With the help of the Sherman-Morrison-Woodbury identity we obtain 
\begin{multline*}
 \bC_i\left( \ri\omega\bE_0 - \bA_0 
+ \bU \boldsymbol{\Omega}(\mu,\ri\omega) \bV^\mathrm{T}\right)^{-1} \bB_j = 
\bC_i \left(\ri\omega\bE_0 - \bA_0 \right)^{-1} \bB_j \\ - \bC_i \left(\ri\omega\bE_0 - 
\bA_0 \right)^{-1} \bU \left(
\boldsymbol{\Omega}(\mu,\ri\omega)^{-1}+\bV^{\mathrm{T}} \left(\ri\omega\bE_0 - \bA_0 
\right)^{-1} \bU \right)^{-1} \bV^{\mathrm{T}}\left(\ri\omega\bE_0 - \bA_0 
\right)^{-1} \bB_j. 
\end{multline*}
The matrices $\bC_i\left(\ri\omega\bE_0 - \bA_0 \right)^{-1} \bB_j$, 
$\bC_i \left( \ri\omega\bE_0 - \bA_0 \right)^{-1} \bU$, $\bV^{\mathrm{T}} 
\left(\ri\omega\bE_0 - \bA_0 \right)^{-1} \bU$, and \linebreak
$\bV^{\mathrm{T}}\left( \ri\omega\bE_0 - \bA_0 \right)^{-1} \bB_j$ are 
all of small dimension and they can be precomputed in the offline phase. In the 
online phase, the small matrix 
$\boldsymbol{\Omega}(\mu,\ri\omega)^{-1}+\bV^{\mathrm{T}} \left(\ri\omega\bE_0 - \bA_0 
\right)^{-1} \bU$ can be formed (under the assumption that all necessary 
matrices are invertible) and used for small and dense system solves in order to 
calculate $\boldsymbol{\mathcal{G}}(\mu,\ri\omega)$ with only little effort.

\section{Numerical results}\label{ChapterResults}

%\textcolor{red}{TODO: Recalculate all results, take care of proper bounds in error estimation, describe in details which bounds have been used, update and put more details in captions ($\mu \to \mu_2$), try $\frac{1}{2}\gamma^*$ in mechanical systems, treat complier warnings.}

In this section, we present the numerical results for the two differential-algebraic systems from Section~\ref{sec:examples}.
The computations have been done on a computer with 2 Intel Xeon Silver 4110 CPUs running at 2.1 GHz and equipped with 192 GB total main memory. 
The experiments use \matlab R2021a (9.3.0.713579) and examples and methods from M-M.E.S.S.-2.2. \cite{SaaKB22-mmess-2.2}.

\subsection{Problem 1: Stokes equation}
We consider a system that describes the creeping flow in capillaries or porous media.
It has the following structure
\begin{align}\label{Stokes}
\begin{split}
\frac{\mathrm{d}}{\mathrm{d}t}v(\zeta,t) &= \mu\Delta v(\zeta,t) - \nabla p(\zeta,t) +f(\zeta,t),\\
0 &= \mathrm{div}(v(\zeta,t)),
\end{split}
\end{align}
with appropriate initial and boundary conditions. The position in the domain $\Omega\subset\mathbb{R}^{\ell}$ is described by $\zeta\in\Omega$ and $t\geq0$ is the time.
For simplicity, we use a classical solution concept and assume that the external force $f: \Omega \times [0,\infty) \to \mathbb{R}^\ell$ is continuous and that 
the velocities $v: \Omega \times [0,\infty) \to \mathbb{R}^\ell$ and pressures $p: \Omega \times [0,\infty) \to \mathbb{R}^\ell$ satisfy the necessary smoothness conditions.
The parameter $\mu\in\mathcal{D}$ represents the dynamic viscosity.
We discretize system \eqref{Stokes} by finite differences as shown in \cite{morMehS05, morSty06} and add an output equation.
Then, we obtain a discretized system of the form \eqref{Index2}, which is of index 2.
The matrix $A(\mu) = \mu \cdot A\in \mathbb{R}^{n, n}$ depends on the viscosity parameter $\mu$.

\begin{figure}[bt]
	\centering
	\includegraphics[scale=0.35]{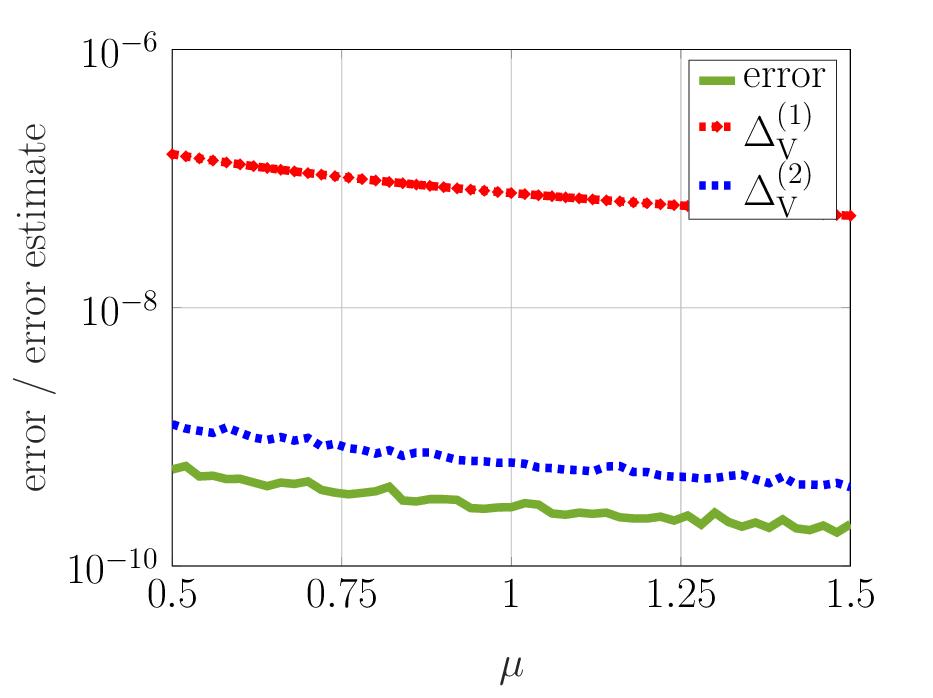}\label{fig:RBM2_2}
	\caption{Error and error estimates of the approximated controllability Gramians for the Stokes system~\eqref{Index2} after the first iteration of the reduced basis method.}
	\label{BBM2pic}
\end{figure}

\begin{figure}[bt]
	\centering
	\subfloat[Sigma plots of the original, reduced, and error transfer functions 
	for the test parameter $\mu=1.28$.]{%
		\includegraphics[scale=0.35]{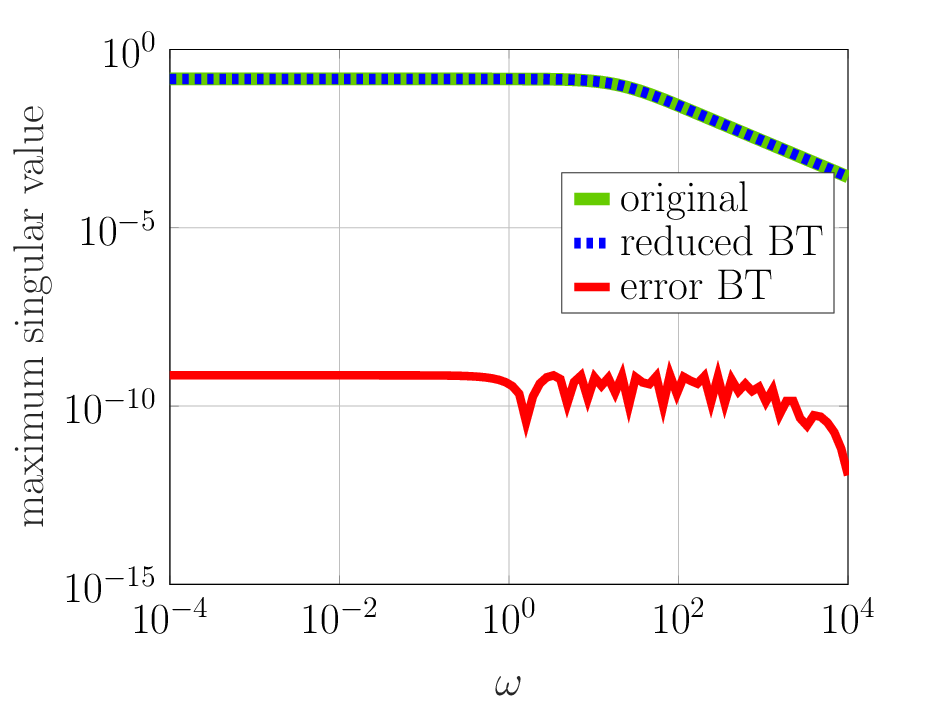}\label{BT2_50pic}}
	\hspace{20pt}
	\subfloat[Sigma plots of the original, reduced, and error transfer functions for the parameter $\mu=0.65$.]{%
		\includegraphics[scale=0.35]{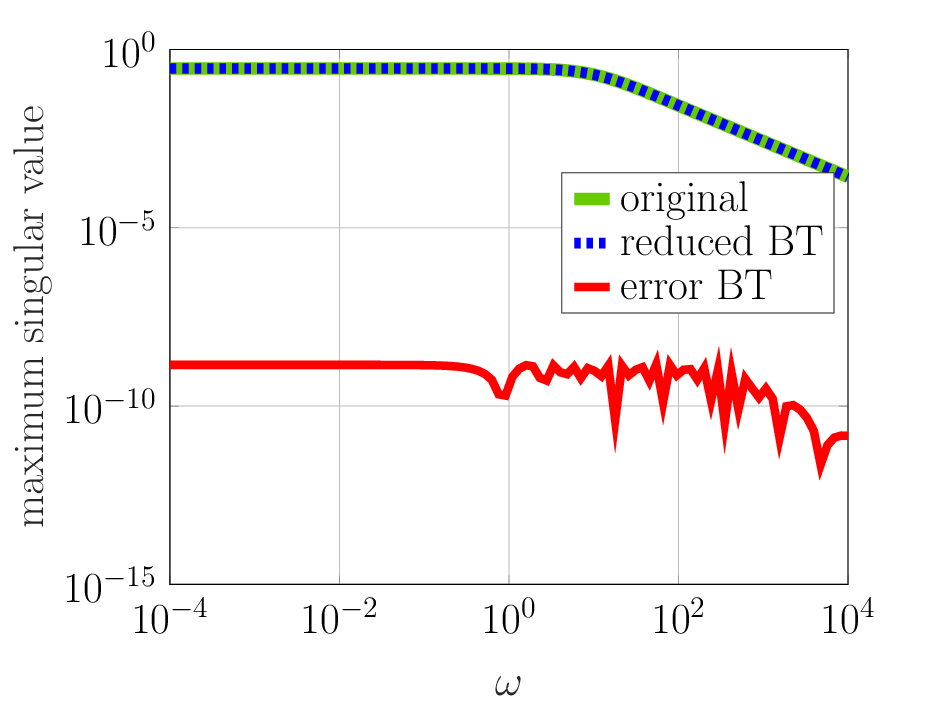}\label{BT2pic}}
	\caption{Results for the reduction of the Stokes system \eqref{Index2}.}
\end{figure}
%\textcolor{red}{Update caption, 1.27 ist laut Beschreibung kein Testparameter mehr oder?}

Our constant example matrices $E=I_n,\, A,\, G,\, B_1,\, C_1$ are created by the M-M.E.S.S. function $\mathtt{stokes\_FVM}$ and we
choose the dimensions $n=3280,\, q=1681$, $m,\, p =1$, and the parameter set $\mathcal{D}=\left[\frac{1}{2},\,\frac{3}{2}\right]$.
The test parameter set $\mathcal{D}_{\mathrm{Test}}$ consists of ten equidistant points within $\mathcal{D}$, i.\,e., 
 $\mathcal{D}_{\mathrm{Test}} = \left\{  \frac{1}{2} + \frac{1}{9}\cdot k \; \big| \; k = 0,\,\ldots,\,9\right\}$.
The matrices $B_2$ and $C_2$ are zero matrices.
The reduced basis method from Section~\ref{reduced basis method} produces bases of dimensions $24$ for both projected continuous-time Lyapunov equations.
This basis generation is done in the offline phase, which takes $2.7\cdot 10^3$ seconds for this example.
For the first iteration of the reduced basis method corresponding to the proper controllability Lyapunov equation, we evaluate the errors and their estimates from equation \eqref{eq:err1_part1} and \eqref{ErrEstF}.
They are presented in Figure~\ref{BBM2pic}, where the error is evaluated by 
\begin{equation*}
\textrm{error} \approx {\big\| \bP_{\mathrm{p}}^{\mathrm{acc}}(\mu) - \bZ(\mu)\bZ(\mu)^{\mathrm{H}} \big\|}_{\mathrm{F}},
\end{equation*}
where $\bP_{\mathrm{p}}^{\mathrm{acc}}(\mu)$ denotes an accurate approximation of the exact solution $\bP_{\mathrm{p}}(\mu)$. For evaluating the error, the Lyapunov equations are also solved with the residual tolerance $10^{-15}$ to obtain an accurate estimate of the exact solution.
After the first iteration step, we obtain an error that is already smaller than the tolerance $\mathrm{tol}=10^{-4}$.
Additionally, we see that the error estimator $\Delta_{\bV}^{(2)}$ provides an almost sharp bound of the actual error while $\Delta_\bV^{(1)}$ leads to more conservative error bounds.

After we have determined the bases, we approximate the proper controllability and observability Gramians to apply balanced truncation in the online phase for two different parameters $\mu=1.28$ and $\mu=0.65$ where the first parameter is the eighth entry from the test parameter set $\mathcal{D}_{\mathrm{Test}}$.
The reduced systems for both parameters are of dimension $r=10$ and consist only of differential states. 
This online phases last for $3.7\cdot 10^{-2}$ and $1.9\cdot 10^{-2}$ seconds, respectively, for this example. 
Figures~\ref{BT2_50pic} and~\ref{BT2pic} depict the transfer functions of the full and the reduced systems and the corresponding errors. We observe that the error is smaller than $ 10^{-8}$ for all $\omega$ in $[10^{-4},\, 10^4]$.
Additionally, in Figures~\ref{out2_50pic} and~\ref{out2pic}, we show the output 
$y(t)$ of the full system and the output  $y_{\mathrm{R}}(t)$ of the reduced system for an input $u(t)=-2 - \sin(t)$, $t\in[0,\, 10]$, and the initial condition is given as $\bz(0)=0$.
We also depict the error ${\|y(t) - y_{\mathrm{R}}(t)\|}_2$ which is 
smaller than $10^{-8}$ in the entire time range.
Hence, we can observe that the system behavior is well-approximated by the reduced system.

\begin{figure}[bt]
	\centering
	\subfloat[Output plots of the original and reduced system and the 
	corresponding error for the test parameter $\mu=1.28$.]{%
		\includegraphics[scale=0.35]{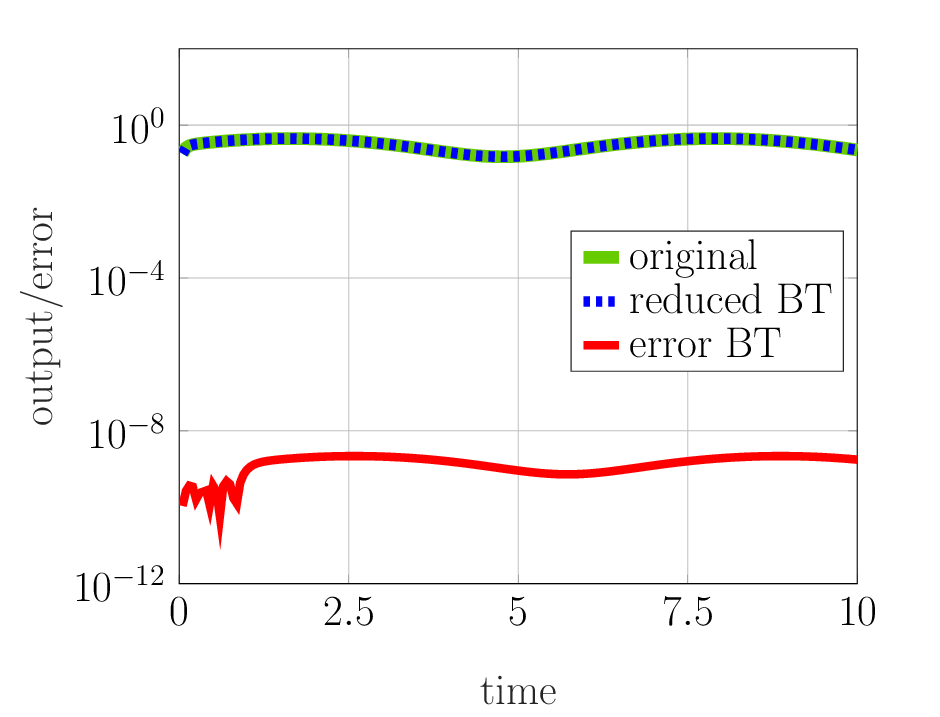}\label{out2_50pic}}
	\hspace{20pt}
	\subfloat[Output plots of the original and reduced system and the corresponding error for the parameter $\mu=0.65$.]{%
		\includegraphics[scale=0.35]{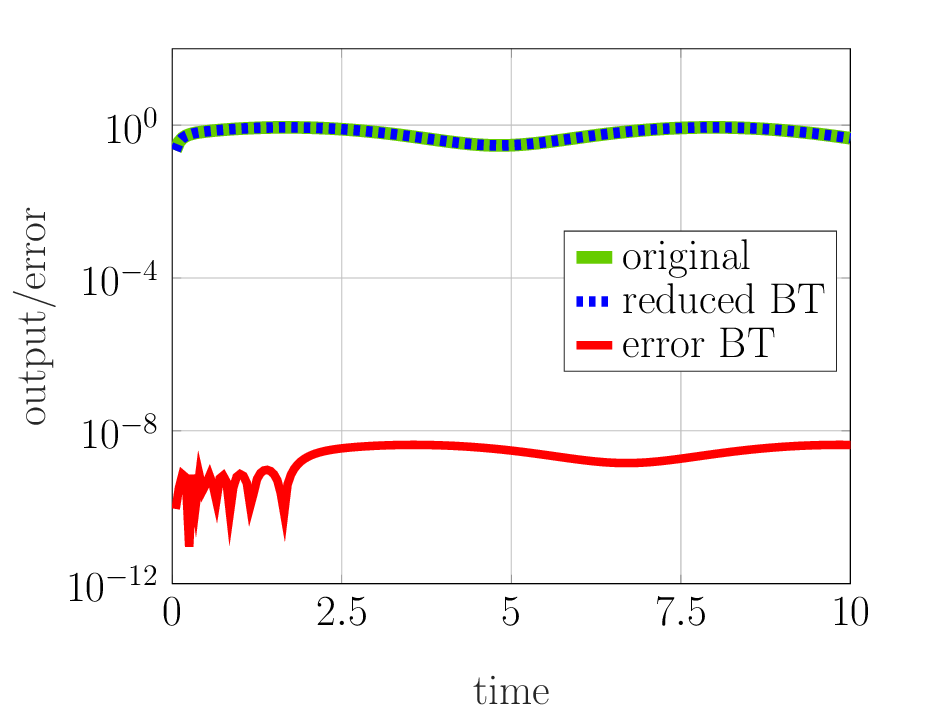}\label{out2pic}}
	\caption{Output and output error of the original and reduced Stokes system \eqref{Index2}}
\end{figure}

Now, to demonstrate that the methods presented in this paper are also 
suitable for systems with an improper transfer function, we modify the system in such a way that the matrices $B_2$ and $C_2$ are chosen to be $
B_2 = \begin{bmatrix}
0 &  \dots & 0 & 1
\end{bmatrix}^{\mathrm{T}}$ and  $C_2 = \begin{bmatrix}
1 &  0 & \dots & 0 & 1
\end{bmatrix}.
$ 
That way, we obtain more complex restrictions for $G^{\mathrm{T}}x(t)$ and also evaluate the pressure described by $\lambda(t)$ to generate an output $y(t)$.
We additionally assume that the external force represented by the input function $B_1(\mu) u(t)$ varies its magnitude depending on the viscosity of the system so that the input matrix $B_1(\mu) = (1+\mu)\cdot B_1$ is parameter-dependent, and hence, the polynomial part of the transfer function is parameter-dependent as well.

Again, in the offline phase, we execute the reduced basis method to generate bases that approximate the proper controllability and observability spaces (i.\,e., the image space of the corresponding proper Gramians) of this system.
This phase takes $9.3\cdot 10^{3}$ seconds for this example.
In Figure~\ref{fig:RBM2im_1} and~\ref{fig:RBM2im_2}, the errors and error estimations in the two stages of the reduced basis method corresponding to the proper controllability space are shown. For evaluating the error, the Lyapunov equations are also solved with the residual tolerance $10^{-15}$ to obtain an accurate estimate of the exact solution.
We can observe that after the first iteration step, we obtain an error that is significantly larger than the tolerance $10^{-4}$ as shown in Figure~\ref{fig:RBM2im_1}.
Additionally, we see that the error estimator $\Delta_{\bV}^{(2)}$ provides an almost sharp bound of the actual error while $\Delta_\bV^{(1)}$ leads to more conservative error bounds.
The second iteration of the reduced basis method and the corresponding errors are depicted in Figure~\ref{fig:RBM2im_2}.
In this case, the error as well as the estimators are smaller than the tolerance $\mathrm{tol}=10^{-4}$. 
The final basis projection space dimensions for the projected continuous-time Lyapunov equations are $48$ and $54$.

The next step is to use the bases to generate approximations of the system Gramians and to apply balanced truncation, which leads to a reduced system of dimension $25$ where we obtain $23$ differential states and $2$ algebraic ones for both parameters $\mu=1.28$ and $\mu=0.65$.
The online phase lasts for $3.2$ and $3.1$ seconds if we use balanced truncation via the improper Gramians to reduce the improper system components and $2.4$ seconds for both parameters if we extract the polynomial part of the transfer function. However, in this example, the parameter dependency does not have a low-rank structure. Thus, for the determination of the polynomial part, it is necessary to solve one linear system of equations of full order in the online phase for each parameter of interest.
Explicit formulas of the improper Gramians 
are also described in \cite{morSty06}. For a fixed parameter, also an explicit formula 
\begin{align}\label{eq:SAR_index2}
\bS_{\mathrm{i}}^{\mathrm{T}}\bA \bR_{\mathrm{i}}
=
\begin{bmatrix}
B_{11} & C_2(G^{\mathrm{T}}G)^{-1}B_2\\
C_2(G^{\mathrm{T}}G)^{-1}B_2 & 0 
\end{bmatrix}
\end{align}
is provided with 
\[
B_{11}:=C_1G(G^\mathrm{T}G)^{-1}B_2 + C_2(G^\mathrm{T}G)^{-1}G^{\mathrm{T}}B_{12},\qquad
B_{12}:= B_1 - AG(G^{\mathrm{T}}G)^{-1}B_2,
\]
compare with \eqref{eq:SVD_improp}.
This formula replaces the computation of the low-rank factors of the improper Gramians in the online phase.
In our setting, \eqref{eq:SAR_index2} can be evaluated efficiently since all involved matrices can be precomputed and then appropriately scaled with $\mu$ in the parametric setting.

\begin{figure}[bt]
	\centering
	\subfloat[First step of the reduced basis method.]{%
		\includegraphics[scale=0.35]{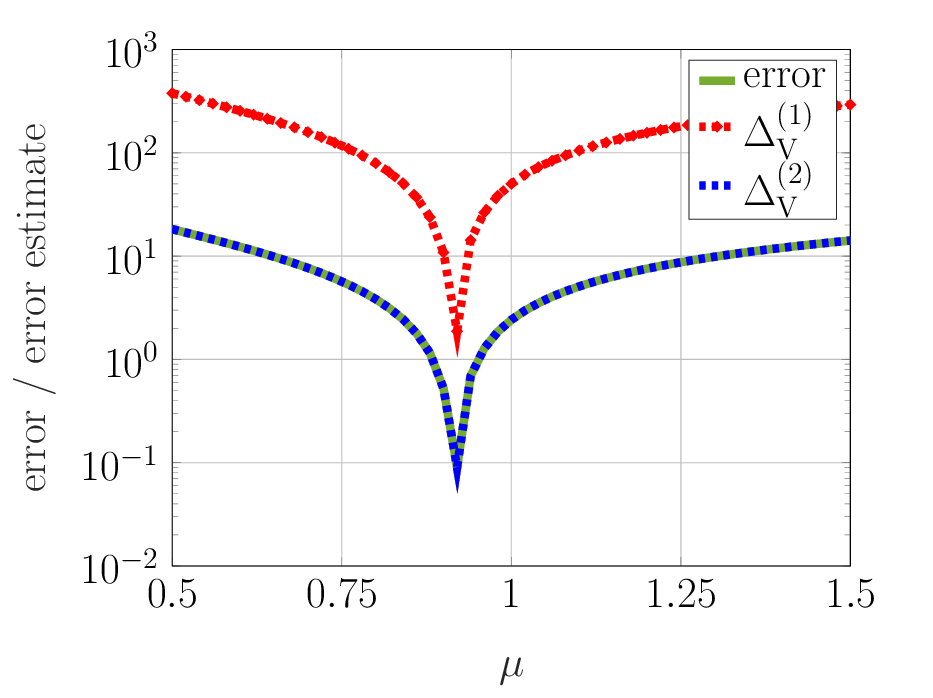}\label{fig:RBM2im_1}}
	\hspace{20pt}
	\subfloat[Second step of the reduced basis method.]{%
		\includegraphics[scale=0.35]{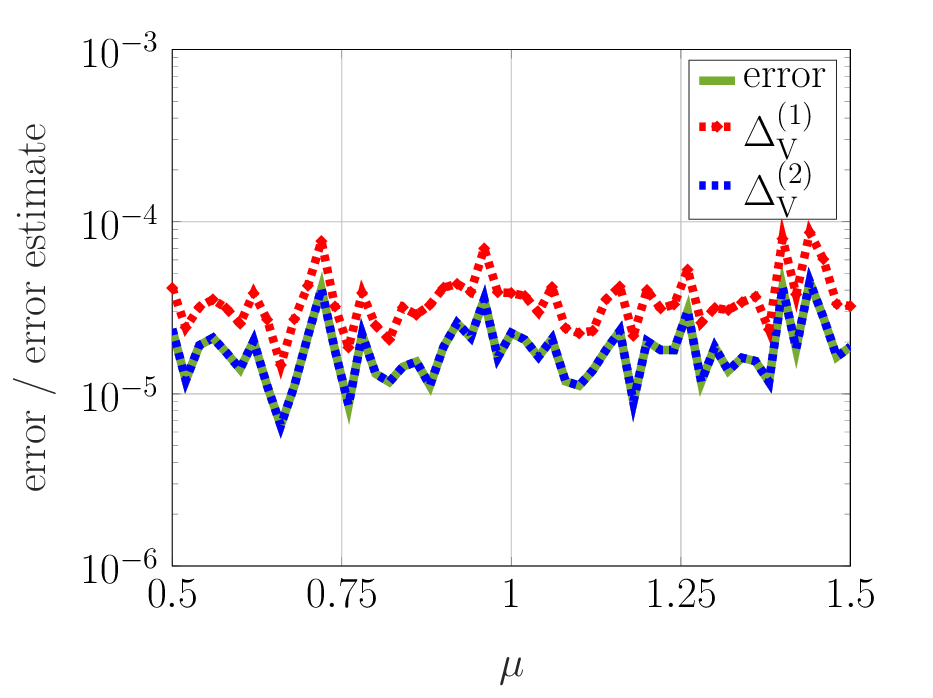}\label{fig:RBM2im_2}}
	\caption{Error and error estimates of the approximated Gramians for the Stokes system~\eqref{Index2} with improper parts. }
	\label{BBMim2pic}
\end{figure}

\begin{figure}[bt]
	\centering
	\subfloat[Sigma plots of the original, reduced, and error transfer functions 
	for the test parameter $\mu=1.28$.]{%
		\includegraphics[scale=0.35]{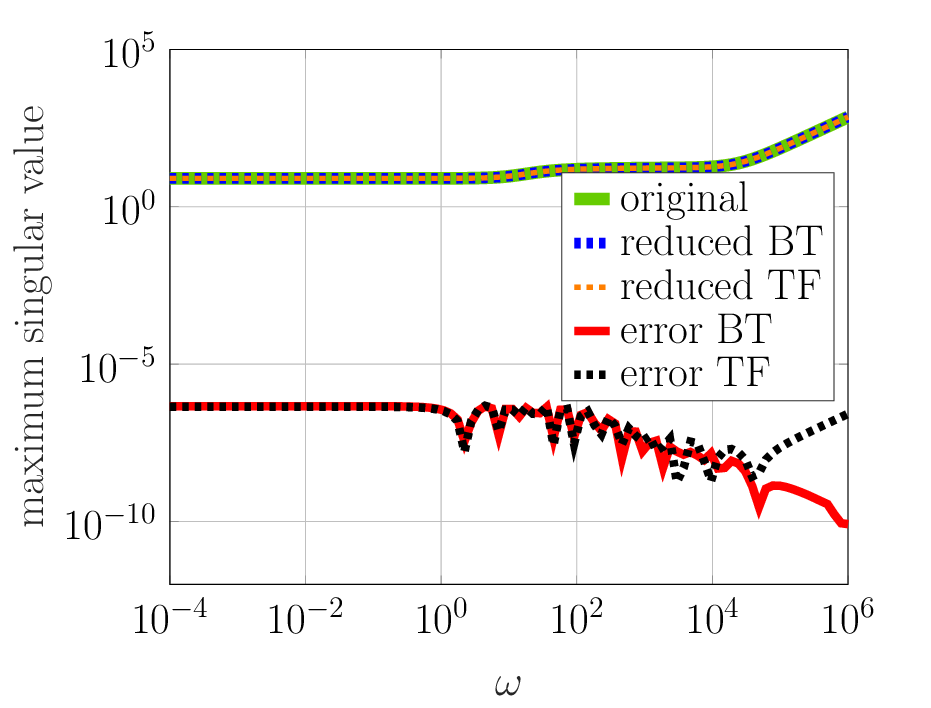}\label{BT2_50impic}}
	\hspace{20pt}
	\subfloat[Sigma plots of the original, reduced, and error transfer functions for the parameter $\mu=0.65$.]{%
		\includegraphics[scale=0.35]{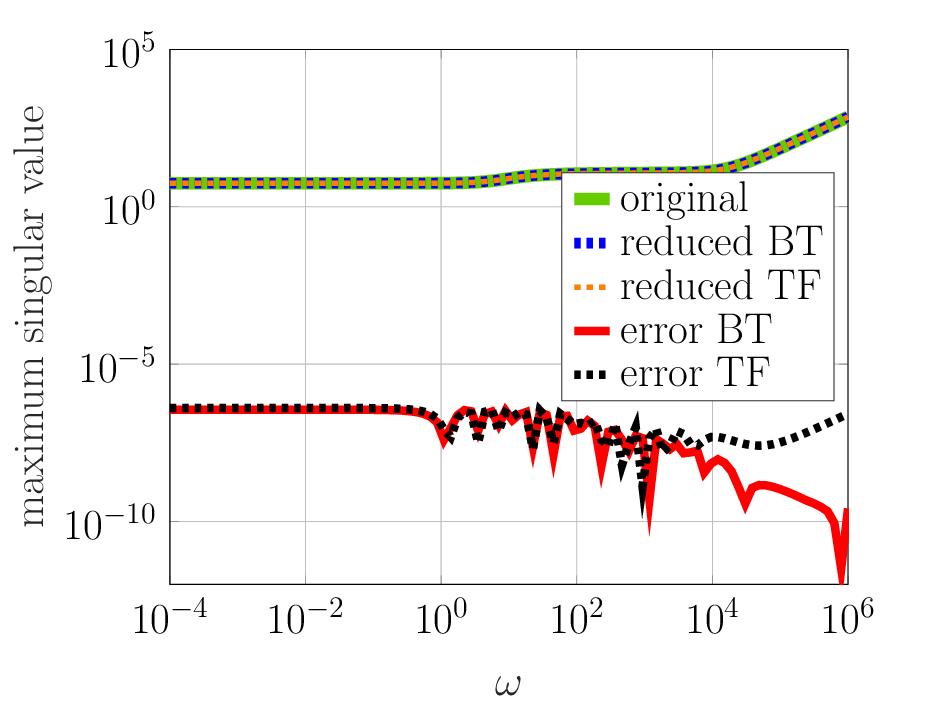}\label{BT2impic}}
	\caption{Results for the reduction of the Stokes system~\eqref{Index2} with improper parts.}
\end{figure}

We demonstrate the quality of the reduced systems by evaluating the error of the transfer functions. 
In the following, we denote the reduced system (error system) with an 
algebraic part generated by balanced truncation as \emph{reduced BT} 
(\emph{error BT}) and the reduced system (error system) with an algebraic 
part  obtained by approximating the polynomial part of the transfer function as 
\emph{reduced TF} (\emph{error TF}).
We show the sigma plot of the original and reduced transfer function as well as the error for two parameters.
The first parameter, depicted in Figure~\ref{BT2_50impic}, is $\mu = 1.28\in\mathcal{D}_{\mathrm{Test}}$ that is an element from the test parameter set and hence we already know that the error estimators $\Delta_{1/2}(\mu)$ are smaller than the tolerance and the proper controllability space is well-approximated.
The second parameter is $\mu = 0.65\in\mathcal{D}$, and the corresponding results are shown in Figure~\ref{BT2impic}. 
The polynomial part of the transfer function is clearly observable since it does not tend to zero for growing values of $\omega$.
We observe that for both parameters, the error is smaller than $10^{-5}$ in the entire frequency band $[10^{-4},\,10^6]$ and that the parameter chosen from the test parameter set does not lead to significantly better approximations than an arbitrarily chosen parameter from the parameter set $\mathcal{D}$.
This allows the conclusion that the basis $\bV_{\mathrm{glob}}$ approximates the proper controllability space well --- for all parameters in $\mathcal{D}$.
Additionally, we can observe that both procedures to approximate the improper part of the system lead to satisfactory results.

Finally, to illustrate that the output $y$ is well approximated by the 
reduced output $y_{\mathrm{R}}$, in Figure \ref{out2_50impic} and  
\ref{out2impic} we evaluate the norms of the output error ${\|y(t) - 
	y_{\mathrm{R}}(t)\|}_2$ for $t \in [0,\,100]$ for the test parameter $\mu=1.28$ and the parameter 
$\mu=0.65$.
We choose the input function $u(t)= 5-5\cdot \mathrm{cos}(t)$ and the initial condition $\bz(0)=\begin{bsmallmatrix}
x(0)\\ \lambda(0)
\end{bsmallmatrix}=0$, in such a way that the consistency conditions are satisfied.
We obtain an output error that is smaller than $10^{-5}$ for all 
$t\in[0,\, 100]$ for both evaluated parameters.
Again, we observe that both approximations of the algebraic part of the system lead to results of similar quality.

\begin{figure}[bt]
	\centering
	\subfloat[Output plots of the original and reduced system and the 
corresponding error for the test parameter $\mu=1.28$.]{%
		\includegraphics[scale=0.35]{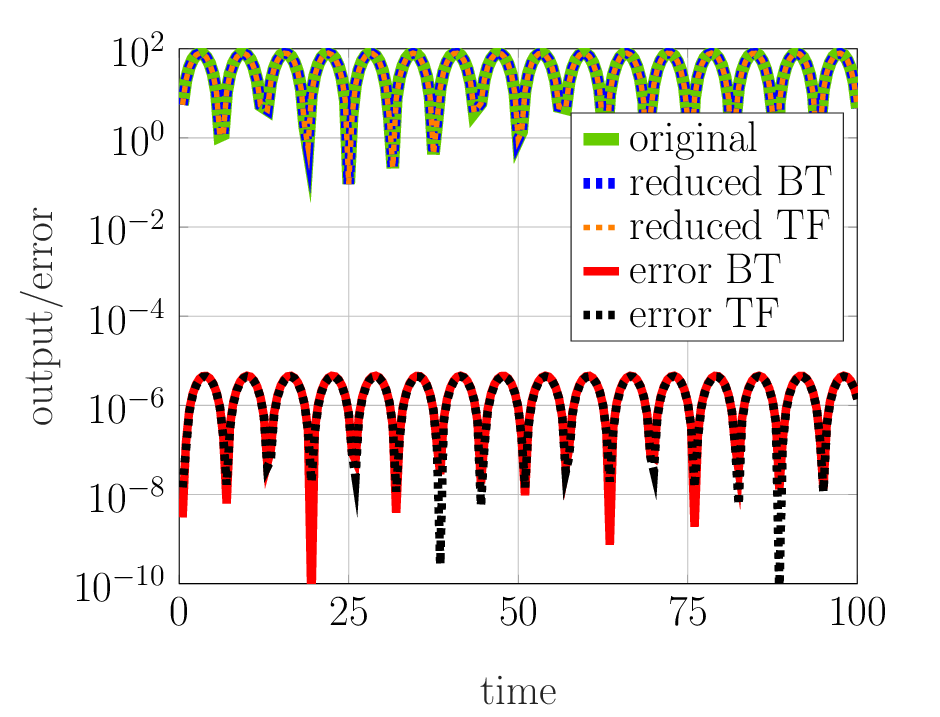}\label{out2_50impic}}
\hspace{20pt}
	\subfloat[Output plots of the original and reduced system and the corresponding error for the parameter $\mu=0.65$.]{%
		\includegraphics[scale=0.35]{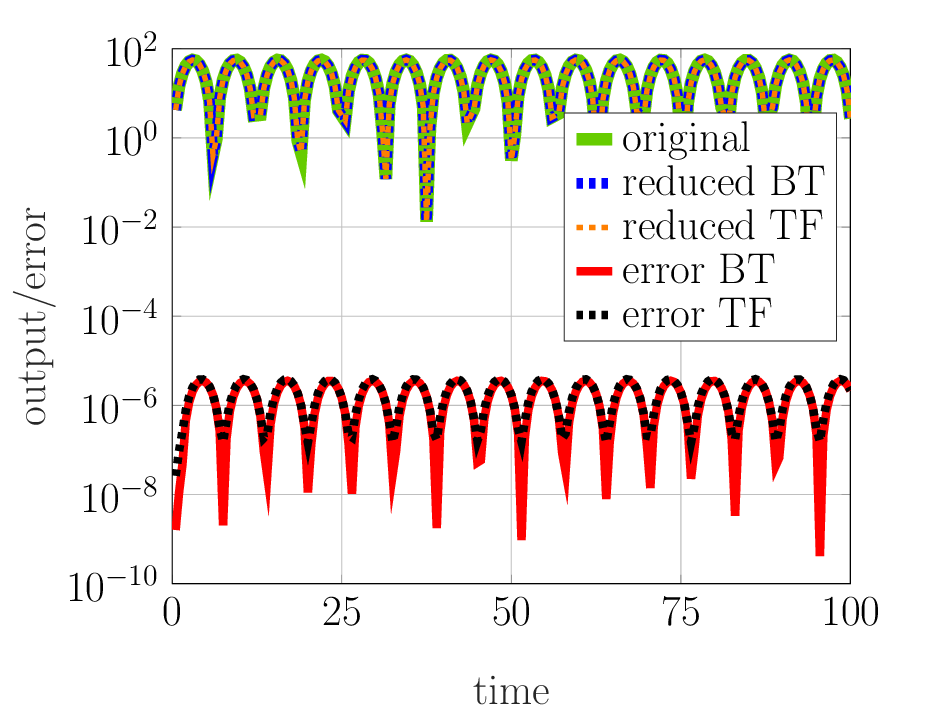}\label{out2impic}}
	\caption{Output and output error of the original and reduced Stokes system~\eqref{Index2} with improper parts}
\end{figure}

\subsection{Problem 2: Mechanical system}

\begin{figure}[bt]
\center
\centering
\scalebox{0.8}{\includegraphics{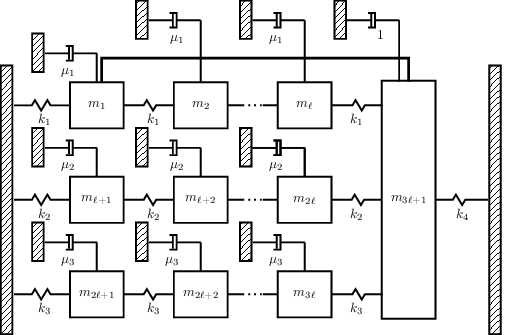}}
\caption{Constrained mass-damper-spring system including three chains of masses that are connected by the mass $m_{3\ell+1}$.}
\label{Index3pic}
\end{figure}

As an example for system \eqref{Index3}, we consider a constrained mass-spring-damper system which is depicted in Figure~\ref{Index3pic} and which is similar to the one considered in \cite{morUgr20}.
The matrices are generated using the function \texttt{triplechain\_MSD} from the M-M.E.S.S.-2.2. toolbox, see \cite{SaaKB22-mmess-2.2} where the mass matrix $M$ is built as
\[
M=\begin{bmatrix}
M_1 &     &     &  \\
& M_2 &     &      \\ 
&     & M_3 &      \\
&     &     &  m_0 \\
\end{bmatrix}, \qquad 
M_i = m_i\cdot I_\ell \in \mathbb{R}^{\ell, \ell},\quad i = 1,\,2,\,3
\]
for $m_1=m_2=m_3=m_0=1$.
The stiffness matrix is built as 
\begin{align*}
K = \begin{bmatrix}
K_{11} & & & \kappa_1 \\
& K_{22} &  & \kappa_2 \\
&  & K_{33} & \kappa_3 \\
\kappa_1^{\mathrm{T}} & \kappa_2^\mathrm{T} & \kappa_3^{\mathrm{T}} & k_1+k_2+k_3+k_0
\end{bmatrix},\quad K_{ii} = k_i \begin{bmatrix}
2  & -1     &        & & \\
-1 & 2      & -1     & &  \\
& \ddots & \ddots & \ddots  &\\
&        & -1     &2        & -1 \\
&        &        &-1 & 2
\end{bmatrix}\in\mathbb{R}^{\ell, \ell},
\end{align*}
with $\kappa_i = \begin{bmatrix}
0&
\dots &
0&
k_i
\end{bmatrix}^{\mathrm{T}}$ and $k_1=k_2=k_3=k_0=1$.
We set $\ell=200$ so that the dimension of $M$ and $K$ is $n=601$ and the matrix $G\in\mathbb{R}^{n,1}$ is a zero matrix except for the entries $G_{1,1} = 1$ and $G_{n,1} = -1$.
Hence, the first and the last mass are additionally connected by a rigid structure, as shown in Figure~\ref{Index3pic}.
Hence, the first-order system is of dimension $2n+2=1204$ after applying the transformation presented in Subsubsection~\ref{sssec:sdIndex3}.
The input matrix $B_x\in\mathbb{R}^{n,1}$ is chosen as a zero matrix with one nonzero entry 
$
(B_x)_{450,1} = 1.
$
The output matrix $C_x\in\mathbb{R}^{1,n}$ is given by $C_x = B_x^{\mathrm{T}}$.

These systems occur in the field of damping optimization, see \cite{morBenTT11a,morTruTP19,morUgr20,TruV09}, where the damping matrix $D\in\mathbb{R}^{n, n}$ consists of an internal damping $D_{\mathrm{int}}$ and some external dampers that need to be optimized.
In our example, we choose Rayleigh damping as a model for the internal damping, i.\,e., $D_{\mathrm{int}} = \eta M + \rho K$, where we set $\eta = 0.01$ and $\rho = 0.02$.
To describe the external damping, we assume that every mass is connected to a grounded damper, where the masses $m_1,\,\ldots,\, m_{\ell}$ are connected by dampers with viscosity $\mu_1$, the masses $m_{\ell+1},\,\ldots,\, m_{2\ell}$ by dampers with viscosity $\mu_2$ and the masses $m_{2\ell+1},\,\ldots,\, m_{3\ell}$ by dampers with viscosity $\mu_3$ as depicted in Figure \ref{Index3pic}.
The mass $m_{3\ell+1}$ is connected to a grounded damper with fixed viscosity $1$.
The external dampers can be described by the matrices 
\[
F_1 = \begin{bmatrix}
I_{\ell} &  \\
& 0_{2\ell+1,2\ell+1} 
\end{bmatrix}, \quad F_2 = \begin{bmatrix}
0_{\ell,\ell} & &  \\
& I_{\ell} &  \\
& & 0_{\ell+1,\ell+1} 
\end{bmatrix}, \quad F_3 = \begin{bmatrix}
0_{2\ell,2\ell} & & \\
& I_{\ell} & \\
& & 0
\end{bmatrix}, \quad F_4 = \begin{bmatrix}
0_{3{\ell},3\ell}& \\
& 1
\end{bmatrix},
\]
where $0_{n_1,n_2}$ denotes the zero matrix in $\mathbb{R}^{n_1,n_2}$. 
Therefore, the overall damping matrix can then be constructed as 
\begin{align*}
D(\mu) &:= D_{\mathrm{int}} + \mu_1F_1 + \mu_2F_2 + \mu_3F_3 + F_4,
\end{align*}
where $\mu = [\mu_1, \mu_2, \mu_3] \in \mathcal{D}$ are the viscosities of the dampers, which are optimized and thus variable in our model.
This model leads to a differential-algebraic system of index 3.
We generate the index-2 surrogate system as shown in Subsubsection \ref{sssec:sdIndex3} using the value $\gamma(\mu)=\frac{1}{2}\gamma^*(\mu)$ from \eqref{eq:gamma_star}.

We consider two parameter settings. In the first one, the system is damped more strongly, and in the second one more weakly. The latter setting is used to illustrate that the slightly damped system is more difficult to approximate by projection spaces of small dimensions, which poses a challenge from the numerical point of view.
For both parameter settings, the reduced models are obtained by truncating all state variables corresponding to proper Hankel singular values with $\sigma_{\mathrm{p},k}/\sigma_{\mathrm{p},1} < 10^{-8}$.

\subsubsection{Setting 1: Stronger external damping}

We consider the parameter set $\mathcal{D}=[0.1, 1]^3$ and choose the test parameter set $\mathcal{D}_{\mathrm{Test}}$ to consist of 50 randomly selected parameter triples from $\mathcal{D}$. 
Therefore, we use the $\mathtt{rand}$ operator from \matlab{} so that the $k$-th test parameter is equal to $\mu= 0.1+0.9\cdot\mathtt{rand}(3,1)$ with fixed seed.

\begin{figure}[bt]
	\centering
	\includegraphics[scale=0.35]{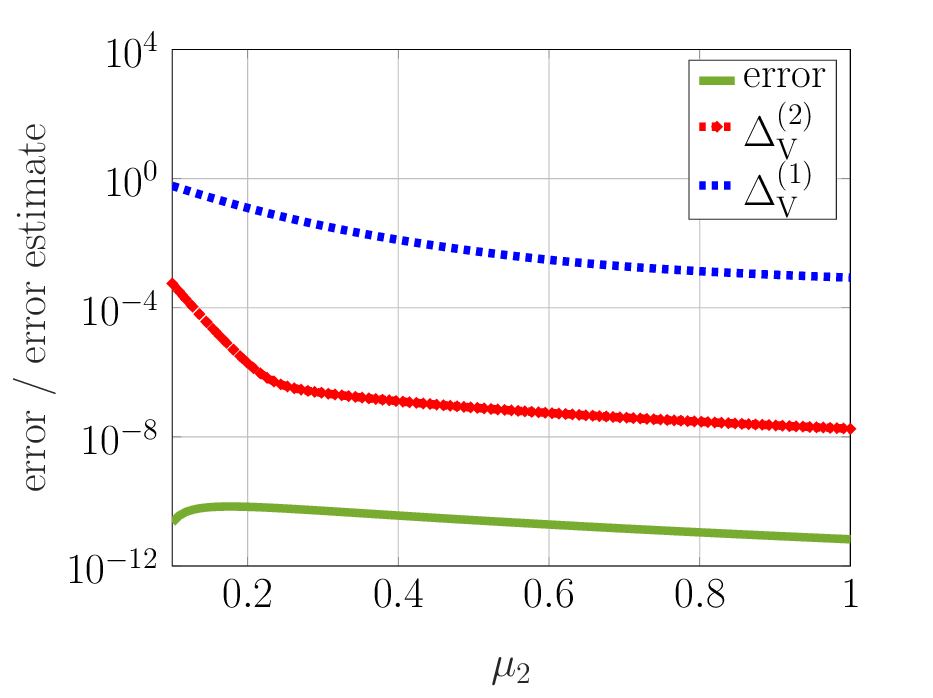}%Index3_RBM_P1
	\caption{Error and error estimates of the approximated controllability Gramians of the mechanical system~\eqref{Index3} with stronger external damping and $\mu_1= 0.25$, $\mu_3=0.35$ after the first iteration of the reduced basis method.}
	\label{fig:RBM_contr}
\end{figure}

\begin{figure}[bt]
	\centering
	\subfloat[Sigma plots of the original, reduced, and error transfer functions 
	for the test parameter $\mu=\begin{bmatrix}
	0.2622,\, 0.1175,\, 0.5169
	\end{bmatrix}$.]{%
		\includegraphics[scale=0.35]{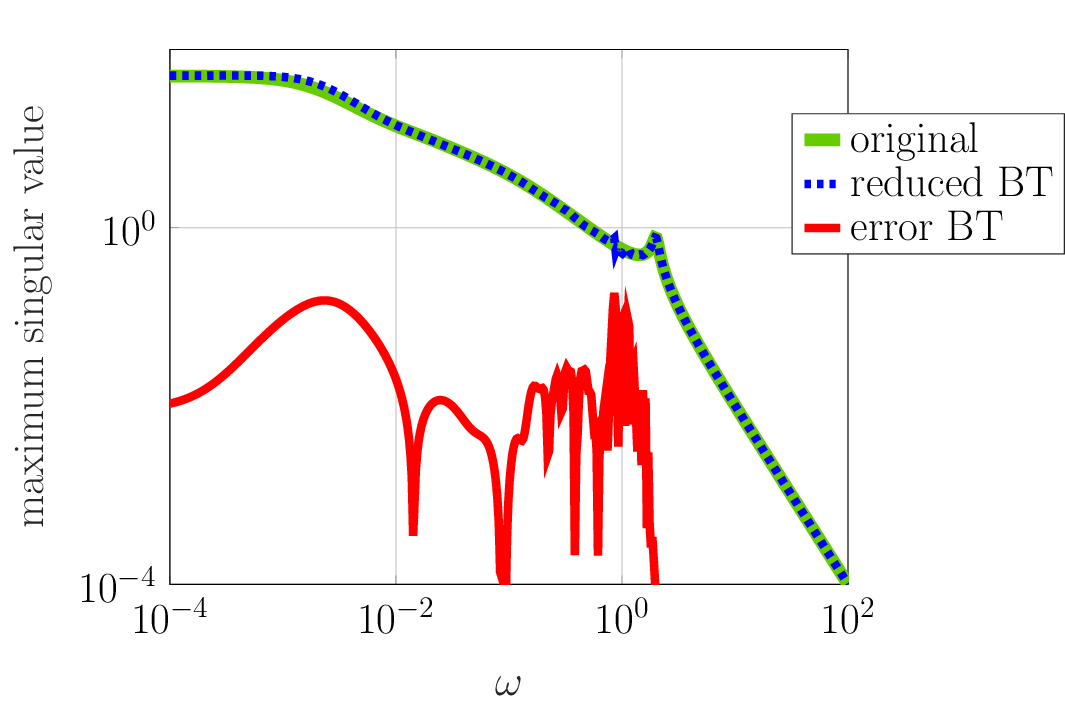}}%Index3BT_smd6_para
	\hspace{20pt}
	\subfloat[Sigma plots of the original, reduced, and error transfer functions 
	for the parameter $\mu=\begin{bmatrix}
	0.15,\, 0.15,\, 0.5
	\end{bmatrix}$.]{%
		\includegraphics[scale=0.35]{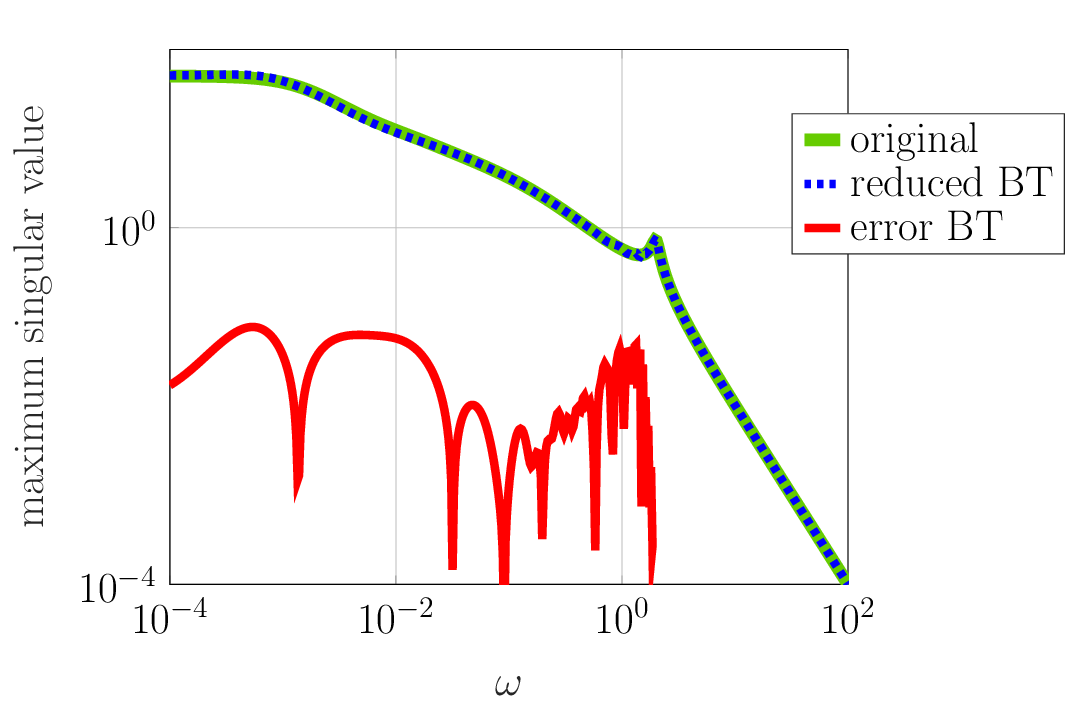}}%Index3BT_smd6_015
	\caption{Results for the reduction of the mechanical system \eqref{Index3} with stronger external damping.}
	\label{BT3pic}
\end{figure}

In the offline phase, the approximation of the controllability and observability spaces takes $1.4\cdot 10^{4}$ seconds for this example.
For the proper controllability space, we evaluate the errors and the error estimates that are depicted in Figure \ref{fig:RBM_contr} for the first iteration of the reduced basis method. For evaluating the error, the Lyapunov equations \eqref{eq:propLE_contr} are also solved by computing the projection matrices $\bPil(\mu)$ and $\bPir(\mu)$ explicitly and solving the corresponding Lyapunov equation.
Therefore, we fix the parameters $\mu_1= 0.25$ and $\mu_3=0.35$ and show the errors for the varying second parameter $\mu_2 \in[0.1,1]$.
After the first iteration and for a basis of dimension $219$, the maximal error estimate $\Delta_{\bV}^{(2)}$ is equal to $9.48\cdot 10^{-5}$.
Hence, the basis approximates the proper controllability space sufficiently well.
Again, we see that the error estimator $\Delta_{\bV}^{(2)}$ approximates the error $\mathfrak{E}(\cdot)$ better than the estimation by $\Delta_\bV^{(1)}$, which massively overestimates the error. 
Hence, we only use the error estimator $\Delta_{\bV}^{(2)}$ to determine the basis quality.

We need two iterations of the reduced basis method to determine the respective observability subspace of dimension of $304$.
Hence, we have to solve Lyapunov equations of dimensions $219$ and $304$ in the online phase to compute the Gramian approximations.
We apply the online phase for the parameters $\mu=\begin{bmatrix}0.2622,\, 0.1175,\, 0.5169
\end{bmatrix}\in\mathcal{D}_{\mathrm{Test}}\subset\mathcal{D}$ (which is rounded to 4 digits here) and $\mu = \begin{bmatrix}0.15,\, 0.15,\, 0.5\end{bmatrix}\in\mathcal{D}$.
%Note that the first parameter is given in rounded form and is chosen explicitly from the parameter set.
We obtain the two respective reduced systems of dimension $75$ and $63$, where all states are differential ones. 
The online phases take $6.4\cdot 10^{-1}$ and $1.3\cdot 10^{-1}$ seconds, respectively, for this example.
\begin{figure}[bt]
	\centering
	\subfloat[Output plots of the original and reduced system and the corresponding error for the test parameter $\mu=\begin{bmatrix}
	0.2622,\, 0.1175,\, 0.5169
	\end{bmatrix}$.]{%
		\includegraphics[scale=0.35]{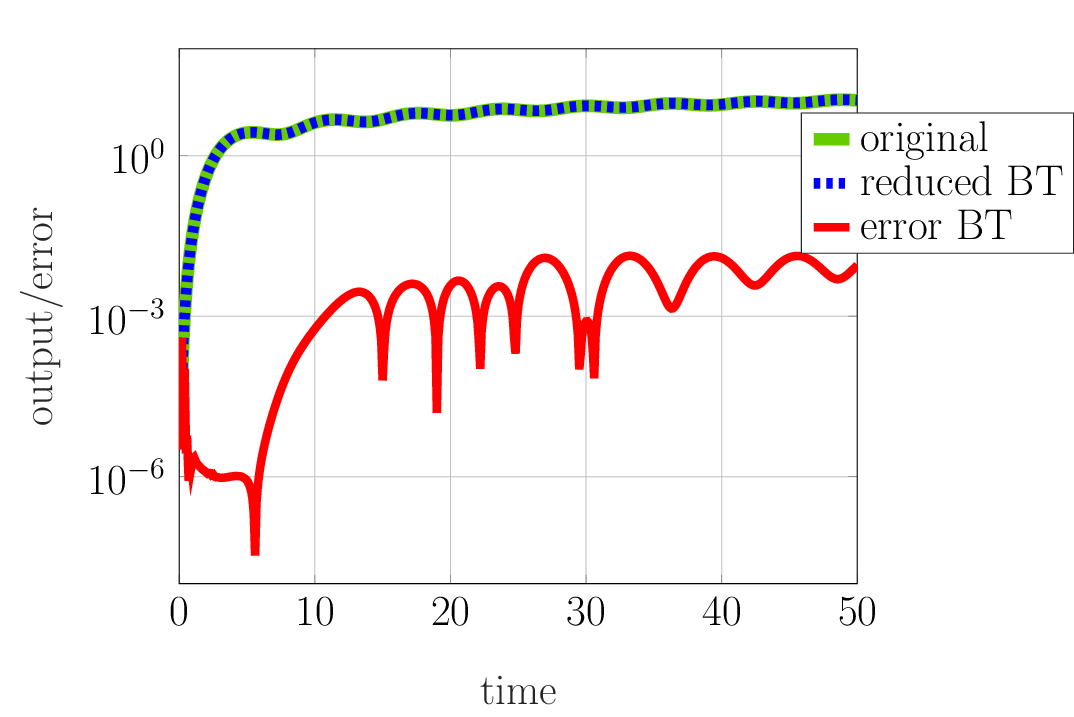}}%Index3output_smd6_para
	\hspace{20pt}
	\subfloat[Output plots of the original and reduced system and the corresponding error for the parameter $\mu=\begin{bmatrix}
	0.15,\, 0.15,\, 0.5
	\end{bmatrix}$.]{%
		\includegraphics[scale=0.35]{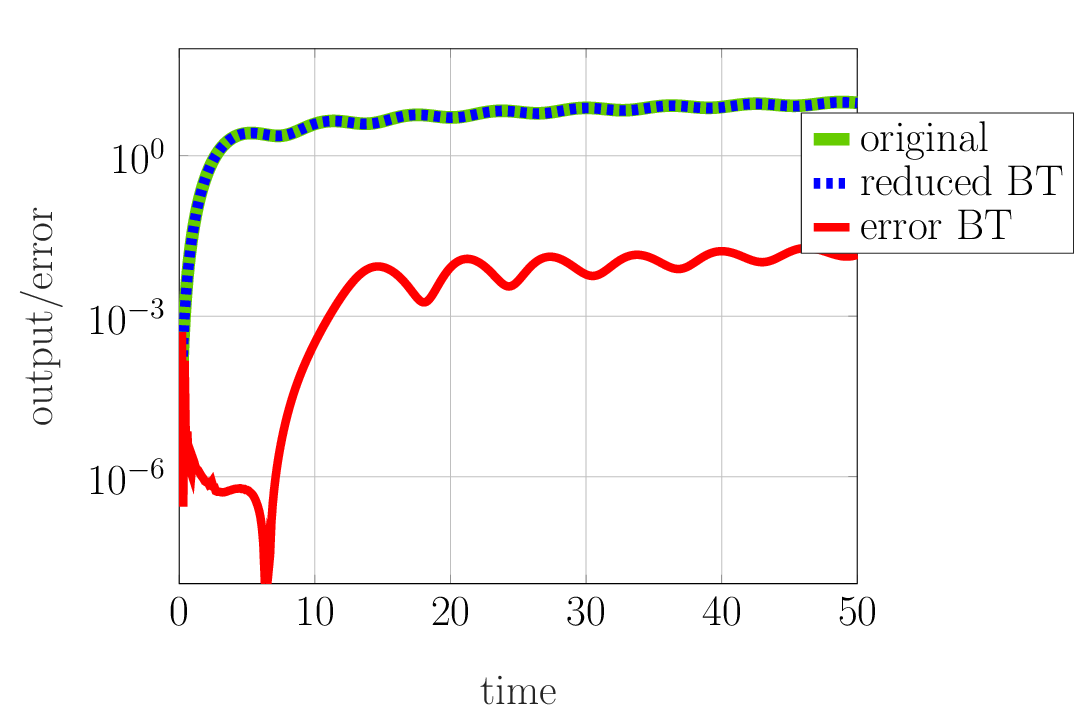}}%Index3output_smd6_015
	\caption{Output and output error of the original and reduced system \eqref{Index3} with stronger external damping.}
	\label{out3pic}
\end{figure}
In Figure~\ref{BT3pic}, we depict the original, the reduced, and the error transfer function.
Additionally, we evaluate the outputs of the original and the reduced system $y$ and $y_{\mathrm{R}}$. The results are shown in Figure~\ref{out3pic}, where we observe the norm of the output error ${\|y(t) - y_{\mathrm{R}}(t)\|}_2$ for $t \in [0,50]$.
We choose the input function $u(t)= 1-1\cdot \cos(t)$ and the initial condition $\bz(0)=\begin{bsmallmatrix}
x(0)\\ \lambda(0)
\end{bsmallmatrix}=0$, so that the consistency conditions are satisfied. 
We observe that the output error is smaller than $10^{-2}$ for all $t\in[0,\, 50]$.
Note that for damping optimization examples, the proper controllability and observability spaces are often hard to approximate by spaces of small dimensions, see \cite{morTomBG18,morBenKTetal16}. Hence, the tolerances are chosen rather large for this kind of example to achieve satisfying results.

\subsubsection{Setting 2: Weaker external damping}

Now, we consider a parameter set with viscosities of small magnitudes, that is $\mathcal{D}=[0.01, 0.1]^3$.
The test-parameter set $\mathcal{D}_{\mathrm{Test}}$ consists of 50 randomly selected parameters from $\mathcal{D}$.
To generate the test parameters, we again use the $\mathtt{rand}$ operator from \matlab{} so that the $k$-th parameter is equal to $\mu = 0.01+0.09\cdot\mathtt{rand}(3,1)$ with fixed seed.

\begin{figure}[bt]
	\centering
	\includegraphics[scale=0.35]{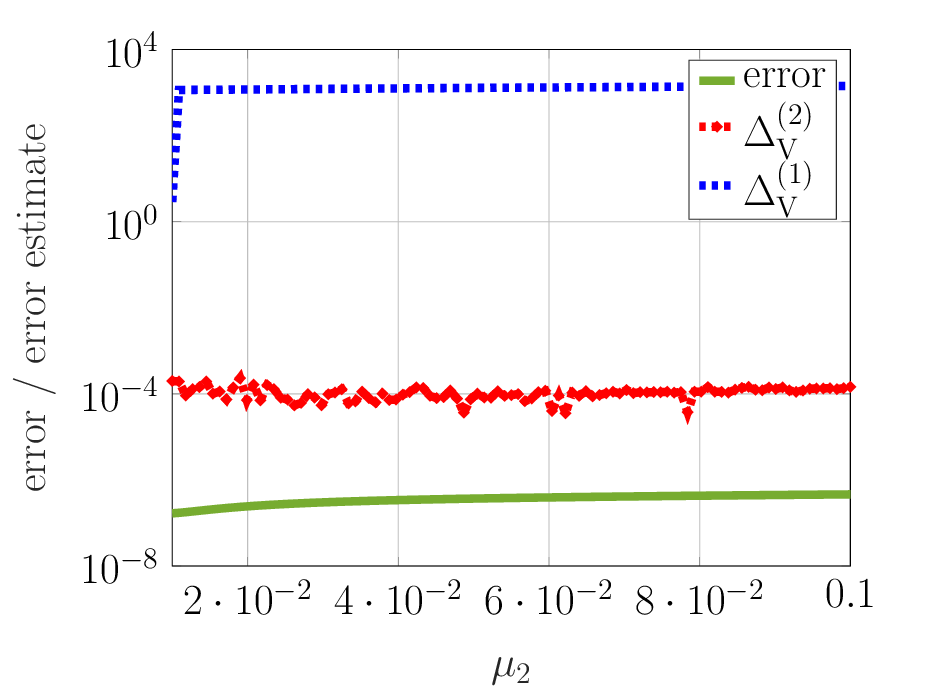}%Index3_RBM_P1_small
	\caption{Error and error estimates of the approximated observability Gramians of the mechanical system~\eqref{Index3} with weaker external damping and $\mu_1= 0.025$, $\mu_3=0.035$ after the first iteration of the reduced basis method.}
	\label{fig:RBM_contr_Dsmall}
\end{figure}

\begin{figure}[bt]
	\centering
	\subfloat[Sigma plots of the original, reduced, and error transfer functions 
	for the test parameter $\mu=\begin{bmatrix}
	0.0886,\, 0.0972,\, 0.0882
	\end{bmatrix}$.]{%
		\includegraphics[scale=0.35]{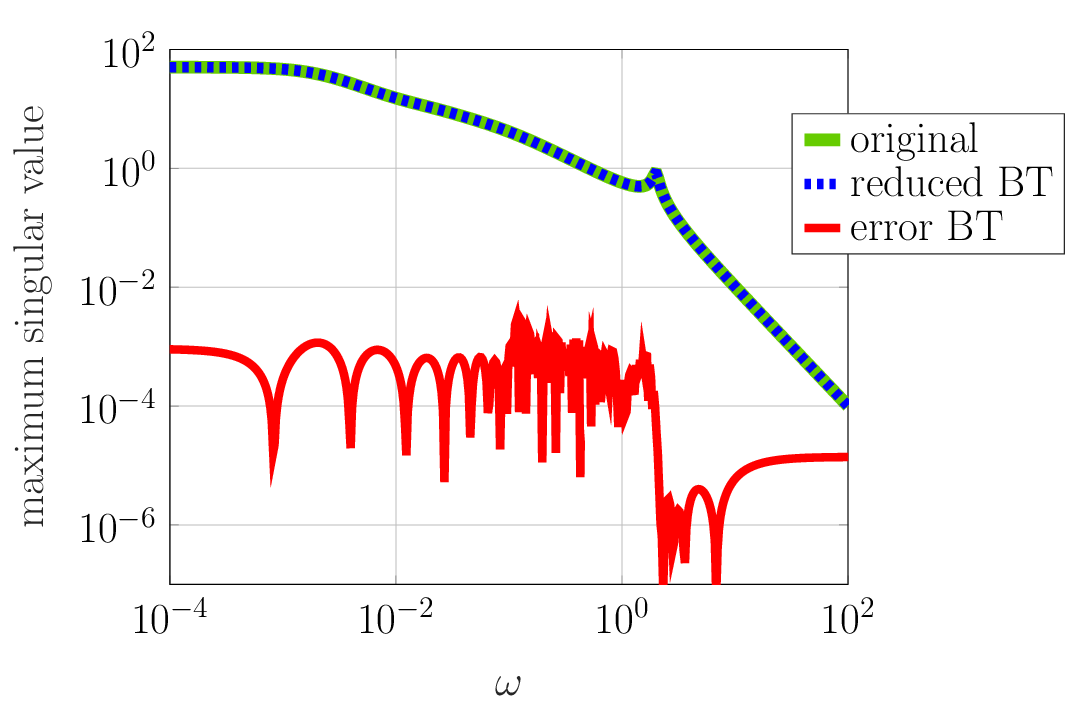}}%Index3BT_smd4_para
	\hspace{20pt}
	\subfloat[Sigma plots of the original, reduced, and error transfer functions 
	for the parameter $\mu=\begin{bmatrix}
	0.01,\, 0.02,\, 0.01
	\end{bmatrix}$. The magnifying glass illustrates that the reduced-order model captures the strong oscillations in the sigma plot very well.]{%
		\includegraphics[scale=0.35]{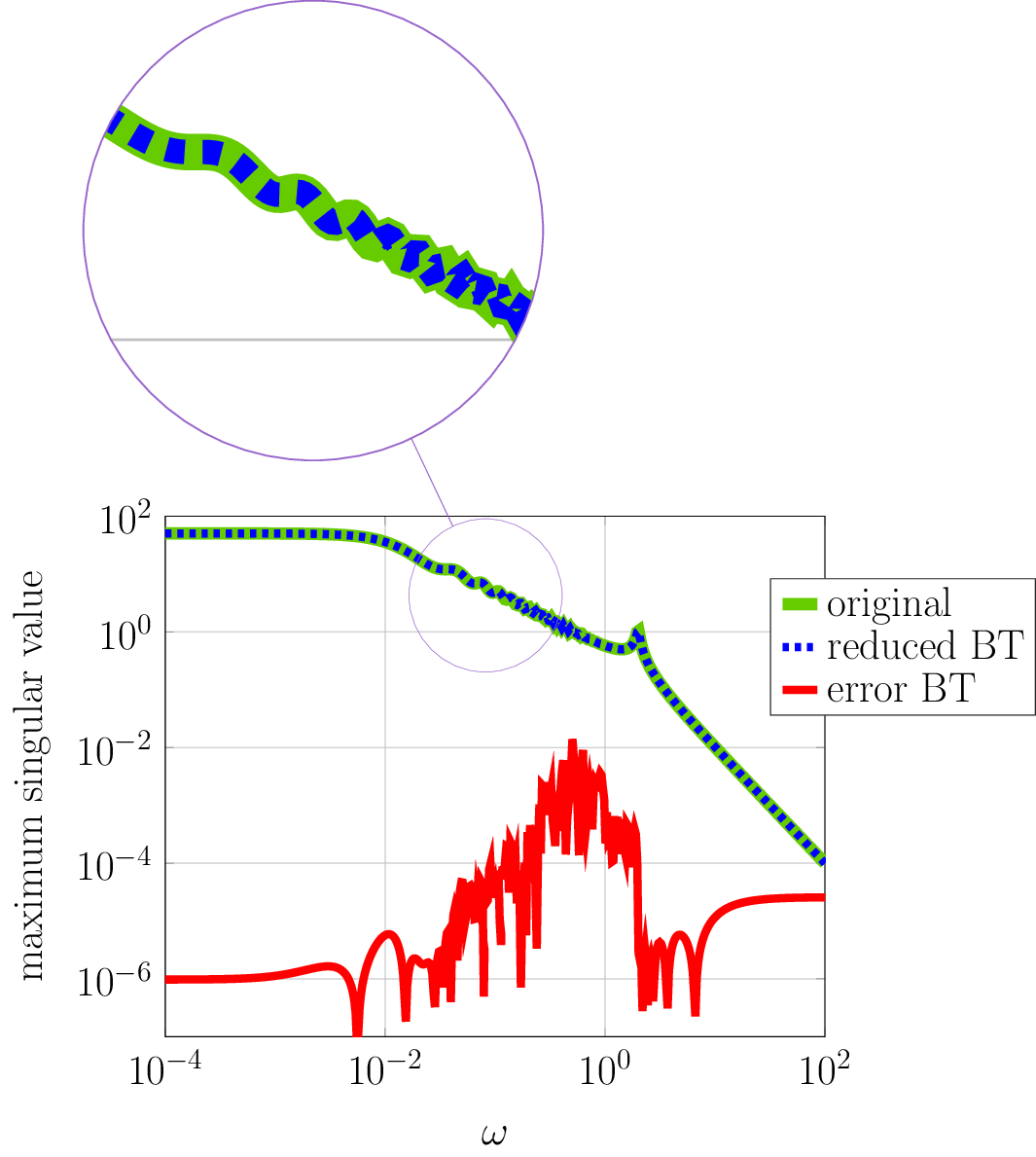}}%Index3BT_smd4_002
	\caption{Results for the reduction of the mechanical system \eqref{Index3} with weaker external damping.}
	\label{BT3pic_Dsmall}
\end{figure}

First, we apply the offline phase of the reduced basis method that takes $6.9\cdot 10^{3}$ seconds.
When computing the proper controllability space, two iterations are needed to generate a subspace of dimension $449$ that leads to error estimates $\Delta_{\bV}^{(2)}(\mu)$ that are at most $5.6 \cdot 10^{-4}$.
For the proper observability space, the errors and the error estimates are illustrated in Figure~\ref{fig:RBM_contr_Dsmall} within the first iteration of the reduced basis method with fixed parameters $\mu_1= 0.025$ and $\mu_3=0.035$ and varying $\mu_2 \in [0.01,0.1]$. For evaluating the error, the Lyapunov equations \eqref{eq:propLE_observ} are also solved by computing the projection matrices $\bPil(\mu)$ and $\bPir(\mu)$ explicitly and solving the corresponding Lyapunov equation.
After the first iteration, we determine a basis of dimension $223$ for which the maximal error estimate $\Delta_{\bV}^{(2)}$ is equal to $5.8\cdot 10^{-4} $, and hence, sufficient accuracy is achieved. 
As for the first parameter setting, the error estimator $\Delta_{\bV}^{(2)}$ approximates the error $\mathfrak{E}(\cdot)$ better than the estimation by $\Delta_\bV^{(1)}$.
Hence, we use the error estimator $\Delta_{\bV}^{(2)}$ within the reduced basis method.
We apply the online phase for the parameters $\mu=\begin{bmatrix}
0.0886,\, 0.0972,\, 0.0882
\end{bmatrix}\in\mathcal{D}_{\mathrm{Test}}\subset\mathcal{D}$ (which is rounded to 4 digits) and $\mu = \begin{bmatrix}0.01,\, 0.02,\, 0.01\end{bmatrix}\in\mathcal{D}$ to derive a reduced-order model of dimension $78$ for the first parameter and the reduced dimension $178$ for the second one. 
%Note that the first parameter is given in rounded form and is chosen explicitly from the parameter set.
The online phases need $2.21$ and $2.04$ seconds, respectively.
\begin{figure}[bt]
	\centering
	\subfloat[Output plots of the original and reduced system and the corresponding error for the test parameter $\mu=\begin{bmatrix}
	0.089,\, 0.097,\, 0.088
	\end{bmatrix}$.]{%
		\includegraphics[scale=0.35]{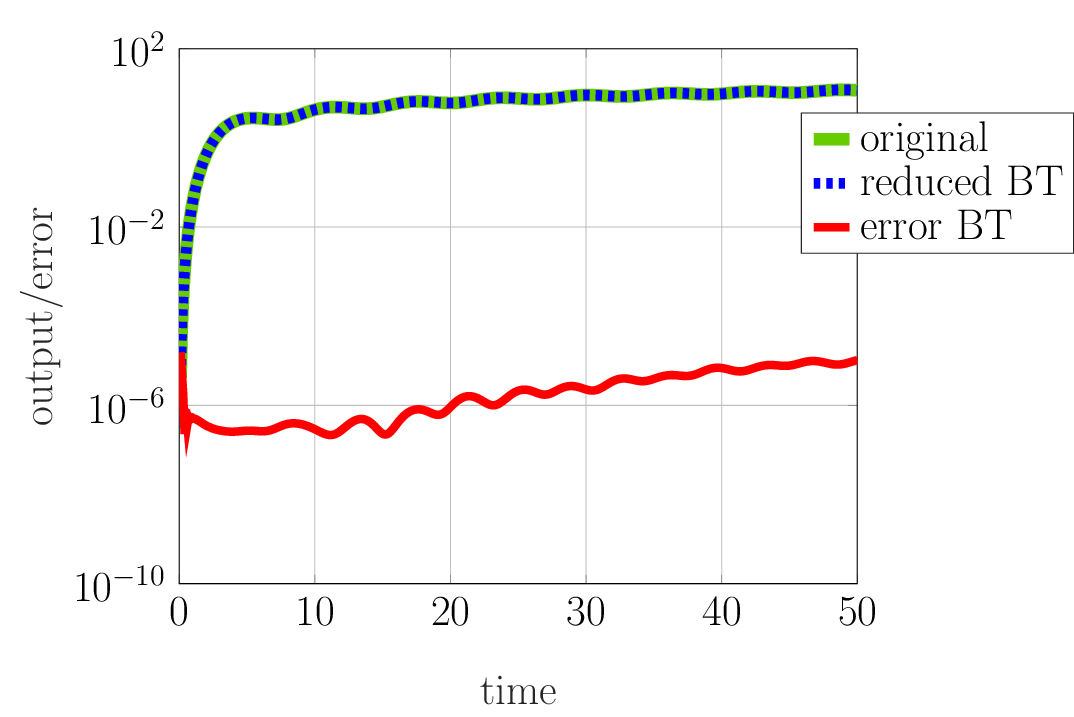}}%Index3output_smd4_para
	\hspace{20pt}
	\subfloat[Output plots of the original and reduced system and the corresponding error for the parameter $\mu=\begin{bmatrix}0.01,\, 0.02,\, 0.01\end{bmatrix}$.]{%
		\includegraphics[scale=0.35]{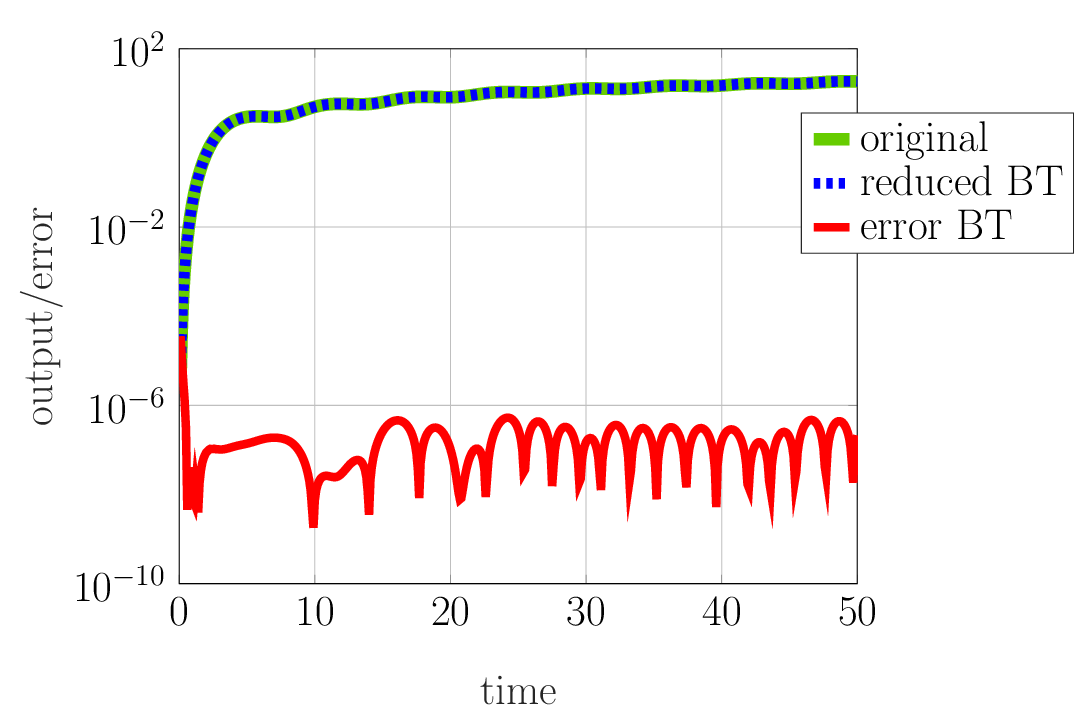}}%Index3output_smd4_002
	\caption{Output and output error of the original and reduced mechanical system \eqref{Index3} with weaker external damping.}
	\label{out3pic_Dsmall}
\end{figure}
Figure \ref{BT3pic_Dsmall} depicts the original, the reduced, and the error transfer function.
Additionally, we evaluate the norm of the output error ${\|y(t) - y_{\mathrm{R}}(t)\|}_2$ for $t \in [0,50]$ in Figure \ref{out3pic_Dsmall}, where the input function $u(t)= 1-1\cdot \mathrm{cos}(t)$ and the initial condition $\bz(0)=\begin{bsmallmatrix}
x(0)\\ \lambda(0)
\end{bsmallmatrix}=0$ are chosen in such a way that the consistency conditions are satisfied. 
We observe that the output error is smaller than $10^{-4}$ for all $t\in[0,\, 50]$.

%Due to the small viscosity values in this example, in particular the proper controllability space is hard to approximate with bases of small dimensions and hence, the original model cannot be reduced well by a low-order surrogate model. Such a behavior is already well-documented in the literature on model reduction of weakly damped mechanical systems without parameters (see, e.\,g., \cite{morSaaV18}) and this poses also a limitation of our method.

\section{Conclusions}
This paper has addressed the model reduction of parameter-dependent differential-algebraic systems by balanced truncation.
To apply the balancing procedure we have utilized specific projectors to eliminate the algebraic constraints. This has enabled us to compute the necessary Gramians efficiently by solving projected Lyapunov equations.

To handle the parameter-dependency, we have applied the reduced basis method, which is split into the \emph{offline phase} and the \emph{online phase}. 
In the \emph{offline phase}, we have computed the basis of the subspace which approximates the solution space of the parametric Lyapunov equations.
To assess the quality of this subspace we have derived two error estimators.
Afterwards, in the \emph{online phase}, we have solved a reduced Lyapunov equation to obtain an approximation of the Gramian for a parameter value of interest efficiently. Therefore, a balanced truncation for this parameter value can also be carried out fast.

This method has been illustrated by numerical examples of index two and three. In particular, we were often able to reduce the associated Lyapunov equations to very small dimensions. We have evaluated our error estimators, where the second one can often estimate the error almost exactly. 

\section*{Code availability} The code and data that has been used to generate the numerical results of this work are freely available under the DOI
10.5281/zenodo.10040792
%10.5281/zenodo.7468619
%10.5281/zenodo.5145752 
under the 3-clause BSD license and is authored by Jennifer Przybilla.

\section*{CRediT author statement} \textbf{Jennifer Przybilla:} Methodology, Software, Data Curation, Writing --- Original Draft, Writing --- Review \& Editing, Visualization; \textbf{Matthias Voigt:}  Conceptualization, Writing --- Original Draft, Writing --- Review \& Editing, Supervision.

\bibliographystyle{plainurl}
\bibliography{References,csc,mor,software}

\end{document}